\documentclass[10pt]{amsart}
 
\usepackage{amsthm,amsmath,amssymb}
\usepackage{graphicx}
\usepackage[comma, sort&compress]{natbib}
\usepackage{tabularx}
\usepackage{multirow}
\usepackage{amssymb}
\usepackage{tikz}
\usepackage{xcolor}
\usetikzlibrary{shapes,arrows,positioning}

\usepackage{paralist}
\RequirePackage{hyperref}

\newtheorem{theorem}{Theorem}

\newtheorem{lemma}[theorem]{Lemma}

\newtheorem{property}[theorem]{Property}

\theoremstyle{definition} 

\theoremstyle{remark}

\numberwithin{equation}{section}
\numberwithin{theorem}{section}
\numberwithin{example}{section}
\numberwithin{definition}{section}

\numberwithin{figure}{section}

\DeclareMathOperator{\bE}{\mathbb{E}}
\DeclareMathOperator{\bR}{\mathbb{R}}
\DeclareMathOperator{\cX}{\mathcal{X}}
\DeclareMathOperator{\cE}{\mathcal{E}}
\DeclareMathOperator{\barPhi}{\bar{\Phi}}
\DeclareMathOperator{\barDelta}{\bar{\Delta}}

\newcommand{\secref}[1]{Section~\ref{sec:#1}}
\newcommand{\appref}[1]{Appendix~\ref{app:#1}}
\newcommand{\appsref}[1]{Appendices~\ref{app:#1}}
\newcommand{\appssref}[1]{\ref{app:#1}}

\newcommand{\lemref}[1]{Lemma~\ref{lem:#1}}
\newcommand{\lemsref}[1]{Lemmas~\ref{lem:#1}}
\newcommand{\lemssref}[1]{\ref{lem:#1}}

\newcommand{\thmref}[1]{Theorem~\ref{thm:#1}}

\newcommand{\propertyref}[1]{Property~\ref{property:#1}}

\title[]{Another look at Stein's method for Studentized nonlinear statistics with an application to U-statistics}

\author[D.~Leung]{Dennis Leung} 
\address{School of Mathematics and Statistics, University of Melbourne}
\email{dennis.leung@unimelb.edu.au}

\author[Q.~Shao]{Qi-Man Shao} 
\address{Department of Statistics and Data Science, SICM, National Center for Applied Mathematics Shenzhen, Southern University of Science and Technology}
\email{shaoqm@sustech.edu.cn}

\author[L.~Zhang]{Liqian Zhang} 
\address{School of Mathematics and Statistics, University of Melbourne}
\email{liqian@student.unimelb.edu.au}

\begin{document}

\begin{abstract}
We take another  look at using Stein's method  to establish uniform Berry-Esseen bounds for Studentized nonlinear statistics, highlighting variable censoring and  an exponential randomized concentration  inequality for a sum of censored variables as the essential  tools to carry the arguments involved.  As an important application, we prove a uniform Berry-Esseen bound for Studentized U-statistics in a form that  exhibits the dependence on the degree of the kernel.
\end{abstract}



\maketitle

\section{Introduction} \label{sec:intro}

We revisit the use of Stein's method to prove uniform Berry-Esseen (B-E) bounds for Studentized nonlinear statistics. Let $X_1, \dots, X_n$ be independent random variables that serve as some raw data, and for each $i =1, \dots, n$, let
\begin{equation} \label{xi_i_def}
\xi_i \equiv g_{n, i} (X_i)  
\end{equation}
for a function $g_{n, i}(\cdot)$ that can also depend on $i$ and $n$, such that
\begin{equation} \label{assumptions}
\bE[\xi_i] = 0 \text{ for all }i \text{ and }\sum_{i=1}^n\bE[\xi_i^2] = 1. 
\end{equation}
A  \emph{Studentized nonlinear statistic} is an asymptotically normal statistic  that can be represented in the general form
\begin{equation} \label{sn_stat}
T_{SN} \equiv \frac{W_n + D_{1n}}{(1+ D_{2n})^{1/2}},
\end{equation}
with $W_n \equiv \sum_{i=1}^n \xi_i$, where the ``remainder" terms
\begin{equation}\label{D1D2_as_fn_of_data}
D_{1n} = D_{1n} (X_1, \dots, X_n) \text{ and } D_{2n} = D_{2n} (X_1, \dots, X_n)
\end{equation}
are  some functions of the data, with the additional properties that
\begin{equation} \label{Dj_generic_properties}
D_{1n}, D_{2n} \longrightarrow 0 \text{ in probability as $n$ tends to $\infty$}, \text{ and } D_{2n} \geq -1 \text{ almost surely}.
\end{equation}
We adopt the convention that if $1 + D_{2n} = 0$, the value of $T_{SN}$ is taken to be $0$, $+ \infty$ or $-\infty$ depending on the sign of $W_n + D_{1n}$. 
Such a statistic  is a generalization of the classical Student's t-statistic \citep{student1908probable}, where the denominator $1 + D_{2n}$ acts as a data-driven ``self-normalizer" for the numerator $W_n + D_{1n}$.

Many statistics used in practice can  be seen as  examples of \eqref{sn_stat}, hence developing a general Berry-Esseen-type inequality for $T_{SN}$  is  relevant to many  applications. The first such attempt based on Stein's method can be found in the semi-review article of \citet{shao2016stein}, whose proof critically relies upon an exponential-type randomized concentration inequality  first appearing in \citet{shao2010stein}. However, while their methodology is sound, there are numerous gaps; most notably,  \citet{shao2016stein} overlooked that the original  exponential-type randomized concentration inequality of \citet{shao2010stein} is developed for a sum of independent random variables with mean zero, which is not well-suited for their  proof  wherein the truncated summands generally do not have mean 0. In fact, truncation itself is an insufficient device to carry  the arguments involved, as will be explained in this article. 

Our contributions are twofold. First,  we put the methodology of \citet{shao2016stein} on solid footing; this, among other things, is accomplished by adopting variable \emph{censoring} instead of truncation, as well as developing a modified  randomized concentration inequality for a sum of censored variables, to rectify the gaps in their arguments. We also present a more user-friendly B-E bound for the statistic $T_{SN}$ when the denominator remainder $D_{2n}$ admits a certain standard form. Second, as an application to a prototypical example of Studentized nonlinear statistics, we establish a uniform B-E bound of the rate $1/\sqrt{n}$ for Studentized U-statistics whose dependence on the  degree of the kernel is also explicit; all prior  works in this vein only treat the simplest case with a kernel of degree 2. 
This bound is the most optimal known to date,  and serves to complete the literature in uniform B-E bounds for Studentized U-statistics.  

{\bf \emph{Notation}}.   $\Phi(\cdot)$ is the standard normal distribution function  and $\bar{\Phi}(\cdot) = 1- \Phi(\cdot)$. The indicator function is denoted by $I(\cdot)$. For $p \geq 1$, $\|Y\|_p \equiv (\bE[|Y|^p])^{1/p}$ for a random variable $Y$. For any $a, b \in \bR$, $a \vee b = \max(a, b)$ and $a \wedge b = \min(a, b)$. $C, C_1, C_2 \cdots..$ denotes positive \emph{absolute} constants that may differ in value from place to place, but does not depend on other quantities nor the distributions of the random variables. For two (possibly multivariate) random variables $Y_1$ and $Y_2$, ``$Y_1 =_d Y_2$" means $Y_1$ and $Y_2$ have the same distribution. 

\section{General Berry-Esseen bounds for Studentized nonlinear statistics}

Let $\xi_1, \dots, \xi_m$  be as in \secref{intro} that satisfy the assumptions in \eqref{assumptions}. 
For each $i = 1, \dots, n$, define 
\begin{equation}\label{lower_and_upper_censored_xi_i}
\xi_{b, i} \equiv \xi_i I(|\xi_i| \leq 1) +  I(\xi_i >1) -   I(\xi_i <-1),
\end{equation}
an upper-and-lower \emph{censored} version of $\xi_i$, and their  sum 
\begin{equation} \label{W_b_def}
W_b = W_{b, n}\equiv \sum_{i=1}^n \xi_{b, i}.
\end{equation}
 Moreover, for each $i = 1, \dots, n$, we define $W_b^{(i)} \equiv W_b - \xi_{b, i}$ and $ W_n^{(i)} \equiv W_n - \xi_i$. 
We also let 
\[
\beta_2 \equiv \sum_{i=1}^n \bE[\xi_i^2 I(|\xi_i| > 1)] \text{ and } \beta_3 \equiv \sum_{i=1}^n \bE[|\xi_i|^3 I(|\xi_i| \leq 1)].
\]
For any $x \in \bR$,  
\begin{equation} \label{fx}
f_x (w) \equiv
  \begin{cases} 
  \sqrt{2 \pi} e^{w^2/2} \Phi(w)\bar{\Phi}(x) &  w \leq x\\
 \sqrt{2 \pi} e^{w^2/2} \Phi(x) \bar{\Phi}(w)        &  w > x
     \end{cases};
\end{equation}
is the solution to the Stein equation \citep{stein1972bound}
\begin{equation} \label{steineqt}
f_x'(w) - wf_x(w) = I(w \leq x) - \Phi(x).
\end{equation}
Our first result is the following  uniform Berry-Esseen bound  for the Studentized nonlinear statistic in \eqref{sn_stat}: 

\begin{theorem}[Uniform B-E bound for  Studentized nonlinear  statistics] \label{thm:main}
Let $X_1, \dots, X_n$ be independent random variables. 
Consider the Studentized nonlinear statistic $T_{SN}$ in \eqref{sn_stat}, constructed with the linear summands  in  \eqref{xi_i_def} that satisfy the condition in \eqref{assumptions}, and the remainder terms in \eqref{D1D2_as_fn_of_data} that satisfy the condition in \eqref{Dj_generic_properties}. There exists a positive absolute constant $C > 0$ such that
\begin{multline} \label{generic_overarching_BE_bdd}
\sup_{x \in \mathbb{R}} \Big|P(T_{SN} \leq x) - \Phi(x)\Big| \leq  \sum_{j=1}^2P(|D_{jn}| > 1/2)  \\
+ C \Bigg\{\beta_2  + \beta_3  +  \|\bar{D}_{1n}\|_2 +  \bE\Big[(1 +e^{W_b}) \bar{D}_{2n}^2\Big] + 
\sup_{x \geq 0} \Big|x \bE[\bar{D}_{2n} f_x(W_b)]\Big| \\
+  \sum_{j=1}^2
\sum_{i=1}^n \Big( \bE[\xi^2_{b, i}] \Big\| (1+ e^{W_b^{(i)}})(\bar{D}_{jn} - \bar{D}_{jn}^{(i)} )\Big\|_1 +  \Big\|  \xi_{b, i} ( 1+e^{W_b^{(i)}/2} ) ( \bar{D}_{jn} - \bar{D}_{jn}^{(i)})   \Big\|_1\Big)
  \Bigg\}   ,
\end{multline}
where for each $j \in \{1, 2\}$ and each $i \in \{1, \dots, n\}$, 
\begin{itemize}
\item $D_{jn}^{(i)} \equiv D_{jn}^{(i)}(X_1, \dots, X_{i-1}, X_{i+1}, \dots, X_n)$ is any function in  the raw data except $X_i$;
\item $\bar{D}_{jn}$ is a censored version of  $D_{jn}$ defined as
\[
\bar{D}_{jn}\equiv  D_{jn}I\Big(|D_{jn}| \leq \frac{1}{2} \Big) + \frac{1}{2}I\Big(D_{jn} > \frac{1}{2}\Big) - \frac{1}{2} I\Big(D_{jn} <- \frac{1}{2}\Big);
\]
\item $\bar{D}^{(i)}_{jn}$ is a   censored version of  $D^{(i)}_{jn}$ defined as
\[
\bar{D}_{jn}^{(i)} \equiv  D_{jn}^{(i)}I\Big(|D_{jn}^{(i)}| \leq \frac{1}{2}\Big) +\frac{1}{2} I\Big(D_{jn}^{(i)} > \frac{1}{2}\Big) - \frac{1}{2}I\Big(D_{jn}^{(i)} <- \frac{1}{2}\Big).
\]

\end{itemize}
\end{theorem} 
In applications, $D_{1n}^{(i)}$ and $D_{2n}^{(i)}$ are typically  taken as  ``leave-one-out" quantities  constructed in almost identical manner as $D_{1n}$ and $D_{2n}$ respectively, but without any terms involving the datum $X_i$; for instance, compared $D_{1n}$ and $D_{1n}^{(i)}$ in \eqref{D1_ustat}  and \eqref{D1ni_def} below for the case of a U-statistic. 
The proof of \thmref{main} (\appref{main_pf}) bypasses the gaps in the proof of the original B-E bound for $T_{SN}$ stated in \citep[Theorem 3.1]{shao2016stein}.  As a key step in their approach to proving \citet[Theorem 3.1]{shao2016stein} based on Stein's method, the exponential-type randomized concentration inequality developed in \citet[Theorem 2.7]{shao2010stein} is applied to control a probability of the type
\[
P\left(\Delta_1 \leq \sum_{i=1}^n \xi_i I( |\xi_i| \leq 1 ) \leq \Delta_2\right), 
\]
where $\Delta_1$ and $\Delta_2$ are some context-dependent random quantities. Unfortunately, \citet{shao2016stein} overlooked that \citet[Theorem 2.7]{shao2010stein}  was originally developed for a sum of mean-0 random variables, such as $W_n$, instead of the  sum $\sum_{i=1}^n \xi_i I( |\xi_i| \leq 1 )$ figuring in the prior display, whose truncated summands do not have mean 0 in general. 
 The latter  needs to be  addressed in some way to mend their arguments, which leads to the  exponential randomized concentration inequality (\lemref{modified_RCI_bdd}) developed in this work for the sum $W_b$ in \eqref{W_b_def}. Here, the censored summands $\xi_{b, i}$'s are considered instead so that the  new inequality can still be proved in much the same way as \citet[Theorem 2.7]{shao2010stein}; replacing the truncated $\xi_i I( |\xi_i| \leq 1 )$ with the censored $\xi_{b, i}$ is otherwise permissible, because only the boundedness of the summands is essential under  the approach. 

The B-E bound stated in \thmref{main} is in  a primitive form. When applied to specific examples of $T_{SN}$, the various terms in \eqref{generic_overarching_BE_bdd}  have to be further estimated to render a more expressive bound. In that respect, the following apparent  properties of censoring will become very useful:

 \begin{property}[Properties of variable censoring] \label{property:censoring_property}
Let $Y$ and $Z$ be any two real-value variables. The following facts hold:
 
 \begin{enumerate}
 \item  Suppose, for some  $a, b \in \bR \cup \{-\infty, \infty\}$ with $a \leq b$, 
\[
\bar{Y} \equiv a I(Y < a) + Y I(a \leq Y \leq b)+bI(Y>b)
\] and 
\[
\bar{Z}  \equiv a I(Z < a) + Z I(a \leq Z \leq b)+bI(Z>b).
\]
Then it must be that 
$
|\bar{Y} - \bar{Z} | \leq |Y- Z|.
$
\item
If $Y$ is a non-negative random variable, then it must also be true that 
\[
Y I(0 \leq Y \leq b)+bI(Y>b) \leq Y \text{ for any } b \in (0, \infty), 
\]
i.e., the upper-censored version of $Y$ is always no larger than $Y$ itself.
 \end{enumerate}
\end{property} 
  In applications of \thmref{main}, that $\bar{D}_{1n}$ and $\bar{D}_{1n}^{(i)}$ are lower-and-upper censored  by the same interval $[-1/2, 1/2]$ implies the bound
 \begin{equation} \label{easy_bdd}
|\bar{D}_{1n} - \bar{D}_{1n}^{(i)}| \leq |{D}_{1n} - {D}_{1n}^{(i)}|, 
 \end{equation}
by virtue of \propertyref{censoring_property}$(i)$, as well as  
\begin{equation} \label{easy_easy_bdd}
|\bar{D}_1| \leq |D_1|
\end{equation}
by virtue of \propertyref{censoring_property}$(ii)$ because $|\bar{D}_1|$ is essentially the non-negative  $|D_1|$ upper-censored at $1/2$. 
These bounds imply one can form the further  norm estimates 
 \begin{equation} \label{D1_first_error_bdd}
 \| (1+ e^{W_b^{(i)}})(\bar{D}_{1n} - \bar{D}_{1n}^{(i)} )\|_1  \leq C \| D_{1n} - D_{1n}^{(i)}\|_2,
 \end{equation}
  \begin{equation} \label{D1_second_error_bdd}
\|  \xi_{b, i}( 1+ e^{W_b^{(i)}/2} ) ( \bar{D}_{1n} - \bar{D}_{1n}^{(i)})   \|_1 \leq  C \| \xi_i \|_2  \| D_{1n} - D_{1n}^{(i)}\|_2
 \end{equation}
 and 
 \begin{equation} \label{D1_third_error_bdd}
 \|\bar{D}_1\|_2 \leq   \| D_1\|_2, 
 \end{equation}
for  the terms in \eqref{generic_overarching_BE_bdd} related to the numerator remainder $D_1$; see \appref{main2_pf} for the simple arguments leading to these bounds. The right hand sides of \eqref{D1_first_error_bdd}-\eqref{D1_third_error_bdd} are then amenable to direct second moment calculations to render more expressive terms. We also remark that if, instead, the truncated remainder terms 
\begin{equation} \label{truncated_terms}
 D_{jn}I\Big(|D_{jn}| \leq \frac{1}{2} \Big) \text{ and }  D_{jn}^{(i)}I\Big(|D_{jn}^{(i)}| \leq \frac{1}{2} \Big), \text{ for } j =1, 2,
\end{equation}
are adopted  as in \citet[Theorem 3.1]{shao2016stein}, a bound analogous to \eqref{easy_bdd} does not hold in general; this also attests to censoring as a useful tool for developing nice B-E bounds under the current approach.

In comparison to the terms related to $D_1$,  some of the terms related to $D_2$ in \eqref{generic_overarching_BE_bdd}, such as 
\[
\sup_{x \geq 0}|x \bE[\bar{D}_{2n} f_x(W_b)]| \text{ and }\bE[e^{W_b}\bar{D}_{2n}^2]  ,
 \]
 are more obscure and have to be estimated on a case-by-case basis for specific examples of $T_{SN}$. However, in certain applications, the denominator remainder can  be perceivably manipulated into the  form
\begin{equation} \label{D2_specific_form}
D_{2n} = \max\Big(-1,  \quad \Pi_1+ \Pi_2\Big)
\end{equation}
lower censored  at $-1$, 
where $\Pi_1$ is defined as
\begin{equation}\label{Pi_1_def}
\Pi_1  \equiv \sum_{i =1}^n \Big( \xi_{b, i}^2 -\bE[\xi_{b,i}^2] \Big),
\end{equation}
and $\Pi_2 \equiv  \Pi_2 (X_1, \dots, X_n)$ is another data-dependent term. For instance, if a non-negative self-normalizer $1 + D_{2n}$ can be written as the intuitive form
\[
1 + D_{2n} = \sum_{i=1}^n \xi_i^2 + E
\]
for a data-dependent term  $E \equiv E(X_1, \dots, X_n)$ of perceivably smaller order, then $D_{2n}$ can be cast into the form \eqref{D2_specific_form} because $\sum_{i=1}^n(\bE[\xi_{b,i}^2]  + \bE[ (\xi_i^2 - 1)  I(|\xi_i|> 1)] )=\sum_{i=1}^n \bE[\xi_i^2]=  1$ and
one can take 
\[
\Pi_2 = E - \sum_{i=1}^n\bE[ (\xi_i^2 - 1)  I(|\xi_i|> 1)] + \sum_{i=1}^n (\xi_i^2 - 1)  I(|\xi_i|> 1).
\] 
We now present a more refined version of \thmref{main}  for Studentized nonlinear statistics whose $D_{2n}$ admits the form \eqref{D2_specific_form} under an absolute third-moment assumption on $\xi_i$; the proof  is included in \appref{main2_pf}. 
\begin{theorem}[Uniform B-E bound for  Studentized nonlinear  statistics with the denominator remainder \eqref{D2_specific_form} under a third moment assumption]\label{thm:main2}
Suppose all the conditions in \thmref{main} are met, and that $\bE[|\xi_i|^3] < \infty$ for all $1 \leq i \leq n$. In addition, 
 assume $D_{2n}$ takes the specific form  \eqref{D2_specific_form} with $\Pi_1$ defined in \eqref{Pi_1_def} and $\Pi_2 \equiv \Pi_2(X_1, \dots, X_n)$ being a function in the raw data $X_1, \dots, X_n$. For each $i = 1, \dots, n$, let
 \[
 \Pi_2^{(i)}  \equiv  \Pi_2^{(i)}(X_1, \dots, X_{i-1}, X_{i+1}, \dots, X_n)
 \]
 be any function in the raw data except $X_i$. Then
\begin{multline}\label{user_friendly_BE_bdd}
\sup_{x \in \mathbb{R}} \Big|P(T_{SN} \leq x) - \Phi(x)\Big| \leq  C \Bigg\{ \sum_{i=1}^n \bE[|\xi_i|^3]  + \|D_{1n}\|_2 +  \|\Pi_2\|_2 + \\
   \sum_{i=1}^n \|\xi_i\|_2 \|D_{1n}- D_{1n}^{(i)}\|_2 +
 \sum_{i=1}^n \|\xi_i\|_2\|\Pi_2 - \Pi_2^{(i)}\|_2
 \Bigg\},
\end{multline}
where $D_{1n}^{(i)} \equiv D_{1n}^{(i)}(X_1, \dots, X_{i-1}, X_{i+1}, \dots, X_n)$ is as in \thmref{main}.
\end{theorem}
The $\|\cdot\|_2$ terms in \eqref{user_friendly_BE_bdd} are now amenable to direct second moment calculations. 
Hence, if one can cast the denominator remainder into the form \eqref{D2_specific_form}, \thmref{main2} provides a user-friendly framework to establish B-E bounds for such instances of $T_{SN}$.

\section{Uniform Berry-Esseen bound for Studentized U-statistics} \label{sec:u_stat_sec}

We will apply \thmref{main2} to establish a uniform B-E bound of the rate $1/\sqrt{n}$ for Studentized U-statistics of any degree; all prior works in this vein  \citep{jing2000berry, callaert1981order, helmers1985berry, zhao1983rate, shao2016stein} only offer bounds for Studentized U-statistics  of degree 2.  We refer the reader  to \citet{shao2016stein} and \citet{jing2000berry} for other examples of applications, including L-statistics  and random sums and functions of nonlinear statistics.

Given independent and identically distributed random variables $X_1, \dots, X_n$ taking value in a measure space $(\cX, \Sigma_{\mathcal{X}})$, a U-statistic of degree $m \in \mathbb{N}_{\geq 1}$ takes the form 
\[
U_n = {n \choose m}^{-1} \sum_{1 \leq i_1< \dots < i_m \leq n} h(X_{i_1} , \dots, X_{i_m}),
\]
where $h : \mathcal{X}^m \longrightarrow \bR$ is a real-valued function symmetric in its $m$ arguments, also known as the kernel of $U_n$; throughout, we will assume that 
\begin{equation} \label{mean_zero_kernel}
\bE[h(X_1, \dots, X_m)] = 0,
\end{equation}
as well as
\begin{equation} \label{u_stat_deg_assumption}
 2m <n.
\end{equation}
 An important related function of $h(\cdot) $ is the \emph{canonical function}
\[
g(x) = \bE[h( X_1, \dots, X_{m-1}, x)] = \bE[h(X_1, \dots, X_m)|X_m = x],
\]
which determines the first-order asymptotic behavior of the U-statistic. We will only consider \emph{non-degenerate} U-statistics, which are U-statistics with the property that 
\begin{equation*}
\sigma_g^2 \equiv \text{var}[g(X_1)] > 0. 
\end{equation*}

It is well known that, when $\bE[h^2(X_1, \dots, X_m)] < \infty$, $\frac{\sqrt{n} U_n}{m \sigma_g}$ converges weakly to the standard normal distribution as $n$ tends to infinity \citep[Theorem 4.2.1]{korolyuk2013theory}; however, the limiting variance  $\sigma_g^2$ is typically unknown and has to be substituted with a data-driven estimate. By constructing
\[
q_i \equiv \frac{1}{{n-1 \choose m-1}} \sum_{\substack{1 \leq i_1 < \dots < i_{m-1} \leq n\\ i_l \neq i \text{ for } l = 1, \dots, m-1}} h(X_i, X_{i_1}, \dots, X_{i_{m-1}}), \qquad i = 1, \dots, n,
\]
as natural proxies for $g(X_1), \dots, g(X_n)$, the most common jackknife estimator for $\sigma_g^2$ is
\[
s_n^ 2\equiv  \frac{n-1}{(n-m)^2} \sum_{i=1}^n (q_i - U_n)^2
\]
 \citep{arvesen1969jackknifing}, which gives rise to the \emph{Studentized} U-statistic
  \[
 T_n \equiv \frac{\sqrt{n} U_n}{m s_n}. 
 \]
 Without any loss of generality, we will assume that 
 \begin{equation} \label{unit_var}
 \sigma_g^2 = 1,
 \end{equation} 
 as one can always replace $h(\cdot)$ and $g(\cdot)$ respectively by $h(\cdot)/\sigma_g$ and $g(\cdot)/\sigma_g$ without changing the definition of $T_n$.  Moreover, for $s^*_n$  defined as
\begin{equation*}
{s^*_n}^2 \equiv \frac{n-1}{(n-m)^2} \sum_{i=1}^n q_i^2,
\end{equation*}
 we will also consider the  statistic
\begin{equation} \label{Tnstar}
T_n^* \equiv \frac{\sqrt{n} U_n}{m s^*_n}.
\end{equation}
For any $x\in \bR$, the event-equivalence relationship
\begin{equation} \label{relationship}
 \{T_n > x\} = \left\{    T_n^* > \frac{x }{\left(1 + \frac{m^2 (n-1)x^2}{(n-m)^2}\right)^{1/2}}    \right\} 
\end{equation}
is known in the literature; see \citet{lai2011cramer, shao2016cramer} for instance.

We now state  a uniform Berry-Esseen bound for $T_n$ and $T_n^*$. In the sequel, for any $k \in \{1, \dots, n\}$ and $p \geq 1$, where no ambiguity arises, we may use $\bE[\ell]$ and $\|\ell\|_p$ as the respective shorthands for $\bE[\ell(X_1, \dots, X_k)]$ and $\|\ell(X_1, \dots, X_k)\|_p$, for a given function $\ell: \cX^k \longrightarrow \bR$ in $k$ arguments. For example, we may use $\bE[|h|^3]$ and $\|h\|_3$ to respectively denote the  third absolute moment and $3$-norm of  $h(X_1, \dots, X_m)$ with inserted data, and $\bE[g^2] = \|g\|_2^2 = \sigma_g^2 = 1$ under  \eqref{mean_zero_kernel} and \eqref{unit_var}.

\begin{theorem}[Berry-Esseen bound for Studentized U-statistics] \label{thm:BE_student_U} 
Let $X_1, \dots, X_n$ be 
independent and identically distributed random variables  taking value in a measure space $(\cX, \Sigma_{\mathcal{X}})$. 
Assume \eqref{mean_zero_kernel}-\eqref{unit_var} and 
\begin{equation} \label{3rd_kernel_moment_assumption}
\bE[|h|^3] < \infty,
\end{equation}
then the following Berry-Esseen bound holds:
\begin{equation} \label{BE_Tn}
\sup_{x \in \bR}|P(T_n \leq x) - \Phi(x)| \leq 
 C\frac{ \bE[|g|^3] + m( \bE[h^2] + \|g\|_3\|h\|_3)}{\sqrt{n}}
 \end{equation}
for a positive absolute constant $C$; \eqref{BE_Tn} also holds with $T_n$ replaced by $T_n^*$.

\end{theorem}
To the best of our knowledge, this  bound is the most optimal  to date in the following sense: improving upon the preceding works of \citep{callaert1981order, helmers1985berry, zhao1983rate}, for Studentized U-statistics of degree 2, under the same assumptions as \thmref{BE_student_U}, \citet[Theorem 3.1]{jing2000berry} states a bound of the form
\[
\sup_{x \in \bR} |P(T_n \leq x) - \Phi(x)| \leq C \frac{ \bE[|h(X_1, X_2)|^3]}{\sqrt{n}} 
\]
for an absolute constant $C > 0$. In comparison,  \eqref{BE_Tn} is more optimal for $m =2$ because all  the moment quantities 
\[
 \bE[|g(X_1)|^3] , \quad \bE[|h(X_1, X_2)|^2] \text{ and } \|g(X_1)\|_3\|h(X_1, X_2)\|_3
\] 
from  \eqref{BE_Tn} are all no larger than $\bE[|h(X_1, X_2)|^3]$, given  the standard moment properties for U-statistics; see \eqref{Jensen} below.

In addition, we remark that  the original B-E bound for Studentized U-statistics of degree $2$ in \citet[Theorem 4.2 \& Remark 4.1]{shao2016stein} may have been falsely stated. Given \eqref{mean_zero_kernel}-\eqref{unit_var},  for an absolute constant $C >0$, they stated a seemingly better bound (than \eqref{BE_Tn}) of the form
\[
\sup_{x \in \bR} |P(T_n \leq x) - \Phi(x)| \leq C \frac{\|h(X_1, X_2)\|_2+ \bE[|g(X_1)|^3]}{\sqrt{n}},
\]
under the weaker assumption (than \eqref{3rd_kernel_moment_assumption}) that $\|g(X_1)\|_3\vee\|h(X_1, X_2)\|_2 < \infty$\footnote{Actually, the  bound claimed in \citet[Remark 4.1]{shao2016stein} is $n^{-1/2}(\|h(X_1, X_2)\|_2 + \|g(X_1)\|_3^3)$, but the omission of the exponent $2$ for $\|h(X_1, X_2)\|_2$ is itself a typo in that paper.}. Unfortunately, the latter  assumption is inadequate under the current approach based on  Stein's method. The main issue is that  \citet{shao2016stein} has ignored  crucial calculations that require forming estimates of the rate $O(1/n)$ for an expectation of the type
\[
\bE[\xi_{b, 1} \xi_{b, 2} \bar{h}_2 (X_{i_1},  X_{i_2})  \bar{h}_2(X_{j_1}, X_{j_2})],
\]
where  $1 \leq i_1 <  i_2\leq n$ and $1 \leq j_1 <  j_2\leq n$ are  two pairs of sample indices, and $\bar{h}_2(\cdot)$ is the second-order canonical function in the Hoeffding's decomposition of $U_n$ for $m = 2$; see \eqref{h_k_def}.
To do so,  we believe one cannot do away with a third moment assumption on the kernel as in \eqref{3rd_kernel_moment_assumption}, where the anxious reader can skip ahead to \lemref{ustat_results}$(iii)$ and $(iv)$ for a preview of our estimates. 
Our proof of \thmref{BE_student_U} rectifies such errors; moreover, it  generalizes to a kernel of any degree $m$, for which the enumerative calculations needed are considerably more involved.

We first set the scene for establishing \thmref{BE_student_U}, by letting
\begin{equation} \label{u_stat_xi}
\xi_i = \frac{g(X_i)}{\sqrt{n}}
\end{equation}
and defining
\begin{equation} \label{h_k_def}
\bar{h}_k(x_1 \dots, x_k) = h_k(x_1 \dots, x_k) - \sum_{i=1}^k g(x_i) \text{ for } k = 1, \dots, m,
\end{equation}
where 
\[
h_k(x_1, \dots, x_k) = \bE[h(X_1, \dots, X_m) |X_1 = x_1, \dots, X_k = x_k ];
\] 
in particular, $g(x) = h_1(x)$ and $h(x_1, \dots, x_m) = h_m(x_1, \dots, x_m)$. An important property of the functions $h_k$ is that 
\begin{equation} \label{Jensen}
\bE\big[ |h_k|^p\big] \leq \bE\big[ |h_{k'}|^p\big] \text{ for any } p \geq 1 \text{ and } k \leq k',
\end{equation}
which is a  consequence of Jensen's inequality: 
 \begin{align*} \label{Jensen}
\bE\Big[ |h_{k}(X_1, \dots, X_k )|^p\Big] 
&= \bE\Big[ |\bE[h(X_1, \dots, X_m) |X_1, \dots, X_k]|^p\Big]  \\
&= \bE\Big[ \Big|\bE[h_{k'}(X_1, \dots, X_{k'})| X_1, \dots, X_k]\Big|^p\Big] \\
&\leq  \bE\Big[ \bE\Big[|h_{k'}(X_1, \dots, X_{k'})|^p \mid X_1, \dots, X_k\Big]\Big]  =  \bE\Big[ |h_{k'}(X_1, \dots, X_{k'})|^p\Big].
\end{align*}
One can then  write the part of \eqref{Tnstar} without the Studentizer $s^*_n$ as
\begin{equation} \label{W_plus_D1}
\frac{\sqrt{n} U_n}{m} = W_n + D_{1 n},
\end{equation}
where $W_n \equiv \sum_{i=1}^n \xi_i$ and 
\begin{equation} \label{D1_ustat}
D_{1n}
 \equiv {n -1 \choose m -1}^{-1} \sum_{1 \leq i_1 < \dots < i_m \leq n} \frac{\bar{h}_m(X_{i_1}, X_{i_2}, \dots, X_{i_m}) }{\sqrt{n}},
\end{equation}
 are considered as the numerator components under the framework of \eqref{sn_stat}. 
To handle $s^*_n$,  we shall first define 
\[
\Psi_{n, i} =  \sum_{\substack{1 \leq i_1 < \dots < i_{m-1} \leq n\\ i_l \neq i \text{ for } l = 1, \dots, m-1}} \frac{\bar{h}_m(X_i, X_{i_1}, \dots, X_{i_{m-1}}) }{\sqrt{n}} 
\]
and write
\begin{align*}
q_i &
=  \frac{1}{{n-1 \choose m-1}} \sum_{\substack{1 \leq i_1 < \dots < i_{m-1} \leq n\\ i_l \neq i \text{ for } l = 1, \dots, m-1}} \left[g(X_i) + \sum_{l=1}^{m-1} g(X_{i_l})+ \bar{h}_m(X_i, X_{i_1}, \dots, X_{i_{m-1}}) \right]\\
&
= \sqrt{n}\left[\left(\frac{n-m}{n-1}\right)\xi_i+ \frac{m-1}{n-1}W_n\right]+ \frac{\sqrt{n}}{{n-1 \choose m-1}} \Psi_{n, i}
\end{align*}
for each $i$. By further letting
\[
 \Lambda_n^2 = \sum_{i=1}^n \Psi_{n, i}^2 \quad \text{ and } \quad V_n^2 = \sum_{i=1}
^n \xi_i^2,
\]
the sum $\sum_{i=1}^n q_i^2$ can be consequently written as 
 \begin{multline*}
 \sum_{i=1}^n q_i^2 =  n \left(\frac{n-m}{n-1}\right)^2 V_n^2  + \left[n^2 \left(\frac{m-1}{n-1}\right)^2 + \frac{2n(n-m)(m-1)}{(n-1)^2}\right]W_n^2 \\
+ \frac{n}{{n-1 \choose m-1}^2} \Lambda_n^2 + 2 n \left(\frac{n-m}{n-1}\right) {n-1 \choose m-1}^{-1} \sum_{i=1}^n \xi_i \Psi_{n,i} + \frac{2 n (m-1)}{(n-1){n-1 \choose m-1}} \sum_{i=1}^nW_n \Psi_{n, i},
\end{multline*}
which implies one can re-express ${s_n^*}^2$ as
\begin{equation} \label{s_star_re_expr}
{s_n^*}^2 = d_n^2 (V_n^2 + \delta_{1n} + \delta_{2n}) \quad \text{ for }\quad d_n^2 \equiv \frac{n}{n-1}
\end{equation}
for
\begin{equation}\label{delta_1n}
\delta_{1n} = \left[ \frac{ n(m-1)^2}{(n-m)^2} + \frac{2(m-1)}{(n-m)}\right] W_n^2 + \frac{(n-1)^2}{ {n-1 \choose m-1}^2 (n-m)^2} \Lambda_n^2 + \frac{2(n-1)(m-1) }{(n-m)^2 {n-1 \choose m-1}} \sum_{i = 1}^n  W_n \Psi_{n, i}
\end{equation}
and
\[
\delta_{2n} \equiv \frac{2 (n-1) }{(n-m)} {n-1 \choose m-1}^{-1} \sum_{i=1}^n \xi_i \Psi_{n,i}.
\]
We now present  the proof of \thmref{BE_student_U}.

 \begin{proof}[Proof of \thmref{BE_student_U}]
It suffices to consider $x \geq 0$  since otherwise one can replace $h(\cdot)$ by $- h(\cdot)$. Defining 
\[
b_n = \frac{ m^2(n-1)}{ (n-m)^2} \text{ and } a_{n, x} = a_{n}(x) = \frac{1}{ (1 + b_n x^2)^{1/2}},
\]
we first simplify the problem using the bound
\begin{equation} \label{bridging}
 | \barPhi(x a_n(x)) -  \barPhi(x)| \leq \min \left(  \frac{ m^2(n-1)x^3}{\sqrt{2 \pi} (n-m)^2}  ,  \frac{2}{\max(2 , \sqrt{2 \pi} x a_{n, x})}  \right)\exp\left( \frac{-x^2 a_{n, x}^2}{2}\right),
\end{equation}  
which will be shown by a ``bridging argument" borrowed from \citet{jing2000berry} at the end of this section. Then,  
by the triangular inequality, \eqref{relationship} and \eqref{bridging},
\begin{align}
&|P(T_n \leq x) - \Phi(x)|\notag\\
 &= |P(T_n > x) - \barPhi(x)| \notag\\
&\leq  |P(  T_n^* > x a_n(x) ) - \barPhi(x a_n(x))| + | \barPhi(x a_n(x)) -  \barPhi(x)| \notag\\
&\leq    |P(  T_n^* > x a_n(x) ) - \barPhi(x a_n(x))| + \min \left(  \frac{ m^2(n-1)x^3}{\sqrt{2 \pi} (n-m)^2}  ,  \frac{2}{\max(2 , \sqrt{2 \pi} x a_{n, x})}  \right)\exp\left( \frac{-x^2 a_{n, x}^2}{2}\right) \notag\\
&\leq  |P(  T_n^* > x a_n(x) ) - \barPhi(x a_n(x))| +C\frac{m^2}{\sqrt{n}}, \label{bridging_to_Tn_star_BE}
\end{align}
where the last inequality in \eqref{bridging_to_Tn_star_BE} holds as follows: For $0 \leq x \leq n^{1/6}$, the term 
\[
\frac{ m^2(n-1)x^3}{\sqrt{2 \pi} (n-m)^2} \leq \frac{m^2(n-1) \sqrt{n}}{\sqrt{2 \pi} (n-m)^2} \leq  \frac{m^2(n-1) \sqrt{n}}{\sqrt{2 \pi} (n-n/2)^2} \leq \frac{2 \sqrt{2} m^2}{\sqrt{\pi n}}.
\]  For $n^{1/6} <  x < \infty$,
since $x a_n(x)$ is strictly increasing in $x \in [0, \infty)$, we have that 
\begin{align*}
\exp(-x^2 a_{n, x}^2 /2) &\leq\exp(-n^{1/3} (1 + b_n n^{1/3})^{-1} /2)\leq \exp\bigg(-\frac{n^{1/3}}{2} \left(1 +  \frac{4 m^2 (n-1)n^{1/3}}{n^2}\right)^{-1}\bigg) \\
&
\underbrace{\leq}_{\text{by  \eqref{u_stat_deg_assumption}} } \exp\left(- \frac{n^{1/3}}{2 (1 +  (2m)^{4/3})} \right) \leq \exp\big(- \frac{n^{1/3}}{8m^{4/3}}\big) \leq \frac{Cm^2}{\sqrt{n}}.
\end{align*}
Since
\begin{equation}\label{extract_extra_m}
m = m\bE[g^2] \leq \bE[h^2]
\end{equation} 
by \eqref{unit_var} and a classical U-statistic moment bound \citep[Lemma 1.1.4]{korolyuk2013theory}, in light of \eqref{bridging_to_Tn_star_BE}, to prove \eqref{BE_Tn} it suffices to show 
\begin{equation} \label{Tn_star_BE}
|P(T_n^* > x) - \barPhi(x)|  \leq 
 C\frac{\bE[|g|^3]+ m ( \bE[h^2] + \|g\|_3\|h\|_3)}{\sqrt{n}},
\end{equation}
as we have claimed to also hold in \thmref{BE_student_U}.

 Note that since  $2 |W_n   \sum_{i=1}^n \Psi_{n, i}|  \leq 2 \sqrt{n} |W_n| \Lambda_n$ by Cauchy's inequality, 
\begin{multline}\label{WPsi_bdd}
 \frac{2(n-1)(m-1) }{(n-m)^2 {n-1 \choose m-1}} \Bigg|\sum_{i = 1}^n  W_n \Psi_{n, i}\Bigg| \leq 2 \Bigg\{ \frac{\sqrt{n}(m-1)}{n-m}|W_n| \Bigg\} \Bigg\{\frac{(n-1)}{{n-1 \choose m-1}(n-m) } \Lambda_n \Bigg\}\\
\leq \frac{n (m-1)^2}{(n-m)^2} W_n^2 + \frac{(n-1)^2}{{n-1 \choose m-1}^2 (n-m)^2 } \Lambda_n^2,
\end{multline}
 and hence we can deduce from \eqref{delta_1n} that 
\begin{equation} \label{delta1ngeq0}
\delta_{1n} \geq 0.
\end{equation}
With \eqref{W_plus_D1} and \eqref{s_star_re_expr}, one can then rewrite $T_n^*$ as
\begin{align*}
T_n^* 
&= \frac{W_n +  D_{1n} }
{ d_n \sqrt{V_n^2  + \delta_{1n} + \delta_{2n}}}.\\
\end{align*}
Now, consider the related statistic
\[
\tilde{T}^*_n = \frac{W_n + D_{1n}}
{ \{\max(0, V_{n,b}^2 +  \delta_{1n, b}  + \delta_{2n, b})\}^{1/2}},
\]
with  suitably censored components in the denominator defined as
\begin{multline*}
 V^2_{n, b} = \sum_{i =1}^n \xi_{b, i}^2, \quad \delta_{1n, b} =  \min(\delta_{1n}, n^{-1/2}) 
 \quad \text{ and } \quad 
 \delta_{2n, b} =  \frac{2 (n-1)}{(n-m)} {n-1 \choose m-1}^{-1} \sum_{i=1}^n \xi_{b, i} \Psi_{n,i}, 
\end{multline*}
 Note that $T_n^* $ and $\tilde{T}^*_n $ can be related by the inclusions of events
\[
\{\tilde{T}^*_n  \leq d_n x\}\backslash \cE \subset \{T^{*}_n \leq x\} \subset  \{\tilde{T}^*_n  \leq d_n x\}\cup \cE,
\]
where  $\cE \equiv \{\max_{1\leq i \leq n} |\xi_i| > 1\} \cup\{|\delta_{1n}| > n^{-1/2}\}$. The latter fact 
 implies
\begin{align}
|P(T^{*}_n \leq x) - \Phi(x)| 
&\leq |P(\tilde{T}^*_n  \leq d_n x) - \Phi(x)| + P(\cE) \notag\\
&\leq |P(\tilde{T}^*_n  \leq d_n x) - \Phi(x)| + \sum_{i=1}^n P(|\xi_i| > 1) + P(|\delta_{1n}| > n^{-1/2})\notag\\
&\leq  |P(\tilde{T}^*_n  \leq d_n x) - \Phi(x)| + \beta_2 + \sqrt{n}\bE[ |\delta_{1n}|] \notag \\
& \leq  |P(\tilde{T}^*_n  \leq d_n x) - \Phi(x)| + 
 \frac{ \bE[|g|^3]}{\sqrt{n}}+C\frac{ m \bE [ h^2]}{\sqrt{n}} ,  \label{last_bdd}
\end{align} 
with \eqref{last_bdd} coming from $\beta_2 \leq  \sum_{i=1}^n \bE[|\xi_i|^3] =  \bE[|g|^3]/\sqrt{n}$, as well as combining \eqref{WPsi_bdd} with \eqref{delta_1n} as:
\begin{align*} \label{cauchy_bdd}
&\bE [|\delta_{1n} |]\\
&\leq  2 \left[\frac{ m (m-1) (n -1)}{(n-m)^2} \right]\bE[W_n^2]   +  \frac{2(n-1)^2}{ {n-1 \choose m-1}^ 2 (n-m)^2 }\bE[\Lambda_n^2 ]\\
 & =  2 \left[\frac{ m (m-1) (n -1)}{(n-m)^2} \right] +  
 \frac{2(n-1)^2}{ {n-1 \choose m-1}^ 2 (n-m)^2 }
 \bE \left[ \left(\sum_{2 \leq i_1 < \dots < i_{m-1} \leq n}\bar{h}_m(X_1, X_{i_1}, \dots, X_{i_{m-1}}) \right)^2\right] \\
 &\leq    \Bigg(\frac{8m}{n} + 
  \frac{4(n-1)^2(m-1)^2}{ (n-m)^2(n-m+1)m } \Bigg)
 \bE [ h^2],
\end{align*}
where the last inequality follows from \eqref{extract_extra_m} and $2m < n$, as well as a standard U-statistic bound in \lemref{ustat_results}$(ii)$. 

In light of \eqref{last_bdd}, to prove \eqref{Tn_star_BE}, it suffices to bound  $|P(\tilde{T}^*_n  \leq d_n x) - \Phi(x)|$. To this end, we first define 
\[
\check{T}^{**}_n = \frac{W_n + D_{1n}}
{ \{\max(0, V_{n,b}^2  + \delta_{2n, b})\}^{1/2}}
\]
and 
\[
\hat{T}^{**}_n = \frac{W_n + D_{1n}}
{ \{\max(0, V_{n,b}^2  + n^{-1/2} + \delta_{2n, b})\}^{1/2}},
\]
which, by \eqref{delta1ngeq0}, have the property
\begin{equation}\label{sandwich}
P(\check{T}^{**}_n \leq d_n x)   \leq  P(\tilde{T}^*_n  \leq d_n x)   \leq P(\hat{T}^{**}_n \leq d_n x )
\end{equation}
Hence, to establish a bound for $|P(\tilde{T}^*_n  \leq d_n x) - \Phi(x)|$, our strategy is to prove the same bound for $|P(\check{T}^{**}_n \leq d_n x) - \Phi(d_n x)|$ and 
$|P(\hat{T}^{**}_n \leq d_n x) - \Phi(d_n x)|$, as well as using the bound
\begin{equation} \label{Phi_dn_x_minus_Phi_x}
|\Phi(d_n x) - \Phi(x)| =  \phi(x')(d_n x - x) \leq C (d_n- 1) \leq C n^{-1/2},
\end{equation}
coming from the mean-value theorem, where $x' \in (x, d_nx)$ and $x\phi(x')$ is a bounded function in $x \in [0, \infty)$. 
To simplify notation we will put $\check{T}^{**}_n$ and $\hat{T}^{**}_n$ under one umbrella and define their common placeholder 
\begin{equation} \label{placeholder}
T_n^{**} = \frac{W_n + D_{1n}}
{(1 + D_{2n})^{1/2}},
\end{equation}
where 
\begin{equation} \label{D2_ustat}
D_{2n} \equiv \max(-1 , V_{n,b}^2 - 1+  (n^{-1/2}|0)+ \delta_{2n, b})
\end{equation}
and for $a, b \in \bR$, $(a|b)$ represents either $a$ or $b$; so $T_n^{**}$ is either $\hat{T}_n^{**}$ or $\check{T}_n^{**}$.

Now, we cast \eqref{D2_ustat}  into the form  \eqref{D2_specific_form} by taking $\Pi_1 = V_{n, b}^2 - \sum_{i=1}^n \bE[\xi_{b,i}^2]$ and 
\begin{equation} \label{Pi2_for_ustat}
\Pi_2 = \delta_{2n, b} + (n^{-1/2}|0)  - \sum_{i=1}^n \bE[ (\xi_i^2 - 1)I(|\xi_i|> 1) ]
\end{equation}
In order to apply \thmref{main2} to bound $|P(T^{**}_n \leq d_n x) - \Phi(d_n x)|$, we will  let $D_{1n}^{(i)}$ and $\Pi_2^{(i)}$ respectively to be the ``leave-one-out" versions of $D_{1n}$ and $\Pi_2$  in \eqref{D1_ustat} and  \eqref{Pi2_for_ustat} that omit all the terms involving $X_i$, i.e,  
\begin{equation} \label{D1ni_def}
D_{1n}^{(i)} \equiv {n -1 \choose m -1}^{-1} \sum_{\substack{1 \leq i_1 < \dots < i_m \leq n\\ i_l \neq i \text{ for } l = 1, \dots, m}} \frac{\bar{h}_m(X_{i_1}, X_{i_2}, \dots, X_{i_m}) }{\sqrt{n}}
\end{equation}
and 
\begin{equation} \label{Pi2_for_ustati_def}
\Pi_2^{(i)} \equiv  \delta_{2n, b}^{(i)} + (n^{-1/2}|0) - \sum_{\substack{j =1\\j \neq i}}^n\bE[ (\xi_j^2 - 1)I(|\xi_j|> 1) ] 
\end{equation}
for
\[
\delta_{2n, b}^{(i)} \equiv \frac{2 (n-1) }{\sqrt{n}(n-m)} {n-1 \choose m-1}^{-1} \sum_{\substack{j =1\\ j \neq i}}^n \xi_{b, j} \sum_{\substack{1 \leq i_1 < \dots < i_{m-1} \leq n \\
i_l \neq j, i \text{ for } l  = 1, \dots, m-1}} \bar{h}_m (X_j, X_{i_1}, \dots, X_{i_{m-1}}).
\]
We also need the following bounds:
\begin{lemma}[Moment bounds related to $D_{1n}$ in \eqref{D1_ustat}] \label{lem:Dbdds}  Let $D_{1n}$ and $D_{1n}^{(i)}$ be defined as in \eqref{D1_ustat} and   \eqref{D1ni_def}. Under the assumptions of \thmref{BE_student_U}, the following  hold: 

\begin{equation} \label{bar_D_1n_bdd}
\|D_{1n}\|_2 \leq \frac{(m-1) \|h\|_2}{\sqrt{m(n-m+1)}},
\end{equation}
and
\begin{equation}\label{D1n_minus_D1ni_bdd}
\|D_{1n} - D_{1n}^{(i)} \|_2 \leq \frac{\sqrt{2} (m-1) \|h\|_2}{\sqrt{nm(n-m+1)}}
\end{equation}

\end{lemma}
\begin{proof}[Proof of \lemref{Dbdds}]
This is known in the literature. Refer to \citet[Lemma 10.1]{chen2011normal} for a proof.
\end{proof}

\begin{lemma} [Moment bounds related to $\Pi_2$ in   \eqref{Pi2_for_ustat}] \label{lem:Djn_minus_D_ni}
Consider $\Pi_2$ and $\Pi_2^{(i)}$ defined in   \eqref{Pi2_for_ustat} and  \eqref{Pi2_for_ustati_def}. Under the assumptions of \thmref{BE_student_U}, the following bounds hold:
\begin{enumerate}
\item
\begin{equation*} \label{bar_D_2n_bdd}
\|\Pi_2\|_2  \leq  C\frac{\|g\|_3^3 + m\|g\|_3 \|h\|_3}{\sqrt{n}},
\end{equation*}
and
\item 
\begin{equation*}\label{D2n_minus_D2ni_bdd}
\|\Pi_2 - \Pi_2^{(i)} \|_2 \leq  C \frac{m\|g\|_3\|h\|_3  + m^{1.5}  \sqrt{\|h\|_2}}{n} \end{equation*}
\end{enumerate}
\end{lemma}

The proof of \lemref{Djn_minus_D_ni} is deferred to \appref{proof_of_D2_stuff}. 
One can then apply \thmref{main2},  along with \lemsref{Dbdds} and \lemssref{Djn_minus_D_ni} as well as \eqref{extract_extra_m},  to give the  bound
\begin{equation}\label{BE_with_threshold_dnx}
|P(T^{**}_n \leq d_n x) - \Phi(d_nx)| \leq  C\frac{ \bE[|g|^3] + m(\|g\|_3 \|h\|_3 + \|h\|_2^{3/2})}{\sqrt{n}}
\end{equation}
where we have used the fact that $\sigma_g^2 = 1$ in \eqref{unit_var} and  $\sigma_g \leq \|h\|_2$ by virtue of \eqref{Jensen}.
From \eqref{BE_with_threshold_dnx}, one can  establish  
\eqref{Tn_star_BE} with \eqref{last_bdd}-\eqref{placeholder} and that $\|h\|_2^{3/2} \leq \bE[h^2]$.


\ \

\ \

 It remains to finish the proof for \eqref{bridging}: First,  it can be seen that
\begin{equation} \label{anx_bdd}
0 < a_{n, x} \leq 1.
\end{equation}
Because of \eqref{anx_bdd}, we have
\begin{align*}
|x a_{n, x} - x| &= 
\left|\frac{ (a_{n, x}^2 - 1)x}{a_{n,x } +1}\right| = 
\left| \left( \frac{b_n }{ 1 + b_n x^2} \right) \left(\frac{x^3}{ a_{n, x} + 1}   \right)\right|
\leq b_n x^3
=  \frac{ m^2(n-1)x^3}{ (n-m)^2},
\end{align*}
which implies, by the mean-value theorem,  that
\begin{equation*} 
|\Phi(x a_{n, x}) - \Phi(x)|
 \leq  \phi(x a_{n, x} )\frac{ m^2(n-1)x^3}{ (n-m)^2} =\frac{ m^2(n-1)x^3}{\sqrt{2 \pi} (n-m)^2}  \exp\left( \frac{-x^2 a_{n, x}^2}{2}\right)  .
\end{equation*}
At the same time, we also have, by the well-known normal tail bound and \eqref{anx_bdd},
\begin{equation*} 
|\Phi(x a_{n, x} ) - \Phi(x)| \leq \barPhi(x a_{n, x} ) +\barPhi(x)
 \leq \frac{2}{\max(2 , \sqrt{2 \pi} x a_{n, x})}  \exp\left(\frac{- x^2 a_{n, x}^2}{2} \right).
\end{equation*}

\end{proof}

\section*{Acknowledgment}

This research is partially supported by National Nature Science Foundation of China NSFC 12031005 and Shenzhen Outstanding Talents Training Fund, China. Liqian Zhang would like to express sincere gratitude to the Professor Maurice H. Belz Scholarship for its generous support.

\appendix

\newpage
\section{Technical lemmas}

The first two lemmas below concern properties of the $\xi_{b,i}$'s and their sum $W_b$.

\begin{lemma}[Bound on expectation of $\xi_{b, i}$] \label{lem:exp_xi_bi_bdd} Let $\xi_{b, i} = \xi_i I(|\xi_i| \leq 1) + 1 I(\xi_i >1) - 1  I(\xi_i <-1)$ with $\bE[\xi_i] = 0$. Then
\[
  \left|\bE[ \xi_{b, i}]\right| \leq   \bE[ \xi_i^2I(|\xi_i| > 1) ]  \leq  \bE[ \xi_i^2 ]
\]
\end{lemma}
\begin{proof}[Proof of \lemref{exp_xi_bi_bdd}]
 \begin{align*}
  \left|\bE[ \xi_{b, i}]\right| &= \left|\bE[ (\xi_i - 1) I(\xi_i> 1) + (\xi_i + 1)  I(\xi_i < - 1) ]\right| \\
  &\leq \bE[ (|\xi_i| - 1)  I(|\xi_i|> 1)] \leq \bE[ |\xi_i |  I(|\xi_i|> 1)]  \leq \bE[ |\xi_i |^2  I(|\xi_i|> 1)] \leq  \bE[ \xi_i^2 ].
 \end{align*}
\end{proof}

\begin{lemma} [Bennett's inequality for a sum of censored random variables]  \label{lem:Bennett} 
Let $\xi_1, \dots, \xi_n$ be independent random variables with $\bE[\xi_i] = 0$ for all $i = 1, \dots, n$ and $\sum_{i=1}^n \bE[\xi_i^2] \leq 1$, and define $\xi_{b, i} = \xi_i I(|\xi_i| \leq 1) + 1 I(\xi_i >1) - 1  I(\xi_i <-1)$. 
For any $t >0$ and $W_b =  \sum_{i=1}^n \xi_{b, i}$, we have
\[
\bE[e^{t W_b}] \leq  \exp\left( e^{2t}/4 - 1/4 + t/2\right)
\]

\end{lemma}

\begin{proof}[Proof of \lemref{Bennett}]
Note that, by \lemref{exp_xi_bi_bdd},
\[
\bE[e^{t W_b}] = \bE[e^{t (W_b- \bE[W_b])}] e^{t \bE[W_b]} \leq   \bE[e^{t \sum_{i=1}^n(\xi_{b,i} - \bE[\xi_{b, i}])}] e^t.
\]
Moreover, by the standard Bennett's inequality \citep[Lemma 8.1]{chen2011normal},
\[
 \bE[e^{t \sum_{i=1}^n(\xi_{b,i} - \bE[\xi_{b, i}])}] \leq \exp\left( 4^{-1} (e^{2t} - 1 - 2t)\right). 
\]

\end{proof}

The next lemmas concern properties of the solution to the Stein equation, $f_x$ in \eqref{fx}. It is customary to \emph{define} its derivative at $x$ as $f_x'(x) \equiv xf_x(x) +\bar{\Phi}(x)$ so the Stein equation \eqref{steineqt} is valid for all $w$. Moreover, we define 
\begin{equation} \label{gw_def}
g_x(w) = (wf_x(w))' = f_x(w)+ wf_x'(w).
\end{equation}
Precisely,
\begin{equation} \label{fx'}
f_x' (w) = 
  \begin{cases} 
\left(  \sqrt{2 \pi}w e^{w^2/2} \Phi(w) + 1\right)\bar{\Phi}(x)& \text{ for }\quad w \leq  x\\
\left(  \sqrt{2 \pi}w e^{w^2/2} \bar{\Phi}(w) - 1\right) \Phi(x)        & \text{ for }\quad w > x
     \end{cases};
\end{equation}
\begin{equation} \label{wfx'}
g_x(w) = 
  \begin{cases} 
 \sqrt{2 \pi}\bar{\Phi}(x) \left( (1 + w^2)  e^{w^2/2} \Phi(w) + \frac{w}{\sqrt{2\pi}}\right) & \text{ for }\quad w \leq x\\
  \sqrt{2\pi}\Phi(x) \left(   (1 + w^2) e^{w^2/2}    \bar{\Phi}(w) - \frac{w}{\sqrt{2 \pi}} \right)  &\text{ for }\quad  w > x
     \end{cases}.
\end{equation}

\begin{lemma}[Uniform bounds for $f_x$] \label{lem:helping_unif}
For $f_x$ and $f_x'$, the following bounds are true:
\[
 |f_x'(w)| \leq 1, \quad  \quad  0 < f_x(w) \leq 0.63 \quad  \text{ and }\quad 0 \leq g_x(w) \quad \text{ for all } \quad w, x \in \bR.
\]
Moreover, for any $x \in [0, 1]$, $g_x(w) \leq 2.3$ for all $w \in \bR$.
\end{lemma}

\begin{lemma}[Nonuniform bounds for $f_x$ when $x \geq 1$] \label{lem:helping}
For $x\geq 1$, 
the following are true:  \begin{equation} \label{fxbdd}
 f_x (w) \leq  
  \begin{cases} 
  1.7 e^{-x} &  \text{ for }\quad w \leq x - 1\\
  1/x     &    \text{ for }\quad x  - 1 < w \leq x \\
   1/w     &   \text{ for }\quad x  < w 
     \end{cases}
     \end{equation}
     and
    \begin{equation} \label{fx'bdd}
      | f_x'(w)| \leq  
  \begin{cases} 
    e^{1/2 -x } &   \text{ for }\quad w \leq x - 1\\
   1       &    \text{ for }\quad x  - 1 < w \leq x \\
      (1 + x^2)^{-1}    &  \text{ for }\quad w > x  
     \end{cases}.
\end{equation}
Moreover, $g_x(w) \geq 0$ for all $w \in \mathbb{R}$, 
\begin{equation}  \label{gbdd}
 g_x(w) \leq  
  \begin{cases} 
    1.6 \;\bar{\Phi}(x) &  \text{ for }\quad w \leq 0\\
   1/w      &  \text{ for }\quad w > x
     \end{cases},
\end{equation}
and $g_x(w)$ is increasing for $0 \leq w  \leq x$ with 
\[
g_x(x - 1) \leq x e^{1/2 -x} \;\ \text{ and } \;\ g_x(x)  \leq  x+ 2.
\]
\end{lemma}
We remark that the nonuniform bounds in \lemref{helping} refine the ones previously collected in  \citet[Lemma 5.3]{shao2016stein}; as an aside, a property analogous to \eqref{fx'bdd} has been incorrectly stated in \citet{shao2016stein} without the absolute signs $|\cdot|$ around $f_x'(w)$. The proofs below repeatedly use the  well-known inequality \citep[p.16 \& 38]{chen2011normal}
\begin{equation} \label{qfn_bdd}
\frac{w e^{-w^2/2}}{ (1 + w^2)\sqrt{2\pi}}\leq \bar{\Phi}(w)\leq \min \left( \frac{1}{2}, \frac{1}{w \sqrt{2 \pi}}\right) e^{-w^2/2} \text{ for } w > 0.
\end{equation}

\begin{proof}[Proof of \lemref{helping_unif}]
The bounds for $f_x$ and $f_x'$, and that $g_x(w) \geq 0$, are well-known; see \citet[Lemma 2.3]{chen2011normal}. We will show that $g_x$ in  \eqref{wfx'} is less than $2.3$ when $x \in [0,1]$.  Using \eqref{qfn_bdd}, for $w > x$,  we have
\begin{multline*}
g_x(w) \leq  \sqrt{2\pi}\Phi(x) \left(   (1 + w^2) e^{w^2/2}    \bar{\Phi}(w) - \frac{w}{\sqrt{2 \pi}} \right)\\
 \leq  \sqrt{2\pi}\Phi(x) \left(\frac{1}{2} + \frac{w}{\sqrt{2\pi}}- \frac{w}{\sqrt{2 \pi}}\right) \leq \frac{\sqrt{2\pi} \Phi(x)}{2} \leq 2.
\end{multline*}
For $0 \leq w \leq x$,
\begin{align*}
 g_x(w) &=  \sqrt{2 \pi}\bar{\Phi}(x) \left( (1 + w^2)  e^{w^2/2} \Phi(w) + \frac{w}{\sqrt{2\pi}}\right)   \\
& \leq  \sqrt{2 \pi}\bar{\Phi}(x)  \left( (1 + x^2 ) e^{x^2/2} \Phi(x) + \frac{x}{\sqrt{2 \pi}} \right)\\
 & \leq \left\{ \left(\frac{\sqrt{2\pi}}{2} + x\right) \Phi(x)+ \frac{e^{-x^2/2}}{\sqrt{2\pi}}\right\} \vee \left( \sqrt{2\pi}\bar{\Phi}(0)\cdot \Phi(0) \right)\\
&  \leq\left\{ \left(\frac{\sqrt{2 \pi}}{2} + 1\right)\Phi(1) + \frac{1}{\sqrt{2\pi}} \right\}\vee 0.63 \leq 2.3.
\end{align*}
For $w < 0$,
\begin{multline*}
 \sqrt{2 \pi}\bar{\Phi}(x) \left( (1 + w^2)  e^{w^2/2} \Phi(w) + \frac{w}{\sqrt{2\pi}}\right)   
 \leq  \sqrt{2 \pi}\bar{\Phi}(x)  \left( \frac{1}{2} + \frac{|w|}{\sqrt{2\pi}}- \frac{|w|}{\sqrt{2 \pi}} \right) \leq   1.26.
\end{multline*}
\end{proof}

\begin{proof}[Proof of \lemref{helping}]

{\bf Proof of  \eqref{fxbdd} by  investigating \eqref{fx}}: When $w\leq 0$,  by  \eqref{qfn_bdd}, $x^2 \geq 2x - 1$,  and the symmetry of $\phi(\cdot)$, we have that
\[
f_x(w) \leq  e^{w^2/2} \Phi(w) \frac{e^{-x^2/2}}{x} \leq \frac{e^{-x^2/2}}{2x} \leq \frac{e^{-x +1/2}}{2}  \leq 0.9 e^{-x}.
\]
When $0 < w \leq x -1$, by \eqref{qfn_bdd}, we have
\[
f_x (w) \leq  e^{(x-1)^2/2} \Phi(w) \frac{e^{-x^2/2}}{x} =  \Phi(w) e^{-x + 1/2} \leq 1.7 e^{-x}.
\]
When $x - 1 < w \leq x$,  by \eqref{qfn_bdd}, we have
\[
f_x (w)\leq  \frac{e^{(w^2 - x^2)/2} \Phi(w)}{x}  \leq \frac{1}{x}.
\]
When $w > x$, by \eqref{qfn_bdd}, we have
\[
f_x (w)\leq  \frac{\Phi(x)}{w} \leq  \frac{1}{w}.
\]

{\bf Proof of  \eqref{fx'bdd} by investigating \eqref{fx'}}: When $w \leq 0$, by the symmetry of $\phi(\cdot)$, \eqref{qfn_bdd} and $x^2 \geq 2x - 1$, we have 
\[
0 =  0 \cdot\bar{\Phi}(x) \leq f_x'(w) \leq       \left( \frac{1}{ 1 + w^2} \right) \frac{e^{-x +1/2}}{\sqrt{2 \pi}} \leq 0.4 e^{1/2-x}.
\]
When $ 0 < w \leq x - 1$, by   \eqref{qfn_bdd}  and $x^2 \geq 2x - 1$,
\[
0 \leq f_x'(w)  \leq  \left( \sqrt{2\pi} (x-1) e^{\frac{(x-1)^2}{2}} +1\right)  \frac{e^{-x^2/2}}{x\sqrt{2 \pi}} \leq \left(\frac{x-1}{x}  + \frac{1}{ x\sqrt{2 \pi}}\right) e^{  1/2 - x}  \leq   e^{1/2-x},
\]
as $ \left(\frac{x-1}{x}  + \frac{1}{ x\sqrt{2 \pi}}\right)$ is increasing as a function in $x$ on $[1, \infty)$. When $x- 1 < w \leq x$, by   \eqref{qfn_bdd} we have
\[
0 \leq f_x'(w) =     \Phi(w)\underbrace{\sqrt{2 \pi} w e^{w^2/2}\bar{ \Phi}(x) }_{\leq 1} + \bar{\Phi}(x) \leq 1.
\]
When $ w > x$, since $\sqrt{2\pi}w e^{w^2/2} \bar{\Phi}(w) \leq 1$ by   \eqref{qfn_bdd}, hence $f'_x(w) \leq 0$. Moreover, by applying \eqref{qfn_bdd} again, we have
\[
 \frac{-1}{x^2 +1}\leq \left(  \frac{w^2}{w^2 +1} - 1\right)\Phi(x)\leq f'_x(w) \leq 0.
\]

{\bf Proof of   \eqref{gbdd} by investigating \eqref{wfx'}}: When $w < 0$, by the symmetry of $\phi$ and \eqref{qfn_bdd},
\[
0 = \sqrt{2\pi}\bar{\Phi}(x) \cdot 0 \leq g_x(w) \leq \left( \min\left( \frac{1 + w^2}{ |w|}, \frac{ (1 + w^2)\sqrt{2 \pi}}{2} \right)+ w \right) \bar{\Phi}(x)  \leq  1.6 \bar{\Phi}(x),
\]
where the last inequality uses the facts that $\frac{ (1 +w^2)\sqrt{2 \pi}}{2} +w \leq 1.6$ for $w \in [-1, 0]$ and that $\frac{1 + w^2}{|w|} + w = 1/|w|^2 \leq 1$ for $w  < -1$. When $w > x$, by \eqref{qfn_bdd},
\[
0 \leq \sqrt{2\pi} \Phi(x) \cdot 0\leq g_x(w) \leq  \Phi(x) \left( \frac{1 + w^2}{w}  - w \right) = \frac{ \Phi(x) }{w} \leq 1/w.
\]
When $0 \leq w \leq x$,  it is easy to see that $g_x(w)$ is non-negative and increasing in $w$. Moreover, from \eqref{qfn_bdd} and $x^2 \geq 2x - 1$,
\begin{align*}
g_x(x-1) &=   \sqrt{2\pi} \bar{\Phi}(x) \left(    (2 + x^2 - 2x) e^{x^2/2 - x +1/2} \Phi(x-1) + \frac{x-1}{\sqrt{2 \pi}} \right)\\
&\leq \frac{(2 + x^2 - 2x)}{x} e^{1/2 - x} \Phi(x-1) + \frac{x-1}{x \sqrt{2 \pi}} e^{-x^2/2}\\
&\leq \frac{(4 + 2x^2 - 4x)}{2x} e^{1/2 - x}  + \frac{x-1}{2x} e^{1/2 - x}\\
& \leq \left( x - \frac{3}{2} + \frac{3}{2x}\right)e^{1/2 - x} \leq x e^{1/2 -x}.
\end{align*}
Lastly, by \eqref{qfn_bdd}, it is easy to see that
\begin{align*}
g_x(x) &=  \sqrt{2 \pi}\bar{\Phi}(x) \left( (1 + x^2)  e^{x^2/2} \Phi(x) + \frac{x}{\sqrt{2\pi}}\right)\\
&\leq \frac{1 +x^2}{x} \Phi(x) + \frac{e^{-x^2/2}}{\sqrt{2 \pi}}  \leq \left(\frac{1}{x} + x\right) + \frac{1}{2} \leq x+ 2
\end{align*}
\end{proof}

\begin{lemma}[Bound on expectation of $f_x'(W_b^{(i)} +t)$]\label{lem:expect_f'x}
Let $x \geq 1$, $t \in \bR$, and $W_b^{(i)}$ be as defined in \secref{intro} under the assumptions \eqref{assumptions}. Then there exists an absolute constant $C >0$ such that
\[
\left|\bE[f_x'(W_b^{(i)}+ t)]\right| \leq C (e^{ -x } + e^{- x +t}).
\]
\end{lemma}

\begin{proof}[Proof of \lemref{expect_f'x}]
From
\eqref{fx'bdd} in \lemref{helping}, we have 
\begin{align*}
|\bE[f_x'(W_b^{(i)} + t)]| 
&\leq e^{1/2 -x }  + \bE[I(W_b^{(i)}  + t > x - 1)]  \\
&\leq e^{1/2 -x } +  e^{1 - x +t}\bE[e^{W_b^{(i)} }],
\end{align*}
then apply the Bennett inequality in \lemref{Bennett}.\end{proof}

\section{Exponential randomized concentration inequality for a sum of censored variables} \label{app:RCI}

\begin{lemma}[Exponential randomized concentration inequality  for a sum of  censored random variables] \label{lem:modified_RCI_bdd}
Let $\xi_1, \dots, \xi_n$ be independent random variables  with mean zero and finite second moments, and for each $i =1, \dots, n$, define 
\[
\xi_{b, i} = \xi_i I(|\xi_i| \leq 1) + 1 I(\xi_i >1) - 1  I(\xi_i <-1),
\]
an upper-and-lower censored version of $\xi_i$; moreover, let $W =\sum_{i=1}^n \xi_i $ and $W_b = \sum_{i=1}^n \xi_{b, i}$ be their corresponding sums, and $\Delta_1$ and $\Delta_2$ be two random variables on the same probability space. Assume there exists $c_1 > c_2 >0$ and  $\delta \in(0, 1/2)$ such that 
\[
\sum_{i=1}^n \bE[\xi_i^2] \leq c_1
\]
and 
\[
\sum_{i=1}^n \bE[|\xi_{b, i}| \min(\delta, |\xi_i|/2 )] \geq c_2.
\]
Then for any $\lambda \geq 0$, it is true that
\begin{align*}
&\bE[e^{\lambda W_b} I(\Delta_1 \leq W_b \leq \Delta_2)] \\
&\leq 
\left( \mathbb{E}\left[e^{2\lambda W_b}\right] \right)^{1/2}\exp\left( - \frac{c_2^2}{16 c_1 \delta^2}\right) \\
&+ \frac{2 e^{\lambda \delta}}{c_2} \Biggl\{ \sum_{i =1}^n \mathbb{E} [ |\xi_{b,  i}|e^{\lambda W_b^{(i)} } (|\Delta_1 - \Delta_1^{(i)}| + |\Delta_2 - \Delta_2^{(i)}|)] \\
& + 
\mathbb{E}[|W_b|e^{\lambda W_b }(|\Delta_2 - \Delta_1| + 2 \delta)]\\
& +   \sum_{i=1}^n  \Big|\bE[ \xi_{b, i}]\Big| \bE[e^{\lambda W_b^{(i)} }(|\Delta_2^{(i)} - \Delta_1^{(i)}| + 2 \delta)]\Biggr\} ,
\end{align*}
where $\Delta_1^{(i)}$ and $\Delta_2^{(i)}$ are any random variables on the same probability space such that $\xi_i$ and $(\Delta_1^{(i)}, \Delta_2^{(i)},  W^{(i)}, W_b^{(i)})$ are independent, where $W^{(i)} = W - \xi_i$ and $W_b^{(i)} = W_b - \xi_{b, i}$. 

In particular, by defining $\beta_2 \equiv \sum_{i=1}^n \bE[\xi_i^2 I(|\xi_i|>1)]$ and $\beta_3 \equiv \sum_{i=1}^n \bE[|\xi_i|^3 I(|\xi_i|\leq1)]$, if $\sum_{i=1}^n \bE[\xi_i^2] = 1$ and $\beta_2 + \beta_3 \leq 1/2$, one can take 
\begin{equation}\label{RCI_suggested_params}
\delta = \frac{\beta_2 + \beta_3}{4},  \quad c_1 = 1 \text{ and }  \quad c_2 = \frac{1}{4}
\end{equation}
to satisfy the conditions of the inequality. 
 \end{lemma}

\begin{proof}[Proof of \lemref{modified_RCI_bdd}]
It suffices to show the lemma under the assumptions that 
\begin{equation}\label{expRIassume}
\Delta_1 \leq \Delta_2 \text{ and } \Delta_1^{(i)} \leq \Delta_2^{(i)}.
\end{equation}
If  \eqref{expRIassume}  is not true, we can let $\Delta_1^* = \min(\Delta_1, \Delta_2)$, $\Delta_2^* = \max(\Delta_1, \Delta_2)$, ${\Delta_1^*}^{(i)} = \min(\Delta_1^{(i)}, \Delta_2^{(i)})$, ${\Delta_2^*}^{(i)} = \max(\Delta_1^{(i)}, \Delta_2^{(i)})$. Then the assumptions in \eqref{expRIassume} can be seen to be not forgoing any generality by noting that $|\Delta_2^* - \Delta_1^*| = |\Delta_2 - \Delta_1|$ (also $|{\Delta_2^*}^{(i)} - {\Delta_1^*}^{(i)}| = |\Delta_2^{(i)} - \Delta_1^{(i)}|$),
\[
\mathbb{E}[e^{\lambda W_b} I(\Delta_1 \leq W_b  \leq \Delta_2 )] \leq \mathbb{E}[e^{\lambda W_b} I(\Delta_1^* \leq W_b  \leq \Delta_2^* )]
\]
and 
\begin{equation}\label{liqian_ineq}
|\Delta_1^* - {\Delta_1^*}^{(i)}| + |\Delta_2^* - {\Delta_2^*}^{(i)}|
 \leq |\Delta_1 - \Delta_1^{(i)}| + |\Delta_2 - \Delta_2^{(i)}|,
\end{equation}
where \eqref{liqian_ineq} is true by the following fact: If we have real numbers $x_1 \leq x_2$ and $y_1 \leq y_2$, it must be  that
\begin{equation} \label{rearrangment_ineq}
|x_1 - y_1| + |x_2 - y_2| \leq |x_1 - y_2| + |x_2 - y_1|.
\end{equation}
Without loss of generality, one can assume $x_1 \leq y_1$ and simply prove \eqref{rearrangment_ineq} by case considerations:
\begin{enumerate}
\item If $x_1 \leq x_2 \leq y_1 \leq y_2$, then 
\begin{align*}
|x_1 - y_1| + |x_2 - y_2| &= y_1 - x_1 + y_2 - x_2 \\
&= y_2 - x_1 + y_1 - x_2 
= |x_1 - y_2| + |x_2 - y_1|.
\end{align*}
\item If $x_1 \leq y_1 \leq x_2 \leq y_2$, , then 
\begin{align*}
|x_1 - y_1| + |x_2 - y_2| &= y_1 - x_1 + y_2 - x_2 \\
&\leq  y_2 - x_1  \leq |x_1 - y_2| + |x_2 - y_1|.
\end{align*}
 
\item If $x_1 \leq y_1 \leq y_2 \leq x_2$, , then 
\begin{align*}
|x_1 - y_1| + |x_2 - y_2| &= \underbrace{y_1 - x_1}_{\leq y_2 - x_1} +  \underbrace{x_2 - y_2}_{\leq x_2 - y_1} \leq |x_1 - y_2| + |x_2 - y_1|.
\end{align*}
\end{enumerate}
More generally, a fact like \eqref{rearrangment_ineq} can be proved by the rearrangement inequality \citep[p.78]{Steele}, but the details are omitted here.

Under the working assumptions in \eqref{expRIassume}, for $a< b$, we define the function
\[
 f_{a, b}(w) = 
  \begin{cases} 
   0 & \text{ for } w \leq  a- \delta \\
   e^{\lambda w} (w - a + \delta)       & \text{ for } a -\delta  < w \leq b + \delta \\
  e^{\lambda w} (b - a + 2\delta)       & \text{ for }w  > b + \delta \\
  \end{cases},
\]
which has the property
\begin{multline}\label{fab_bdd}
|f_{a, b}(w) - f_{a_1, b_1} (w)| \leq 
e^{\lambda w} (|a - a_1| + |b - b_1|) 
 \text{ for all } w, \;\ a < b \text{ and }a_1 < b_1,
\end{multline}
as well as 
\[
f_{a,b}'(w) \geq 0 \text{ almost surely. }
\]
Moreover, we have 
\begin{equation} \label{I1_plus_I2}
I_1 + I_2 = \bE[W_b f_{\Delta_1, \Delta_2} (W_b )] - \sum_{i=1}^n \bE[ \xi_{b, i}] \bE[f_{\Delta_1^{(i)}, \Delta_2^{(i)}} ( W_b^{(i)} ) )] 
\end{equation}
where 
\begin{align*}
I_1 &\equiv  \sum_{i=1}^n \mathbb{E}\left[ \xi_{b, i}\left( f_{\Delta_1, \Delta_2}(W_b ) - f_{\Delta_1, \Delta_2} (W_b^{(i)})      \right)\right] \text{ and }\\
I_2 &\equiv \sum_{i=1}^n \mathbb{E} \left[\xi_{b, i} \left(f_{\Delta_1, \Delta_2} ( W_b^{(i)} ) -f_{\Delta_1^{(i)}, \Delta_2^{(i)}} ( W_b^{(i)} ) \right)\right].
\end{align*}
Given the property in \eqref{fab_bdd}, we have 
 \begin{equation} \label{I2_upper_bdd}
 |I_2| \leq \sum_{i =1}^n \mathbb{E} \Big[ |\xi_{b,  i}|e^{\lambda W_b^{(i)} } \Big(|\Delta_1 - \Delta_1^{(i)}| + |\Delta_2 - \Delta_2^{(i)}|\Big)\Big].
 \end{equation}
 Now we estimate $I_1$, by first rewriting it as
 \begin{align*}
I_1&= \sum_{i=1}^n\bE \left[\xi_{b, i} \Big(f_{\Delta_1, \Delta_2}(W_b ) - f_{\Delta_1, \Delta_2}(W_b^{(i)} )  \Big) \right]\\ 
 &= \sum_{i=1}^n \bE\left[\xi_{b, i} \int^0_{-\xi_{b, i} } f_{\Delta_1, \Delta_2}'(W_b + t) dt \right] =  \sum_{i=1}^n \bE\left[   \int_{-\infty}^\infty  f_{\Delta_1, \Delta_2}'(W_b + t) \hat{K}_i(t) dt \right],
 \end{align*}
 where 
 \[
 \hat{K}_i(t) \equiv \xi_{b, i}\{I(- \xi_{b, i}  \leq t \leq 0) - I(0 < t \leq - \xi_{b, i} )\}.
 \]
Note that $\xi_{b, i}$ and $I(-\xi_{b, i}  \leq t \leq 0) - I(0 < t \leq - \xi_{b, i} )$ have the same sign, and it is also true that 
 $0 \leq \tilde{K}_i(t) \leq \hat{K}_i(t)$ where
 \[
 \tilde{K}_i(t) = \xi_{b, i} \{I(- \xi_{b, i}/2 \leq t \leq 0) - I(0 < t \leq - \xi_{b, i} /2)\}
 \]
By the fact that $f_{\Delta_1, \Delta_2}'(w) \geq  e^{\lambda w}\geq 0$ for all $w \in (\Delta_1 - \delta, \Delta_2 + \delta)$, one can  lower bound $I_1$ as
 \begin{align*}
 I_1 &\geq \sum_{i=1}^n \mathbb{E}\left[ \int_{-\infty}^\infty  f_{\Delta_1, \Delta_2}'(W_b  + t) \tilde{K}_i(t) dt \right]\\
 &\geq \sum_{i=1}^n \mathbb{E}\left[ \int_{|t|\leq \delta}I(\Delta_1 \leq W_b  \leq \Delta_2) f_{\Delta_1, \Delta_2}'(W_b + t) \tilde{K}_i(t) dt \right]\\
 &\geq \sum_{i=1}^n \bE\left[ I(\Delta_1 \leq W_b  \leq \Delta_2) e^{\lambda (W_b  - \delta)} |\xi_{b, i}| \min(\delta, |\xi_{b, i}|/2)\right]\\
 &= \bE \left[I(\Delta_1 \leq W_b  \leq \Delta_2) e^{\lambda (W_b  - \delta)} \left(\sum_{i=1}^n \eta_i\right)\right],
 \end{align*}
 where 
 \[
 \eta_i \equiv |\xi_{b, i}| \min(\delta, |\xi_i|/2),
 \]
  noting that given $\delta < 1/2$, $\min(\delta, |\xi_i|/2) = \min(\delta, |\xi_{b, i}|/2)$ due to the censoring definition of $\xi_{b, i}$. 
 Hence, continuing, 
 we can further lower bound $I_1$ as
 \begin{align}
 I_1 &\geq (c_2/2) \mathbb{E}\left[e^{\lambda (W_b  - \delta)} I(\Delta_1 \leq W_b\leq \Delta_2) I\left(\sum_{i=1}^n \eta_i \geq c_2/2\right)\right]\notag\\
 &\geq \frac{c_2}{2 e^{\lambda \delta}} \left \{  \mathbb{E}\left[ e^{\lambda W_b }  I(\Delta_1 \leq W_b  \leq \Delta_2) \right]- \mathbb{E}\left[e^{\lambda W_b } I\left(\sum_{i=1}^n \eta_i < c_2/2\right)\right]\right\}\notag\\
  &\geq \frac{c_2}{2 e^{\lambda \delta}}\left \{  \mathbb{E}\left[ e^{\lambda W_b }  I(\Delta_1 \leq W_b  \leq \Delta_2)\right] - \sqrt{\mathbb{E}\left[e^{2\lambda W_b }\right] P\left(\sum_{i=1}^n \eta_i < c_2/2\right)}\right\}\notag\\
  &\geq  \frac{c_2}{2 e^{\lambda \delta}}\left \{  \mathbb{E}[ e^{\lambda W_b  }  I(\Delta_1 \leq W_b  \leq \Delta_2)] -\left( \mathbb{E}\left[e^{2\lambda W_b}\right] \right)^{1/2}\exp\left( - \frac{c_2^2}{16 c_1 \delta^2}\right)\right\} \label{embeded},
 \end{align}
 where the last inequality comes from
the sub-Gaussian lower tail bound for sum of non-negative random variables \citep[Theorem 2.19]{victor2009self}, 
 \[
 P\left(\sum_{i=1}^n \eta_i < c_2/2\right) \leq \exp\left(- \frac{(c_2/2)^2}{2 \sum_{i=1}^n \mathbb{E}[\eta_i^2]}\right) \leq  \exp\left( - \frac{c_2^2}{8 c_1 \delta^2}\right).
 \]
 Clearly, since $|f_{\Delta_1, \Delta_2}(w)| \leq e^{\lambda w} (\Delta_2 - \Delta_1 + 2 \delta)$,
 we have, from \eqref{I1_plus_I2}, 
 \begin{equation} \label{I1_plus_I2_lower_bdd}
 I_1 + I_2 
  \leq \mathbb{E}[|W_b|e^{\lambda W_b }(|\Delta_2 - \Delta_1| + 2 \delta)] + 
    \sum_{i=1}^n \Big|\bE[ \xi_{b, i}]\Big| \bE[e^{\lambda W_b^{(i)} }(|\Delta_2^{(i)} - \Delta_1^{(i)}| + 2 \delta)]
 \end{equation}
 Combining \eqref{I2_upper_bdd}, \eqref{embeded} and \eqref{I1_plus_I2_lower_bdd}, the proof is done.  
 
 If  $\sum_{i=1}^n \bE[\xi_i^2] = 1$ and $\beta_2 + \beta_3 \leq 1/2$, one can apparently take $c_1 = 1$. The parameter choices of $c_2$ and $\delta$ in \eqref{RCI_suggested_params} can be justified as follows: Using the fact that \citep[p.259]{chen2011normal}
\[
\min(x, y) \geq y - \frac{y^2}{4 x}\text{ for } x > 0 \text{ and } y \geq 0,
\]
by taking $\delta = (\beta_2 + \beta_3)/4$, we have 
\begin{align*}
\sum_{i=1}^n \bE[|\xi_{b, i}| \min(\delta, |\xi_i|/2)] &\geq \sum_{i=1}^n \bE[ |\xi_i | I(|\xi_i| \leq 1)\min(\delta, |\xi_i|/2)]\\
&\geq \sum_{i=1}^n \Bigg[ \frac{\bE[\xi_i^2 I(|\xi_i| \leq 1)]}{2} - \frac{\bE[|\xi_i |^3 I(|\xi_i| \leq 1)]}{16 \delta} \Bigg] = \frac{1 - \beta_2}{2} - \frac{\beta_3}{16 \delta}\\
&=   \frac{1}{2} - \frac{ 8 \delta\beta_2+ \beta_3}{16 \delta} \underbrace{\geq}_{\delta \leq 1/8} \frac{1}{2} - \frac{\beta_2 + \beta_3}{16 \delta} =  \frac{1}{4}.
\end{align*}
\end{proof}

\newpage
\section{Proof of \thmref{main}} \label{app:main_pf}

This section presents the proof of \thmref{main}. The approach is similar to that of \citet[Theorem 3.1]{shao2016stein} but there are quite a number of differences stemming from correcting the numerous gaps in the latter. It suffices to consider $x \geq 0$, or else we can consider $- T_{SN}$ instead\footnote{For a given $x < 0$, if one can uniformly bound $|P(T_{SN} < x + \epsilon) - \Phi(x+\epsilon)|$ for all $\epsilon \in (x, 0)$,  one can then similarly bound $|P(T_{SN} \leq x) - \Phi(x)|$ by taking limits on both sides as $\epsilon \longrightarrow 0$.}. Moreover, without loss of generality, we can assume
 \begin{equation} \label{beta2_plus_beta3_leq_half}
 \beta_2 + \beta_3 < 1/2,
 \end{equation}
 otherwise it must be true that $|P(T_{SN} \leq x) - \Phi(x)| \leq 2 (\beta_2 + \beta_3)$. Since 
\[
1 + s/2 - s^2/2 \leq (1 + s)^{1/2} \leq 1 + s/2 \text{ for all } s  \geq -1,
\]
we have the two inclusions
\[
\{T_{SN} >x\} \subset \{W_n + D_{1n} - x D_{2n}/2 > x\} \cup \{  x + x(D_{2n} - D_{2n}^2)/2 < W_n + D_{1n} \leq x + x D_{2n}/2 \}
\]
and 
\[
\{T_{SN} > x\} \supset \{W_n + D_{1n} - x D_{2n}/2 > x\}.
\]
Hence, it suffices to establish the bounds 
\begin{multline}\label{part1}
P(x + x(D_{2n} - D_{2n}^2)/2 \leq W_n + D_{1n} \leq x + x D_{2n}/2 ) \leq \sum_{j=1}^2 P(|D_{jn}| > 1/2)\\
 + C\Bigg\{\beta_2 + \beta_3  + \bE\Big[(1 + e^{W_b}) \bar{D}_{2n}^2 \Big] 
+ \sum_{j=1}^2  \sum_{i=1}^n\|  \xi_{b, i}e^{W_b^{(i)}/2}  ( \bar{D}_{jn} - \bar{D}_{jn}^{(i)})   \|_1 
\Bigg\} 
\end{multline}
and
\begin{multline} \label{part2}
|P(W_n + D_{1n} - x D_{2n}/2 \leq x) - \Phi(x)| \leq  \sum_{j=1}^2P(|D_{jn}| > 1/2)  \\
+ C \Bigg\{\beta_2  + \beta_3  +  \|\bar{D}_{1n}\|_2 +  \bE\Big[(1 +e^{W_b}) \bar{D}_{2n}^2\Big] +  \Big|x \bE[\bar{D}_{2n} f_x(W_b)]\Big| \\
+  \sum_{j=1}^2
\sum_{i=1}^n \bigg( \bE[\xi^2_{b, i}] \Big\| (1+ e^{W_b^{(i)}})(\bar{D}_{jn} - \bar{D}_{jn}^{(i)} )\Big\|_1 +  \Big\|  \xi_{b, i}( 1+ e^{W_b^{(i)}/2})  ( \bar{D}_{jn} - \bar{D}_{jn}^{(i)})   \Big\|_1\bigg)
  \Bigg\}   
\end{multline}
separately. Before starting to prove them, we  introduce the following  notation:
\begin{equation*}\label{Delta1_and_Delta2_def}
\bar{\Delta}_{1n,x} =  \frac{x(\bar{D}_{2n} - \bar{D}_{2n}^2)}{2} - \bar{D}_{1n}  \text{ and }
\bar{\Delta}_{2n, x} = \frac{x\bar{D}_{2n}}{2} - \bar{D}_{1n}.
\end{equation*}

\subsection{Proof of \eqref{part1}} 
We further introduce
\[
\bar{\Delta}_{1n,x}^{(i)} =  \frac{x(\bar{D}_{2n}^{(i)} - (\bar{D}_{2n}^{(i)})^2)}{2} - \bar{D}_{1n}^{(i)} \text{ and }
\bar{\Delta}_{2n, x}^{(i)} = \frac{x\bar{D}_{2n}^{(i)}}{2} - \bar{D}_{1n}^{(i)}.
\]
%
Noting that 
\begin{equation} \label{split_out_P0}
P\Big(x + x(D_{2n} - D_{2n}^2)/2 \leq W_n + D_{1n} \leq x + x D_{2n}/2 \Big) \leq 
P_0 + \sum_{j=1}^2P(|D_{jn}|> 1/2) + \beta_2,
\end{equation}
where 
\[
P_0 = P(x + \bar{\Delta}_{1n,x}  \leq W_b \leq x+ \bar{\Delta}_{2n, x} ),
\]
it suffices to bound $P_0$.  Since $\bar{D}_{2n} - \bar{D}_{2n}^2 \geq - 3/4$ and hence $
\frac{1}{2} (x + \bar{\Delta}_{1n,x}) \geq \frac{1}{2}( \frac{5x}{8} - \frac{1}{2}) > \frac{x}{4} - \frac{1}{4}
$,  in light of \eqref{beta2_plus_beta3_leq_half}, applying  \lemref{modified_RCI_bdd} with the parameters in \eqref{RCI_suggested_params} and $\lambda = 1/2$
 implies that
\begin{align}
e^{x/4- 1/4} P_0 &\leq \bE[e^{W_b/2} I(x + \bar{\Delta}_{1n,x} \leq W_b \leq x + \bar{\Delta}_{2n,x})] \notag\\
&\leq 
\left( \mathbb{E}\left[e^{ W_b}\right] \right)^{1/2}\exp\left( - \frac{1}{16(\beta_2 + \beta_3)^2}\right) \notag\\
&+ 8 e^{(\beta_2 + \beta_3)/8}\Biggl\{ \sum_{i =1}^n \mathbb{E} \Big[ |\xi_{b,  i}|e^{W_b^{(i)} /2} \Big(|\bar{\Delta}_{1n,x}- \bar{\Delta}_{1n,x}^{(i)}| + |\bar{\Delta}_{2n,x}-\bar{\Delta}_{2n,x}^{(i)}|\Big)\Big] \notag\\
& + 
\mathbb{E}\left[|W_b|e^{W_b/2 }\left(|\bar{\Delta}_{2n,x}- \bar{\Delta}_{1n,x}| + \frac{\beta_2 + \beta_3}{2}\right) \right]\notag\\
& +   \sum_{i=1}^n  \Big|\bE[ \xi_{b, i}]\Big| \bE\left[e^{ W_b^{(i)} /2 }\left(|\bar{\Delta}_{2n,x}^{(i)} - \bar{\Delta}_{1n,x}^{(i)}| + \frac{\beta_2 + \beta_3}{2}\right)\right]\Biggr\} \label{1st_applied_RCI}
\end{align}
We will bound the different terms on the right hand side of \eqref{1st_applied_RCI}. First,
\begin{equation} \label{bennett_exp_ewb}
\bE[e^{W_b} ] \leq \exp( e^2/4 + 1/4) \text{ by \lemref{Bennett}}
\end{equation}
and
\begin{align}
\exp\left(\frac{-1}{16 (\beta_2 + \beta_3)^2} \right) 
&\leq C (\beta_2 + \beta_3).
\end{align}
Since $\bar{D}_{2n}^2 - (\bar{D}_{2n}^{(i)})^2 = (\bar{D}_{2n} - \bar{D}_{2n}^{(i)})(\bar{D}_{2n} + \bar{D}_{2n}^{(i)})$,
\begin{align}
& \mathbb{E} [ |\xi_{b,  i}|e^{W_b^{(i)} /2} (|\bar{\Delta}_{1n,x}- \bar{\Delta}_{1n,x}^{(i)}| + |\bar{\Delta}_{2n,x}-\bar{\Delta}_{2n,x}^{(i)}|)] \notag \\
&\leq C \mathbb{E} [  |\xi_{b,  i}|e^{W_b^{(i)} /2} (|\bar{D}_{1n} -\bar{D}_{1n}^{(i)} | + x|\bar{D}_{2n} -\bar{D}_{2n}^{(i)}| )  ] .\label{bdd_sth_else1}
\end{align}
 Moreover, since 
$
\frac{|W_b|}{2} \leq e^{|W_b|/2} \leq e^{W_b/2} + e^{-W_b/2}
$,
by \lemref{Bennett},
\begin{equation}
\mathbb{E}\left[|W_b|e^{W_b/2 }\left(|\bar{\Delta}_{2n,x}- \bar{\Delta}_{1n,x}| + \frac{\beta_2 + \beta_3}{2}\right) \right]
\leq C_1  x   \bE[(1 + e^{W_b}) \bar{D}_{2n}^2 ]+ C_2 (\beta_2 + \beta_3).  \label{bdd_sth_else2}
\end{equation}
Lastly, by \lemref{exp_xi_bi_bdd}, Bennett's inequality (\lemref{Bennett}) and \eqref{beta2_plus_beta3_leq_half}, we have
\begin{align}
 &\sum_{i=1}^n\Big|\bE[ \xi_{b, i}]\Big| \bE\left[e^{ W_b^{(i)} /2 }\left(|\bar{\Delta}_{2n,x}^{(i)} - \bar{\Delta}_{1n,x}^{(i)}| + \frac{\beta_2 + \beta_3}{2}\right)\right] \notag\\
 &\leq C \sum_{i=1}^n \Big|\bE[ \xi_{b, i}]\Big|  \underbrace{\Bigg( x\bE[ e^{W_b^{(i)}/2} (\bar{D}_{2n}^{(i)})^2   ] + \beta_2 + \beta_3\Bigg)}_{\leq \quad C(1 +x)} \leq C (1+x)\beta_2. \label{bdd_sth_else3}
\end{align}
Collecting   \eqref{split_out_P0}- \eqref{bdd_sth_else3}, we get \eqref{part1}.

\subsection{Proof of \eqref{part2}} 
For this part, as a proof device, we let
$X_1^*, \dots, X_n^*$ be independent copies of $X_1, \dots, X_n$  and in analogy to \eqref{D1D2_as_fn_of_data}, we introduce
\begin{multline*}
D_{1n, i^*}= D_{1n}(X_1, \dots, X_{i-1}, X_i^*, X_{i+1}, \dots, X_n) \text{ and }\\
 D_{2n, i^*} = D_{2n}(X_1, \dots, X_{i-1}, X_i^*, X_{i+1}, \dots, X_n), 
\end{multline*}
\begin{multline*}
\bar{D}_{1n, i^*} = D_{1n, i^*}I\Bigg(|D_{1n, i^*}| \leq \frac{1}{2} \Bigg) + \frac{1}{2}I\Bigg(D_{1n, i^*} > \frac{1}{2}\Bigg) - \frac{1}{2} I\Bigg(D_{1n, i^*} <- \frac{1}{2}\Bigg)\text{ and }\\
 \bar{D}_{2n, i^*} =  D_{2n, i^*} I\Bigg(| D_{2n, i^*} | \leq \frac{1}{2} \Bigg) + \frac{1}{2}I\Bigg( D_{2n, i^*}  > \frac{1}{2}\Bigg) - \frac{1}{2} I\Bigg( D_{2n, i^*}  <- \frac{1}{2}\Bigg), 
\end{multline*}
as well as
\[
\bar{\Delta}_{2n, x, i^*} = \frac{x \bar{D}_{2n, i^*} }{2} -  \bar{D}_{1n, i^*},
\]
which are correspondingly versions of $D_{1n}$, $D_{2n}$, $\bar{D}_{1n}$, $\bar{D}_{2n}$ and $\bar{\Delta}_{2n,x}$   with $X_i^*$ replacing $X_i$ as input. For any pair $1 \leq i , i' \leq n$ and $j \in \{1, 2\}$, we also define
\[
D_{jn, i^*}^{(i')} \equiv   \begin{cases} 
   D^{(i')} ( X_1, \dots, X_{i - 1}, X_i^*, X_{i+1}, \dots, X_{i'-1}, X_{i'+1}, \dots, X_n ) & \text{if } i < i' \\
  D^{(i')} ( X_1, \dots, X_{i'-1}, X_{i'+1},  \dots, X_{i - 1}, X_i^*, X_{i+1}, \dots,  X_n )      & \text{if } i > i' \\
    D^{(i')} ( X_1,  \dots, X_{i - 1},  X_{i+1}, \dots,  X_n )      & \text{if } i = i' \\
  \end{cases}, 
\]
i.e., $D_{jn, i^*}^{(i')}$ is a version of the ``leave-one-out" $D_{jn}^{(i')}$ with $X_i^*$ replacing $X_i$ as input, and its censored version
\[
\bar{D}_{jn, i^*}^{(i')}  \equiv D_{jn, i^*}^{(i')} I\Bigg(|D_{jn, i^*}^{(i')}| \leq \frac{1}{2} \Bigg) + \frac{1}{2}I\Bigg(D_{jn, i^*}^{(i')}> \frac{1}{2}\Bigg) - \frac{1}{2} I\Bigg(D_{jn, i^*}^{(i')} <- \frac{1}{2}\Bigg).
\]

It suffices to bound $|P(W_b -  \bar{\Delta}_{2n, x}  \leq x) - \Phi(x)|$ since 
\begin{equation} \label{part2_break_down}
|P(W_n -  \Delta_{2n, x}  \leq x) - \Phi(x)| \leq |P(W_b - \bar{\Delta}_{2n, x}  \leq x) - \Phi(x)| + 
\beta_2 + \sum_{j=1}^2P(|D_{jn}| > 1/2).
\end{equation}
First, define the \emph{K function}
\[
k_{b, i} (t) = \bE[\xi_{b, i}\{ I(0 \leq t \leq \xi_{b, i}) - I(\xi_{b, i} \leq t < 0)\}],
\]
which has the properties
\begin{multline} \label{kb_properties}
\int_{-\infty}^\infty k_{b, i}(t) dt =\int_{-1}^1 k_{b, i}(t) dt = \bE[\xi_{b, i}^2]  = \|\xi_{b, i}\|_2^2 \quad \text{ and }\\
\int_{-\infty}^\infty |t| k_{b, i}(t) dt = \int_{-1}^1 |t| k_{b, i}(t) dt = \frac{\bE[|\xi_{b, i}|^3]}{2} = \frac{\|\xi_{b, i}\|_3^3}{2}.
\end{multline}
Since
\[
\bE\Big[\int_{-1}^1 f_x'( W_b^{(i)} - \bar{\Delta}_{2n, x, i^*} + t) k_{b, i}(t) dt\Big] = \bE\Big[\xi_{b, i} \{f_x(W_b - \bar{\Delta}_{2n, x, i^*} ) - f_x(W_b^{(i)} - \bar{\Delta}_{2n, x, i^*}  ) \}\Big]
\]
by independence and the fundamental theorem of calculus, from the Stein equation \eqref{steineqt}, one can then write
\begin{align*}
&P(W_b -   \bar{\Delta}_{2n, x}  \leq x) - \Phi(x) \\
&= \bE[ f_x'(W_b -   \bar{\Delta}_{2n, x} )] - \bE[W_b f_x(W_b -   \bar{\Delta}_{2n, x})] \\
& \hspace{2cm}+ \bE\Big[ \bar{\Delta}_{2n, x}  \Big(f_x(W_b -  \bar{\Delta}_{2n, x}) - f_x(W_b)\Big)\Big] + \bE[\bar{\Delta}_{2n, x} f_x(W_b) ]\\
&= \underbrace{\sum_{i=1}^n \bE\Big[\int_{-1}^1 \{f_x'(W_b -  \bar{\Delta}_{2n, x}) - f_x'(W_b^{(i)} -  \bar{\Delta}_{2n, x, i^*}  + t)\} k_{b, i}(t) dt \Big] }_{R_1}\\
& + \underbrace{\sum_{i=1}^n \bE[ (\xi_i^2- 1)I(|\xi_i|> 1)]  \bE[f_x'(W_b -  \bar{\Delta}_{2n,x})] 
- \sum_{i=1}^n\bE[\xi_{b, i} f_x(W_b^{(i)} - \bar{\Delta}_{2n, x, i^*} )]
 + \bE[ \bar{\Delta}_{2n, x} f_x (W_b)] }_{R_2}\\
 &+ \underbrace{\Bigg\{ - \sum_{i=1}^n \bE\Bigg[\xi_{b, i}  \Big\{ f_x(W_b -  \bar{\Delta}_{2n, x}) - f_x(W_b - \bar{\Delta}_{2n, x, i^*} ) \Big\}\Bigg]\Bigg\}}_{R_3}\\
 &+  \underbrace{\bE\Big[  \barDelta_{2n, x} \int_0^{-  \bar{\Delta}_{2n, x}}f_x'(W_b + t) dt \Big]}_{R_4}\\
 &= R_1 + R_2 + R_3 + R_4.
\end{align*}
To finish the proof, we will establish the following  bounds for $R_1, R_2, R_3, R_4$:
\begin{multline} \label{R1_final_bdd}
|R_1| \leq \\
C \Bigg\{ \beta_2 + \beta_3 +
  \sum_{j=1}^2
\sum_{i=1}^n \bigg( \bE[\xi^2_{b, i}] \Big\| (1+ e^{W_b^{(i)}})(\bar{D}_{jn} - \bar{D}_{jn}^{(i)} )\Big\|_1 +  \Big\|  \xi_{b, i}e^{W_b^{(i)}/2}  ( \bar{D}_{jn} - \bar{D}_{jn}^{(i)})   \Big\|_1\bigg)  \Bigg\}
\end{multline}
\begin{equation}\label{R2_final_bdd}
|R_2| \leq 1.63\beta_2 +   0.63 \|\bar{D}_{1n}\|_2 + \Big|\frac{x}{2} \bE[\bar{D}_{2n} f_x(W_b)]\Big|
,
\end{equation}
\begin{equation} \label{R3_final_bdd}
|R_3|  \leq  C \sum_{j=1}^2 \sum_{i=1}^n \| \xi_{b, i} (1 + e^{W_b^{(i)}/2})(\bar{D}_{jn} - \bar{D}_{jn}^{(i)})\|_1,
\end{equation}
\begin{equation} \label{R4_final_bdd}
|R_4| \leq   C\Big( \|\bar{D}_{1n}\|_2 +  \bE[( 1+e^{W_b}) \bar{D}_{2n}^2]\Big).
\end{equation}
Then \eqref{R1_final_bdd} - \eqref{R4_final_bdd} together with \eqref{part2_break_down}   conclude \eqref{part2}.

\subsubsection{Bound for $R_1$} 
Let $g_x(w) = (w f_x(w))'$ as defined in \eqref{gw_def}. By the Stein equation \eqref{steineqt} and defining $\eta_1 =  t -  \bar{\Delta}_{2n, x, i^*}$  and $\eta_2 = \xi_{b, i} -  \bar{\Delta}_{2n, x}$, we can write
\[
R_1 = R_{11} + R_{12},
\]
where
\begin{align*}
R_{11}  
&=  \sum_{i=1}^n  \int_{-1}^1 \bE \Big[ \int_{t - \bar{\Delta}_{2n, x, i^*}}^{\xi_{b, i} -  \bar{\Delta}_{2n, x}} g_x(W_b^{(i)} + u) du\Big] k_{b, i} (t) dt\\
&= \underbrace{\sum_{i=1}^n \int_{-1}^1 \bE\Bigg[ \int g_x(W_b^{(i)} + u) I( \eta_1 \leq u \leq \eta_2) du\Bigg] k_{b, i} (t) dt}_{R_{11.1}} \\
&\hspace{1cm} - \underbrace{\sum_{i=1}^n \int_{-1}^1 \bE\Bigg[ \int g_x(W_b^{(i)} + u) I(\eta_2\leq u \leq \eta_1) du\Bigg] k_{b, i} (t) dt}_{R_{11.2}}
\end{align*}
and
\begin{align*}
R_{12}  &= \sum_{i=1}^n \int_{-1}^1   
\{P(W_b -  \bar{\Delta}_{2n, x} \leq x) - P(W_b^{(i)} - \bar{\Delta}_{2n, x, i^*} + t \leq x)\} k_{b, i}(t) dt.
\end{align*}
For $0 \leq x <1$, since $|g_x| \leq 2.3$ (\lemref{helping_unif}), using the properties in \eqref{kb_properties}, we have
\begin{align}
|R_{11}| &\leq  C \sum_{i=1}^n \int_{-1}^1 \Big(|t| + \|\xi_{b, i}\|_1 +\sum_{j=1}^2 \|\bar{D}_{jn} - \bar{D}_{jn, i^*}\|_1\Big) k_{b, i}(t) dt  \notag\\
&\leq C\left( \sum_{i=1}^n \|\xi_{b, i}\|_3^3 + \sum_{i=1}^n \|\xi_{b, i}\|_2^2 \|\xi_{b, i}\|_1 + \sum_{j=1}^2 \sum_{i=1}^n  \|\xi_{b, i}\|_2^2 \|\bar{D}_{jn} - \bar{D}_{jn, i^*}\|_1\right)\notag\\
&\leq C\left(\beta_2 + \beta_3 + \sum_{j=1}^2 \sum_{i=1}^n  \|\xi_{b, i}\|_2^2 \|\bar{D}_{jn} - \bar{D}_{jn, i^*}\|_1\right) \text{ for } 0 \leq x < 1\label{R_11_bdd_xleq1},
\end{align}
where we have used  $\|\xi_{b,i}\|_1 \leq \|\xi_{b,i}\|_2 \leq \|\xi_{b,i}\|_3$ and
\begin{align}
\|\xi_{b, i}\|_3^3 &=  \bE[|\xi_i|^3 I(|\xi_i| \leq 1)] +  \bE[I(|\xi_i|>1)] \notag\\
&\leq \bE[|\xi_i|^3 I(|\xi_i| \leq 1)]  + \bE[\xi_i^2 I(|\xi_i|>1)]  \label{split_out_beta2}
\end{align}
in the last inequality.

For $x \geq 1$, we first bound the integrand of $R_{11.1}$.
Using the identity
\begin{align*}
1 &= I(W_b^{(i)} + u \leq x - 1) + I(x -1 < W_b^{(i)} + u, u \leq 3x/4 ) + I(x -1 < W_b^{(i)} + u, u > 3x/4)\\
&\leq I(W_b^{(i)} + u \leq x - 1) + I(x -1 < W_b^{(i)} + u, W_b^{(i)} +1> x/4) + (x -1 < W_b^{(i)} + u, u > 3x/4)
\end{align*}
and the bounds for $g_x(\cdot)$ in \lemref{helping}, 
in light of $| \bar{\Delta}_{2n, x}| \leq \frac{x |\bar{D}_{2n}|}{2} + |\bar{D}_{1n}| \leq \frac{1}{2} + \frac{x}{4}$ and
$1.6 \barPhi(x) \leq x e^{1/2-x}$,
\begin{align*}
&\Bigg|\bE\Big[\int g_x(W_b^{(i)} + u) I(\eta_1 \leq  u \leq \eta_2) du\Big]\Bigg|\\
&\leq  x e^{1/2 - x} \|\eta_2 - \eta_1\|_1 + (x + 2) \Big\{ \| I(W_b^{(i)} +1 > x/4) (\eta_2 - \eta_1) \|_1 +  \| I(\eta_2 > 3x/4)  (\eta_2 - \eta_1)\|_1 \Big\}\\
&\leq  x e^{1/2 - x} \|\eta_2 - \eta_1\|_1 + \frac{x+2}{e^{x/4 -1}} \|e^{W_b^{(i)}}(\eta_2 - \eta_1)\|_1 +
\frac{x+2}{e^{3x/4}} \|e^{\xi_{b, i} - \barDelta_{2n, x}} (\eta_2  - \eta_1) \|_1\\
&\leq  \Bigg(x e^{1/2 - x} + \frac{e^{3/2} (x+2)}{e^{x/2}} \Bigg)\|\eta_2 - \eta_1\|_1 + \frac{x+2}{e^{x/4 -1}} \|e^{W_b^{(i)}}(\eta_2 - \eta_1)\|_1 \\
&\leq \frac{C(x+2)}{e^{x/4}}\Bigg\{ |t| + \|\barDelta_{2n, x, i^*} - \barDelta_{2n, x} + \xi_{b, i}\|_1 +  \| e^{W_b^{(i)}}  (\barDelta_{2n, x, i^*} - \barDelta_{2n, x} + \xi_{b, i})\|_1 \Bigg\}\\
\end{align*}
where we have used the Bennett's inequality (\lemref{Bennett}) via $\|e^{W_b^{(i)}} t\|_1 \leq C |t|$. Continuing, 
\begin{align}
&\Bigg|\bE\Big[\int g_x(W_b^{(i)} + u) I(\eta_1 \leq  u \leq \eta_2) du\Big]\Bigg|\notag\\
&\leq \frac{C(x+2)}{e^{x/4}}\Bigg\{ |t| + \| x (\bar{D}_{2n, i^*}- \bar{D}_{2n}) - (\bar{D}_{1n, i^*}- \bar{D}_{1n}) + \xi_{b, i}\|_1\notag\\
&\hspace{4cm} +  \| e^{W_b^{(i)}}  [ x (\bar{D}_{2n, i^*}- \bar{D}_{2n}) - (\bar{D}_{1n, i^*}- \bar{D}_{1n})+ \xi_{b, i}]\|_1 \Bigg\}\notag\\
&\leq C \Bigg\{|t| +  (1 + \|e^{W_b^{(i)}}\|_2)\|\xi_{b, i}\|_2   +\sum_{j=1}^2 \| (1 + e^{W_b^{(i)}}) (\bar{D}_{jn, i^*}- \bar{D}_{jn})\|_1 \Bigg\} \notag\\
& \leq C \Bigg\{|t| + \|\xi_{b, i}\|_2   +\sum_{j=1}^2 \| (1 + e^{W_b^{(i)}}) (\bar{D}_{jn, i^*}- \bar{D}_{jn})\|_1 \Bigg\}\label{integrand_R11_1_bdd},
\end{align}
where the last inequality uses Bennett's inequality  (\lemref{Bennett} giving $\|e^{W_b^{(i)}}\|_2  \leq C$). 
By a completely analogous argument, we also have the bound
\begin{equation}
\Bigg|\bE\Big[\int g_x(W_b^{(i)} + u) I(\eta_2 \leq  u \leq \eta_1) du\Big]\Bigg|
 \leq C \Bigg\{|t| + \|\xi_{b, i}\|_2   +\sum_{j=1}^2 \| (1 + e^{W_b^{(i)}}) (\bar{D}_{jn, i^*}- \bar{D}_{jn})\|_1 \Bigg\} \label{integrand_R11_2_bdd}.
\end{equation}
for the integrand of $R_{11.2}$, for $x \geq 1$. 
Combining \eqref{integrand_R11_1_bdd} and \eqref{integrand_R11_2_bdd}, as well as the integral and moment properties in \eqref{kb_properties} and \eqref{split_out_beta2},  via integrating over $t$, we have
\begin{align}
|R_{11}| &\leq C  \Bigg\{\beta_2 + \beta_3 +  \sum_{i=1}^n \|\xi_{b, i}\|_2^2 \Bigg( \|\xi_{b, i}\|_2   +\sum_{j=1}^2 \| (1 + e^{W_b^{(i)}}) (\bar{D}_{jn, i^*}- \bar{D}_{jn})\|_1 \Bigg)  \Bigg\} \notag \\  
&\leq C \Bigg\{\beta_2 + \beta_3 +   \sum_{j=1}^2\sum_{i=1}^n \|\xi_{b, i}\|_2^2 \Big\| (1 + e^{W_b^{(i)}}) (\bar{D}_{jn, i^*}- \bar{D}_{jn})\Big\|_1\Bigg\} \text{ for } x \geq 1 \label{R_11_bdd_xgeq1}, 
\end{align}
where the last inequality also uses  $\|\xi_{b, i}\|_2^3 \leq \|\xi_{b, i}\|^3_3$ and  \eqref{split_out_beta2}. Combining \eqref{R_11_bdd_xgeq1} with the bound for $x \in [0, 1)$ in \eqref{R_11_bdd_xleq1}, we get, for all $x \geq 0$,
\begin{align}
|R_{11}| &\leq  C \Bigg\{\beta_2 + \beta_3 +   \sum_{j=1}^2\sum_{i=1}^n \bE[\xi_{b, i}^2] \big\| (1 + e^{W_b^{(i)}}) (\bar{D}_{jn} - \bar{D}_{jn, i^*} )\big\|_1\Bigg\}  \notag\\
&=  C \Bigg\{\beta_2 + \beta_3 +   \sum_{j=1}^2\sum_{i=1}^n \bE[\xi_{b, i}^2] \big\| (1 + e^{W_b^{(i)}}) (\bar{D}_{jn} - \bar{D}_{jn}^{(i)} +  \bar{D}_{jn}^{(i)} -  \bar{D}_{jn, i^*} )\big\|_1\Bigg\} \notag \\
&\leq  C \Bigg\{\beta_2 + \beta_3 +   \sum_{j=1}^2\sum_{i=1}^n \bE[\xi_{b, i}^2] \big\| (1 + e^{W_b^{(i)}}) (\bar{D}_{jn} - \bar{D}_{jn}^{(i)}  )\big\|_1\Bigg\} \label{R_11_bdd}
\end{align}
where in the last inequality, we have used the fact that $(W_b^{(i)}, \bar{D}_{jn} - \bar{D}_{jn}^{(i)}) =_d (W_b^{(i)}, \bar{D}_{jn, i^*} - \bar{D}_{jn}^{(i)})$ .

For $R_{12}$, its integrand for a given $i$ is bounded by
\begin{multline} \label{two_probs}
 P(x + \barDelta_{2n, x} \leq W_b \leq x - t + \barDelta_{2n, x, i^*} + \xi_{b, i})+ P( x - t+ \barDelta_{2n, x, i^*} +\xi_{b, i} \leq  W_b \leq  x + \barDelta_{2n, x})
\end{multline}
Since
\[
(x + \barDelta_{2n, x}) \wedge (x - t + \barDelta_{2n, x, i^*}+\xi_{b, i}) \geq  (3x)/4- 5/2  \quad \text{ for } \quad  |t| \leq 1,
\]
and $\bE[e^{W_b}] \leq C$ by Bennett's inequality (\lemref{Bennett}),
by defining  
\[
\bar{\Delta}_{2n, x, i^*}^{(i')} \equiv  \frac{x \bar{D}_{2n, i^*}^{(i')} }{2} -  \bar{D}_{1n, i^*}^{(i')} \text{ for }
1 \leq i' \leq n,
\]
 we can apply the randomized concentration inequality (\lemref{modified_RCI_bdd})  with the parameters in \eqref{RCI_suggested_params} and $\lambda = 1/2$  to bound \eqref{two_probs} by
\begin{align}
&C e^{-3x/8} \Bigg\{\beta_2 + \beta_3 \notag\\
& \hspace{1cm} + \sum_{i' = 1}^n \bE\Big[ |\xi_{b, i'}|e^{W_b^{(i')}/2} \Big( |\barDelta_{2n, x} - \barDelta_{2n, x}^{(i')}| +|\barDelta_{2n, x, i^*} - \barDelta_{2n, x, i^*}^{(i')}|  + I(i' = i)|\xi_{b, i}| \Big)\Big] \notag\\
& \hspace{1cm}+ \bE\Big[ \underbrace{ |W_b|e^{W_b/2} }_{\leq 2 (1 + e^{W_b})}\Big( |\barDelta_{2n, x} - \barDelta_{2n, x, i^*}| + |\xi_{b, i}| + |t| + \beta_2 + \beta_3 \Big)\Big] \notag\\
& \hspace{1cm}+  \sum_{i'=1}^n  \Big|\bE[ \xi_{b, i'}]\Big| \bE\Big[ e^{W_b^{(i')}/2} \underbrace{\Big(|t|  +|\xi_{b, i}|I(i' \neq i) + |\barDelta_{2n, x}^{(i')} - \barDelta_{2n, x, i^*}^{(i')}| + \beta_2 + \beta_3\Big)}_{\leq C(1 +x)} \Big] 
\Bigg\}\notag\\
&\leq C  \Bigg\{\beta_2 + \beta_3 +  \bE[|\xi_{b, i}|^2 e^{ W_b^{(i)}/2}] \notag\\
&\hspace{0.1cm}+ \sum_{j=1}^2 \sum_{i' = 1}^n \bE\bigg[ |\xi_{b, i'}|e^{W_b^{(i')}/2} \Big( | \bar{D}_{jn} - \bar{D}_{jn}^{(i')}| +|\bar{D}_{jn, i^*} - \bar{D}_{jn, i^*}^{(i')}|   \Big)\bigg] \notag\\
& \hspace{0.1cm}+ \bE\Big[ (1 + e^{W_b})\Big( \sum_{j=1}^2 | \bar{D}_{jn} -\bar{D}_{jn, i^*}| + |\xi_{b, i}| + |t| + \beta_2 + \beta_3 \Big)\Big] +  \sum_{i' = 1}^n  \Big|\bE[ \xi_{b, i'}]\Big| \bE\Big[ e^{W_b^{(i')}/2} \Big] 
\Bigg\}\notag\\
&\leq C \Bigg\{\beta_2 + \beta_3  + \bE[|\xi_{b, i}|^2] + \sum_{j=1}^2 \sum_{i' = 1}^n  \bE\bigg[ |\xi_{b, i'}|e^{W_b^{( i')}/2} \Big( | \bar{D}_{jn} - \bar{D}_{jn}^{(i')}| +|\bar{D}_{jn, i^*} - \bar{D}_{jn, i^*}^{(i')}|   \Big)\bigg]\notag\\
& \hspace{5cm} +  \sum_{j=1}^2 \| (1+ e^{W_b})(\bar{D}_{jn} - \bar{D}_{jn, i^*})\|_1 + \|\xi_{b, i}\|_2 + |t| 
\Bigg\}\label{bdd_integrand_R12};
\end{align}
 in  \eqref{bdd_integrand_R12},  we have used that $\sum_{i' =1}^n |\bE[ \xi_{b, i'}]| \leq \beta_2$ by \lemref{exp_xi_bi_bdd} and 
\[
\max( \|e^{W_b}\|_2,  \|e^{W_b}\|_1, \bE[e^{W_b^{(i')}/2}], \bE[e^{W_b^{(i)}/2}] )\leq C
\]
by Bennett's inequality (\lemref{Bennett}). Since \eqref{bdd_integrand_R12} bounds \eqref{two_probs} which bounds the integrand of $R_{12}$, on integration with respect to $t$ which has the properties in \eqref{kb_properties},  we get
\begin{multline} \label{R12_prefinal_bdd}
|R_{12}| \leq C \Bigg\{ \beta_2 + \beta_3 +
  \sum_{j=1}^2
\Bigg[
\sum_{i=1}^n\bE[\xi_{b, i}^2] \Big\| (1+ e^{W_b^{(i)}})(\bar{D}_{jn} - \bar{D}_{jn, i^*})\Big\|_1 +  \\
\sum_{i=1}^n \bE[\xi_{b, i}^2]\sum_{ i' = 1}^n  \bE\Big[ |\xi_{b, i'}|e^{W_b^{( i')}/2} \big( | \bar{D}_{jn} - \bar{D}_{jn}^{(i')}| +|\bar{D}_{jn, i^*} - \bar{D}_{jn, i^*}^{(i')}|   \big)\Big]
\Bigg]
\Bigg\}
\end{multline}
where we have used $\sum_{i=1}^n \|\xi_{b, i}\|_2^4 \leq \sum_{i=1}^n \|\xi_{b, i}\|_2 \|\xi_{b, i}\|_2^2 \leq \sum_{i=1}^n \bE[|\xi_{b, i}|^3] \leq \beta_2 + \beta_3$ by \eqref{split_out_beta2}. From  \eqref{R12_prefinal_bdd}, by defining 
\[
W_b^{(i, i')}   \equiv \begin{cases}
 W_b - \xi_{b, i} - \xi_{b, i'}  &  \text{ if } i' \neq i \\
 W_b - \xi_{b, i}  &  \text{ if } i' = i
\end{cases},
\]
 with $e^{W_b^{(i')}/2} \leq e^{1/2} e^{W_b^{(i, i')}/2} $, we further get
\begin{align} 
|R_{12}| &\leq C \Bigg\{ \beta_2 + \beta_3 +
  \sum_{j=1}^2
\Bigg[
\sum_{i=1}^n \bE[\xi_{b, i}^2] \| (1+ e^{W_b^{(i)}})(\bar{D}_{jn} - \bar{D}_{jn}^{(i)} + \bar{D}_{jn}^{(i)}      - \bar{D}_{jn, i^*})\|_1 +  \notag \\
& \sum_{i=1}^n \bE[\xi_{b, i}^2]\sum_{i'  = 1}^n  \bE\Big[ |\xi_{b, i'}|e^{W_b^{(i, i')}/2} \big( | \bar{D}_{jn} - \bar{D}_{jn}^{(i')}| +|\bar{D}_{jn, i^*} - \bar{D}_{jn, i^*}^{(i')}|   \big)\Big]\Bigg]
\Bigg\} \notag\\
&\leq  C \Bigg\{ \beta_2 + \beta_3 +
  \sum_{j=1}^2
\Bigg[
\sum_{i=1}^n  \bE[\xi_{b, i}^2] \| (1+ e^{W_b^{(i)}})(\bar{D}_{jn} - \bar{D}_{jn}^{(i)} )\|_1 +  \notag \\
& \hspace{2cm}\sum_{i=1}^n  \bE[\xi_{b, i}^2] \sum_{ i' = 1}^n  \bE\Big[ |\xi_{b, i'}|e^{W_b^{(i, i')}/2}  | \bar{D}_{jn} - \bar{D}_{jn}^{(i')}|    \Big]\Bigg]
\Bigg\}, \label{R12_final_bdd2}
\end{align}
where we have used that 
\begin{multline*}
(e^{W_b^{(i)}}, \bar{D}_{jn} - \bar{D}_{jn}^{(i)})=_d (e^{W_b^{(i)}}, \bar{D}_{jn, i^*} - \bar{D}_{jn}^{(i)}) \text{ and }\\
(|\xi_{b, i'}|e^{W_b^{(i, i')}/2},   \bar{D}_{jn} - \bar{D}_{jn}^{(i')}) =_d (|\xi_{b, i'}|e^{W_b^{(i, i')}/2},   \bar{D}_{jn, i^*} - \bar{D}_{jn, i^*}^{(i')})
\end{multline*}
to arrive at \eqref{R12_final_bdd2}.
Lastly,  \eqref{R12_final_bdd2} can be further simplified as
\begin{multline} \label{R12_final_bdd}
|R_{12}|\leq \\
C \Bigg\{ \beta_2 + \beta_3 +
  \sum_{j=1}^2
\sum_{i=1}^n \bigg( \bE[\xi^2_{b, i}] \Big\| (1+ e^{W_b^{(i)}})(\bar{D}_{jn} - \bar{D}_{jn}^{(i)} )\Big\|_1 +  \bE\Big[ |\xi_{b, i}|e^{W_b^{(i)}/2}  | \bar{D}_{jn} - \bar{D}_{jn}^{(i)}|    \Big]\bigg)  \Bigg\}
\end{multline}
 using $e^{W_b^{(i, i')}/2} \leq e^{(W_b^{(i')} +1)/2}$ and $\sum_{i=1}^n \bE[\xi_{b, i}^2] \leq \sum_{i=1}^n \bE[\xi_i^2] = 1$ by \eqref{assumptions}. Combining \eqref{R_11_bdd} and \eqref{R12_final_bdd} gives  \eqref{R1_final_bdd}.

\subsubsection{Bound for $R_2$} 
Since $|f_x'| \leq 1$ by \lemref{helping_unif}, 
\begin{equation} \label{R2_first_component}
|\sum_{i=1}^n \bE[ (\xi_i^2- 1)I(|\xi_i|> 1)]  \bE[f_x'(W_b -  \bar{\Delta}_{2n,x})]  |\leq \sum_{i=1}^n \bE[\xi_i^2 I(|\xi| >1)] \leq \beta_2.
\end{equation}
Moreover, by independence, \lemref{exp_xi_bi_bdd} and that $|f_x| \leq 0.63$ from  \lemref{helping_unif},
\begin{multline*}
 \bigg|\sum_{i=1}^n\bE[\xi_{b, i} f(W_b^{(i)} -  \bar{\Delta}_{2n, x, i^*})] \bigg|= \bigg| \sum_{i=1}^n\bE[\xi_{b, i}]\bE[ f(W_b^{(i)} -  \bar{\Delta}_{2n, x, i^*})] \bigg|\\
   \leq 0.63 \sum_{i=1}
^n |\bE[\xi_{b, i} ]| \leq 0.63 \sum_{i=1}^n \bE[\xi_i^2 I(|\xi_i| >1)] = 0.63  \beta_2.
\end{multline*}
Lastly, by $|f_x| \leq 0.63$ and the definition of $\bar{\Delta}_{2n, x}$,
\[
|\bE[ \bar{\Delta}_{2n, x} f_x (W_b)] | \leq 0.63 \|\bar{D}_{1n}\|_2 + \Big|\frac{x}{2} \bE[\bar{D}_{2n} f_x(W_b)]\Big|
\]
Hence we established \eqref{R2_final_bdd}.

\subsubsection{Bound for $R_3$} By mean value theorem, given $|f'_x| \leq 1$ (\lemref{helping_unif}), \begin{align*}
|f_x(W_b -  \bar{\Delta}_{2n, x}) - f_x(W_b -  \bar{\Delta}_{2n, x, i^*})| &\leq C | \bar{\Delta}_{2n, x} -  \bar{\Delta}_{2n, x, i^*}| \\ 
&\leq C (|\bar{D}_{1n} - \bar{D}_{1n, i^*}| + x|\bar{D}_{2n} - \bar{D}_{2n, i^*}|).
\end{align*}
Hence 
\begin{multline} \label{R3_bdd_for_x_leq_1}
|R_3|\leq C \sum_{j=1}^2 \sum_{i=1}^n \|\xi_{b, i} (\bar{D}_{jn} - \bar{D}_{jn, i^*})\|_1 \\
= C \sum_{j=1}^2 \sum_{i=1}^n \|\xi_{b, i} (\bar{D}_{jn} - \bar{D}_{jn}^{(i)} + \bar{D}_{jn}^{(i)}   -  \bar{D}_{jn, i^*})\|_1 \text{ for }0 \leq x \leq 1.
\end{multline}
For $ x >1$, given $|\bar{\Delta}_{2n, x}|\vee |\bar{\Delta}_{2n, x, i^*}| \leq \frac{1}{2} + \frac{x}{4}$, by \eqref{fx'bdd} in \lemref{helping} and $|f'_x| \leq 1$ (\lemref{helping_unif}),
\begin{align*}
&|f_x(W_b -  \bar{\Delta}_{2n, x}) - f_x(W_b -  \bar{\Delta}_{2n, x, i^*})| \\
&\leq  |f_x(W_b -  \bar{\Delta}_{2n, x}) - f_x(W_b -  \bar{\Delta}_{2n, x, i^*})| \Big[ I(W_b \leq 3x/4 - 3/2) + I(W_b > 3x/4 - 3/2)\Big]\\
&\leq C \Big(e^{1/2-x} +I(W_b > 3x/4 - 3/2) \Big)\Big(|\bar{D}_{1n} - \bar{D}_{1n, i^*}| + x|\bar{D}_{2n} - \bar{D}_{2n, i^*}|\Big)\\
&\leq C\Big(e^{-x} + e^{-3x/8} e^{W_b/2} \Big) \Big(|\bar{D}_{1n} - \bar{D}_{1n, i^*}| + x|\bar{D}_{2n} - \bar{D}_{2n, i^*}|\Big) \\
&\leq  C\Big(e^{-x} + e^{-3x/8} e^{W_b^{(i)}/2} \Big) \Big(|\bar{D}_{1n} - \bar{D}_{1n, i^*}| + x|\bar{D}_{2n} - \bar{D}_{2n, i^*}|\Big),
\end{align*}
where we have used $e^{W_b/2} \leq e^{1/2} e^{W_b^{(i)}/2}$ in the last inequality. Hence, 
\begin{equation}   \label{R3_bdd_for_x_geq_1}
|R_3|  \leq  C \sum_{j=1}^2 \sum_{i=1}^n \| \xi_{b, i} (1 + e^{W_b^{(i)}/2})(\bar{D}_{jn} - \bar{D}_{jn}^{(i)} + \bar{D}_{jn}^{(i)}- \bar{D}_{jn, i^*})\|_1 \text{ for } x >1
\end{equation}
Because $(\xi_{b, i}, W_b^{(i)}, \bar{D}_{jn} - \bar{D}_{jn}^{(i)}) =_d (\xi_{b, i}, W_b^{(i)}, \bar{D}_{jn, i^*} - \bar{D}_{jn}^{(i)})$, \eqref{R3_bdd_for_x_leq_1} and \eqref{R3_bdd_for_x_geq_1} establishe \eqref{R3_final_bdd}.

\subsubsection{Bound for $R_4$}  Using that $|f_x'|\leq 1$ in \lemref{helping_unif}, for $0 \leq x \leq 1$,
\[
\bE\Big[  \barDelta_{2n, x} \int_0^{-  \bar{\Delta}_{2n, x}}f_x'(W_b + t) dt \Big] \leq C \barDelta_{2n, x}^2
\leq C(\|\bar{D}_{1n}\|_2^2  + \|\bar{D}_{2n}\|_2^2) \leq C(\|\bar{D}_{1n}\|_2 +  \|\bar{D}_{2n}\|_2^2).
\]
For $x > 1$, using \eqref{fx'bdd} in \lemref{helping} and that $|f_x'| \leq 1$ in \lemref{helping_unif}, given  $|\bar{\Delta}_{2n, x}| \leq \frac{1}{2} + \frac{x}{4}$
\begin{align*}
&\bE\Big[  \barDelta_{2n, x} \int_0^{-  \bar{\Delta}_{2n, x}}f_x'(W_b + t) dt \Big]\\
&\leq e^{1/2-x} \bE[\barDelta_{2n, x}^2] + \bE[I(W_b \geq  3x/4 - 3/2) \barDelta_{2n, x}^2]\\
&\leq C( e^{-x} \bE[\barDelta_{2n, x}^2] + e^{-3x/4} \bE[ e^{W_b} \barDelta_{2n, x}^2  ] )\\
&\leq C \Bigg\{2 e^{-x} \bigg(\|\bar{D}_{1n}\|_2^2 + \frac{x^2}{4}\|\bar{D}_{2n}\|_2^2\Bigg) + 2 e^{-3x/4} \bE\Bigg[e^{W_b} \Bigg(\bar{D}_{1n}^2 + \frac{x^2}{4}\bar{D}_{2n}^2\Bigg)\Bigg] \Bigg\}\\
&\leq C( \|\bar{D}_{1n}\|_2 + \bE[(1 +e^{W_b} )\bar{D}_{2n}^2]),
\end{align*}
where we have used $\bE[e^{W_b}|\bar{D}_{1n}|^2]\leq \bE[e^{W_b}|\bar{D}_{1n}|] \leq \|e^{W_b}\|_2 \|\bar{D}_{1n}\|_2 \leq C \|\bar{D}_{1n}\|_2$ by \lemref{Bennett} and $\|\bar{D}_{1n}\|_2^2 \leq \|\bar{D}_{1n}\|_2$. This establishes \eqref{R4_final_bdd}.

\section{Proof of \thmref{main2}} \label{app:main2_pf}

We first verify \eqref{D1_first_error_bdd}-\eqref{D1_third_error_bdd}, which will also be used in the proof of \thmref{main2}; \eqref{D1_third_error_bdd} is immediate from  \eqref{easy_easy_bdd}. We can prove  \eqref{D1_first_error_bdd} with H\"older's inequality as
\begin{align*}
 \| (1+ e^{W_b^{(i)}})(\bar{D}_{1n} - \bar{D}_{1n}^{(i)} )\|_1  
 &\leq   \| 1+ e^{W_b^{(i)}}\|_2 \|\bar{D}_{1n} - \bar{D}_{1n}^{(i)} \|_2  \notag\\
 &\leq  \Big(1 +  \exp( e^{4}/8 - 1/8 + 1/2) \Big)\Big\|D_{1n} - D_{1n}^{(i)} \Big\|_2,
\end{align*}
where we have also used Bennett's inequality (\lemref{Bennett}) and \eqref{easy_bdd} at the end. Similarly, \eqref{D1_second_error_bdd} can be proved as 
\begin{align*}
\|  \xi_{b, i} ( 1+e^{W_b^{(i)}/2})  ( \bar{D}_{1n} - \bar{D}_{1n}^{(i)})   \|_1
 &\leq   \| \xi_{b, i}( 1+e^{W_b^{(i)}/2}) \|_2 \|\bar{D}_{1n} - \bar{D}_{1n}^{(i)} \|_2  \notag\\
 &=  \| \xi_{b, i} \|_2 \|1+e^{W_b^{(i)}/2} \|_2 \|\bar{D}_{1n} - \bar{D}_{1n}^{(i)} \|_2 \notag\\
 &\leq   \Big( 1+\exp( e^2/8 - 1/8 + 1/4)\Big)  \| \xi_i \|_2 \Big\|D_{1n} - D_{1n}^{(i)} \Big\|_2,
\end{align*}
where we have also used the independence of $e^{W_b^{(i)}}$ and $ \xi_{b, i} $.

Our next task is to bound the other terms in the general bound of \thmref{main}.
Let
 \[
 \bar{\Pi}_k = \Pi_ k I(|\Pi_k| \leq 1) + I(\Pi_k >1) - I(\Pi_k <-1)\text{ for } k =1, 2.
 \]
  Since $|D_{2n}| \leq |\Pi_1| +  |\Pi_2|$, and  $|\bar{D}_{2n}|$ is precisely $|D_{2n}|$ as a non-negative random variable upper-censored at $1/2$, it must be that
$
|\bar{D}_{2n}|\leq |\bar{\Pi}_1| + |\bar{\Pi}_2|
$,
which further implies
\begin{equation} \label{D2n_pi12_bdd}
\bar{D}_{2n}^2\leq 2 (   \bar{\Pi}_1^2 +  \bar{\Pi}_2^2).
\end{equation}
From \eqref{D2n_pi12_bdd} and $\bar{\Pi}_2^2 \leq |\bar{\Pi}_2|$, we can get
\begin{equation}\label{bdd_on_D2n2_in_Pi2}
\bE[\bar{D}_{2n}^2] \leq 2(\|\Pi_1\|_2^2 + \|\Pi_2\|_2  )
\end{equation}
On the other hand, define
\[
D_{2n}^{(i)} =  \max \Bigg(-1,  \quad \sum_{ 1 \leq i' \leq n, i' \neq i} (\xi_{b, i'}^2 - \bE[\xi_{b, i'}^2])+  \Pi_2^{(i)}\Bigg). 
\]
By \propertyref{censoring_property}$(i)$, one can then write
\begin{align}
 \| (1+ e^{W_b^{(i)}})(\bar{D}_{2n} - \bar{D}_{2n}^{(i)} )\|_1  
 &\leq    \| (1+ e^{W_b^{(i)}} )(\xi_{b, i}^2 - \bE[\xi_{b, i}^2]) \|_1  +  \| (1+ e^{W_b^{(i)}} )(\Pi_2 -\Pi_2^{(i)}) \|_1 \notag\\
 &\leq  \| 1+ e^{W_b^{(i)}} \|_3\|\xi_{b, i}^2 - \bE[\xi_{b, i}^2] \|_{3/2}  +  \| 1+ e^{W_b^{(i)}} \|_2 \|\Pi_2 -\Pi_2^{(i)} \|_2 \notag \\
 &\leq C\Big( (\bE[|\xi_i|^3])^{2/3} +   \|\Pi_2 -\Pi_2^{(i)} \|_2\Big) \label{D2_first_error_bdd}
\end{align}
and
\begin{align}
& \|  \xi_{b, i}( 1+ e^{W_b^{(i)}/2}) (\bar{D}_{2n} - \bar{D}_{2n}^{(i)} )\|_1  \notag\\
 &\leq    \|  \xi_{b, i}( 1+e^{W_b^{(i)}/2} )(\xi_{b, i}^2 - \bE[\xi_{b, i}^2]) \|_1  +  \| \xi_{b, i}( 1+e^{W_b^{(i)}/2}) (\Pi_2 -\Pi_2^{(i)}) \|_1 \notag\\
 &\leq  \| \xi_{b, i}\|_3  \| 1+e^{W_b^{(i)}}\|_3 \|\xi_{b, i}^2 - \bE[\xi_{b, i}^2] \|_{3/2}  +  \| \xi_{b, i} \|_2  \| 1+e^{W_b^{(i)}}\|_2 \|\Pi_2 -\Pi_2^{(i)} \|_2 \notag\\
 &\leq C\Big( \bE[|\xi_i|^3] +   \| \xi_i \|_2  \|\Pi_2 -\Pi_2^{(i)} \|_2\Big), \label{D2_second_error_bdd}
\end{align}
where we have applied Bennett's inequality (\lemref{Bennett}) to both \eqref{D2_first_error_bdd} and \eqref{D2_second_error_bdd} at the end. 
 To complete the proof, 
 it suffices to show the bounds 
 \begin{equation} \label{bound_on_ewD2_sq_simplified}
\bE[e^{W_b} \bar{D}_{2n}^2] \leq C  \Bigg\{\sum_{i=1}^n  \|\xi_{b, i}\|_3^3 +   \|\Pi_2\|_2  \Bigg\}
\end{equation}
and
\begin{equation} \label{bound_on_x_D2_fx_simplified}
\sup_{x \geq 0} |x \bE[\bar{D}_{2n} f_x(W_b)]| \leq C \Big( \|\Pi_1\|_2^2 +  \sum_{i=1}^n  \|\xi_{b, i}\|_3^3 + \|\Pi_2\|_2 \Big),
\end{equation}
because \thmref{main2} is then just a corollary of \thmref{main} by collecting \eqref{D1_first_error_bdd}-\eqref{D1_third_error_bdd}, \eqref{bdd_on_D2n2_in_Pi2}-\eqref{bound_on_x_D2_fx_simplified}, as well as the simple facts
\begin{multline*}
\beta_2 + \beta_3 \leq \sum_{i=1}^n \bE[|\xi_i|^3], \qquad
  \bE[|\xi_{b, i}|^2] \leq \|\xi_{b, i}\|_2 \leq  \|\xi_i\|_2 \leq    \|\xi_i\|_3,  \\
P(|D_{1n}| > 1/2) \leq 2\|D_{1n}\|_2, \qquad
 \|\Pi_1\|_2^2 \leq \sum_{i=1}^n \bE[\xi_{b, i}^4] \leq \sum_{i=1}^n \bE[|\xi_{b, i}|^3] \leq \sum_{i=1}^n \bE[|\xi_i|^3],
\end{multline*}
and
\begin{align*}
P(|D_{2n}| > 1/2)  &\leq P(|\Pi_1| + |\Pi_2| > 1/2)\\
& \leq P(|\Pi_1| >1/4) + P(|\Pi_2| >1/4) \\
&\leq  C(\|\Pi_1\|^2_2 + \|\Pi_2\|_2).
\end{align*}

\subsection{Proof of  \eqref{bound_on_ewD2_sq_simplified}.}

First, letting  $W_b^{(i, j)} \equiv W_b - \xi_{b, i} - \xi_{b, j}$ for $1 \leq i \neq j \leq n$, we have
\begin{align}
&\bE[ \Pi_1^2   e^{W_b}] = \sum_{i=1}^n \bE[ (\xi_{b, i}^2 - \bE[\xi_{b, i}^2] )^2 e^{\xi_{b, i}}] \bE[e^{ W_b^{(i)}}]
+ \notag\\
&\hspace{3cm}\sum_{1\leq i \neq j \leq n} \bE[(\xi_{b, i}^2 - \bE[\xi_{b, i}^2] ) e^{ \xi_{b, i}}]  \bE[(\xi_{b, j}^2 - \bE[\xi_{b, j}^2] ) e^{ \xi_{b, j}}] \bE[e^{ W_b^{(i,j)}}] \notag\\
&= \sum_{i=1}^n \bE[ (\xi_{b, i}^2 - \bE[\xi_{b, i}^2] )^2 e^{\xi_{b, i}}] \bE[e^{ W_b^{(i)}}]
+ \notag\\
&\hspace{3cm}\sum_{1\leq i \neq j \leq n} \bE[(\xi_{b, i}^2 - \bE[\xi_{b, i}^2] ) (e^{ \xi_{b, i}} - 1)]  \bE[(\xi_{b, j}^2 - \bE[\xi_{b, j}^2] ) (e^{ \xi_{b, j}} - 1)] \bE[e^{ W_b^{(i,j)}}] \notag\\
&\leq C\left( \sum_{i=1}^n \bE[\xi_{b, i}^4] + \sum_{1 \leq i \neq j \leq n} \bE \Big[|\xi_{b, i}^2 - \bE[\xi_{b, i}^2] || \xi_{b, i}|\Big]  \bE\Big[|\xi_{b, j}^2 - \bE[\xi_{b, j}^2] | |\xi_{b, j}|\Big] \bE\Big[e^{W_b^{(i,j)}}\Big] \right) \notag\\
&\leq C \Bigg\{\sum_{i=1}^n  \|\xi_{b, i}\|_3^3 +  \sum_{1 \leq i \neq j \leq n}  \|\xi_{b, i}\|_3^3  \|\xi_{b, j}\|_2^2 \Bigg\} \leq C \sum_{i=1}^n  \|\xi_{b, i}\|_3^3
\label{Pi1_expbdd}
\end{align}
by \lemref{Bennett}, that $|e^s - 1| \leq |s| (e^a - 1)/a$ for $s \leq a$ and $a >0$, 
\begin{align}
 &\bE[|\xi_{b, i}^2 - \bE[\xi_{b, i}^2] || \xi_{b, i}|] \notag\\
  &\leq  \Bigg\{(\|\xi_{b, i}^2 - \bE[\xi_{b, i}^2]\|_{3/2} \| \xi_{b, i}\|_3)\wedge \bE[|\xi_{b, i}^2 - \bE[\xi_{b, i}^2]|] \Bigg\}\notag\\
 &\leq  2 \Bigg\{ \|\xi_{b, i}\|_3^3 \wedge  \|\xi_{b, i}\|_2^2 \Bigg\} \text{ for any } i = 1, \dots, n,\notag
\end{align}
and $\sum_{j=1}^n\|\xi_{b, i}\|_3^3  \|\xi_{b, j}\|_2^2 \leq \|\xi_{b, i}\|_3^3$. Second, by \lemref{Bennett}, 
\begin{equation} \label{Pi2_expbdd}
\bE[\bar{\Pi}_2^2 e^{W_b}] \leq \bE[\bar{\Pi}_2^4]^{1/2} (\bE[e^{2W_b}])^{1/2} \leq C \bE[\Pi_2^2]^{1/2} = C \|\Pi_2\|_2 
\end{equation}
Combining \eqref{D2n_pi12_bdd}, \eqref{Pi1_expbdd} and \eqref{Pi2_expbdd} gives  \eqref{bound_on_ewD2_sq_simplified}.

\subsection{Proof of   \eqref{bound_on_x_D2_fx_simplified}.} \label{sec:pf_for_bound_on_x_D2_fx_simplified}

Since $\sup_{x \geq 0}|xf_x(w)| \leq C$ (which uses \eqref{fxbdd} in \lemref{helping} and that $|f_x| \leq 0.63$ in \lemref{helping_unif}),
\begin{align}
\sup_{x \geq 0}| x\bE[(D_{2n} - \bar{D}_{2n})f_x(W_b)]| &\leq \sup_{x \geq 0} x\bE[ (|D_{2n}| - 1/2)|f_x(W_b)| I(|D_{2n}| > 1/2)] \notag\\
 &\leq C \bE[ |D_{2n}| I(|D_{2n}| > 1/2)] \notag\\
 &\leq  C \Big(\bE[|\Pi_1|  I(|D_{2n}| > 1/2)] + \bE[|\Pi_2|] \Big)\notag\\
 &\leq C \Big( \bE \Big[|\Pi_1|  \Big\{I(|\Pi_1| > 1/4)+I(|\Pi_2| > 1/4) \Big\}\Big]  + \bE[|\Pi_2|] \Big)\notag\\
 &\leq C\Big(  \bE[ 4 \Pi_1^2 + 2 |\Pi_1| |\Pi_2|^{1/2}] + \bE[|\Pi_2|] \Big) \notag\\
 &\leq C\Big(   \bE[ 5\Pi_1^2 +  |\Pi_2| ] + \bE[|\Pi_2|] \Big) \notag\\
 &\leq   C\Big(\|\Pi_1\|_2^2 + \|\Pi_2\|_2 \Big), \notag
 \end{align}
 where we have used that $I(|\Pi_1| > 1/4) \leq 4 |\Pi_1|$, $I(|\Pi_2| > 1/4) \leq 2 |\Pi_2|^{1/2}$ and $2 |\Pi_1| |\Pi_2|^{1/2} \leq |\Pi_1|^2 + |\Pi_2|$.
Noting that
\begin{equation*}
x\bE[\bar{D}_{2n} f_x(W_b)] = x\bE[(\bar{D}_{2n} - D_{2n})f_x(W_b)] +x\bE[ D_{2n}f_x(W_b)],
\end{equation*}
the above implies
\begin{equation} \label{fxd2n_remnants_bdd}
\sup_{x \geq 0} \Big|x\bE[\bar{D}_{2n} f_x(W_b)] \Big| \leq C\Big(\|\Pi_1\|_2^2 + \|\Pi_2\|_2 \Big) + \sup_{x \geq 0} \Big|x\bE[ D_{2n}f_x(W_b)] \Big|,
\end{equation}
so for the rest of this section we  focus on bounding $\sup_{x \geq 0} \Big|x\bE[ D_{2n}f_x(W_b)] \Big|$.
From the form of $D_{2n}$ in \eqref{D2_specific_form}, by defining $\Pi  = \Pi_1 + \Pi_2$,  we have
\[
x\bE[D_{2n} f_x(W_b)] = \bE[x \Pi f_x(W_b)] - \bE[ xf_x(W_b)I(\Pi < -1) (1 + \Pi)],
\]
so it suffices to establish 
\begin{equation}\label{boils_down_bar_D_2n_fx_simplified}
\Big|\bE[x \Pi f_x(W_b)]\Big| \vee \Big|\bE[ xf_x(W_b)I(\Pi < -1) (1 + \Pi)]\Big| \leq C \Bigg( \sum_{i=1}^n \bE[|\xi_{b, i}|^3 ]  + \|\Pi_2\|_2 \Bigg) \text{ for all } x\geq 0.
\end{equation}

We first bound $\Big|\bE[ xf_x(W_b)I(\Pi < -1) (1 + \Pi)]\Big|$. Since 
\begin{equation} 
\bE[ xf_x(W_b)I(\Pi < -1) (1 + \Pi)] = \bE[xf_x(W_b)I(\Pi < -1)] +  \bE[xf_x(W_b) \Pi I(\Pi < -1)],
\end{equation}
we will bound the two terms on the right hand side separately.  As $x f_x(w)$ is   bounded for all $x \geq 0$ (\lemref{helping_unif} and \eqref{fxbdd} in \lemref{helping}), we have
\begin{align*}
\Big|\bE[xf_x(W_b)I(\Pi < -1)]\Big| &\leq \bE\Big[|xf_x(W_b)|I(\Pi < -1)\Big]\\
 &\leq  C \sum_{j=1}^2P(\Pi_j < -1/2) \leq C \Big(\|\Pi_1\|_2^2 + \|\Pi_2\|_2\Big) \end{align*}
and 
\begin{align*}
\Big|\bE[xf_x(W_b) \Pi I(\Pi < -1)]\Big| &\leq C\bE[|\Pi|I(\Pi < -1)]\\\
&\leq C \Bigg( \bE[|\Pi_1| I(\Pi < -1)]  +  \|\Pi_2\|_2  \Bigg)\\
&\leq C  \Bigg(\|\Pi_1\|_2 \sqrt{ \sum_{j=1}^2P(\Pi_j < -1/2)} +  \|\Pi_2\|_2  \Bigg) \\
&\leq C  \Bigg(\|\Pi_1\|_2 \sqrt{ \|\Pi_1\|_2^2 + \|\Pi_2\|_2 } +  \|\Pi_2\|_2 \Bigg)\\
&\leq C \Bigg( \|\Pi_1\|_2^2  + \|\Pi_1\|_2\sqrt{ \|\Pi_2\|_2} + \|\Pi_2\|_2   \Bigg)\\
&\leq C \Bigg( \|\Pi_1\|_2^2  +  \|\Pi_2\|_2 \Bigg) ,
\end{align*}
where the second last inequality uses $\sqrt{ \|\Pi_1\|_2^2 + \|\Pi_2\|_2 } \leq \|\Pi_1\|_2 + \sqrt{\|\Pi_2\|_2}$, and the  last inequality uses that $2|ab| \leq a^2 + b^2$ for any $a, b \in \bR$. So the part of \eqref{boils_down_bar_D_2n_fx_simplified} regarding $\Big|\bE[ xf_x(W_b)I(\Pi < -1) (1 + \Pi)]\Big|$ is proved because $\|\Pi_1\|_2^2 = \sum_{i=1}^n (\bE[\xi_{b, i}^4] - (\bE[\xi_{b, i}^2])^2) \leq \sum_{i=1}^n \bE[|\xi_{b, i}|^3 ]$.

Next we bound $\Big|\bE[x \Pi f_x(W_b)]\Big|$, and we will control the two terms on the right hand side of 
\begin{equation} \label{xD2nfxWb_first_bdd_simplified}
|\bE[x\Pi f_x(W_b)]| \leq x|\bE[\Pi_1 f_x(W_b)]| + x  | \bE[\Pi_2 f_x(W_b)]|.
\end{equation}
For the first term $x|\bE[\Pi_1 f_x(W_b)]|$, we  write
\begin{align}
\Big| \bE[\Pi_1 f_x(W_b)]\Big| &= \Bigg| \sum_{i=1}^n \bE\Big[(\xi_{b, i}^2 - \bE[\xi_{b, i}^2] ) (f_x (W_b) - f_x(W_b^{(i)})) 
\Big] \Bigg|\notag\\
&= \Bigg|\sum_{i=1}^n \bE\Big[(\xi_{b, i}^2 - \bE[\xi_{b, i}^2])  \int_0^{\xi_{b, i}} \bE[f_x'(W_b^{(i)} + t)] dt\Big] 
\Bigg|\notag\\
&\leq  \sum_{i=1}^n \bE\Big[ (\xi_{b, i}^2 + \bE[\xi_{b, i}^2])  \int_0^{|\xi_{b, i}|} |\bE[f_x'(W_b^{(i)} + t)]| dt\Big] \label{fromthere_simplified}, 
\end{align} 
where the second equality uses the independence of $W_b^{(i)} $ and $\xi_{b, i}$. 
From \eqref{fromthere_simplified}  and \lemref{expect_f'x}, for any $x \geq 1$,  we have that 
\begin{align*}
\Big|\bE[\Pi_1 f_x(W_b)]\Big| &\leq C\sum_{i=1}^n \bE\Big[ (\xi_{b, i}^2 + \bE[\xi_{b, i}^2])   \int_0^{|\xi_{b, i}|} (e^{ -x } + e^{- x +t})dt\Big]\\
&\leq C\sum_{i=1}^n \bE\Big[( \xi_{b, i}^2 + \bE[\xi_{b, i}^2]  ) \int_0^{|\xi_{b, i}|} (e^{ -x } + e^{- x +1})dt\Big]
\text{ ( as $|\xi_{b, i}| \leq 1$ )}\\
&\leq   C e^{-x}\sum_{i=1}^n \bigg(\bE[|\xi_{b,i}|^3] + \bE[|\xi_{b,i}|^2] \bE[|\xi_{b,i}|]\bigg) \\
&\leq C e^{-x}\sum_{i=1}^n \bE[|\xi_{b,i}|^3],
\end{align*}
which implies
\begin{equation} \label{supxgeq1_x_Pi1fx_simplified}
\sup_{x \geq 1} x \Big|\bE[\Pi_1 f_x(W_b)]\Big|  \leq C  \sum_{i=1}^n \bE[|\xi_{b, i}|^3] .
\end{equation}
Moreover, for $0 \leq x < 1$, since $|f_x'| \leq 1$ (\lemref{helping_unif}), from \eqref{fromthere_simplified} we get
\begin{equation} \label{supx0leq1_x_Pi1fx_simplified}
\sup_{0 \leq x < 1}  x \Big| \bE[\Pi_1 f_x(W_b)]\Big| \leq  \sum_{i=1}^n \bigg( \bE[|\xi_{b, i}|^3] + \bE[|\xi_{b, i}|^2] \bE[|\xi_{b, i}|]  \bigg) \leq 2 \sum_{i=1}^n \bE[|\xi_{b, i}|^3].
\end{equation}
For the term $x | \bE[\Pi_2 f_x(W_b)]|$, given that
$
\sup_{x \geq 0}|xf_x(w)| \leq C \text{ for all } w
$
(explained at the beginning of  \secref{pf_for_bound_on_x_D2_fx_simplified}), we have
\begin{equation} \label{supx_x_Pi2fx_simplified}
\sup_{x \geq 0}x |\bE[\Pi_2 f_x (W_b)]| \leq \sup_{x \geq 0} \bE[|\Pi_2 | |x f_x(W_b)|] \leq C \|\Pi_2\|_1 \leq C \|\Pi_2\|_2,
\end{equation}
Combining \eqref{xD2nfxWb_first_bdd_simplified} and \eqref{supxgeq1_x_Pi1fx_simplified}-\eqref{supx_x_Pi2fx_simplified} proves the  part of \eqref{boils_down_bar_D_2n_fx_simplified} regarding $|\bE[x\Pi f_x(W_b)]|$.


\newpage
\section{Proof of \lemref{Djn_minus_D_ni}} \label{app:proof_of_D2_stuff}

In this section we adopt the following  notation: For any natural numbers $k' \leq k$, we denote $[k':k] \equiv \{k', \dots, k\}$ and   $[k] \equiv \{1, \dots, k\}$. Moreover, for any natural number $k \geq 1$,   we let
\[
\bar{h}_{k, \{i_1, \dots, i_k\}}  \equiv \bar{h}_k(X_{i_1}, \dots, X_{i_k})
\]
with respect to the function $\bar{h}_k(\cdot)$ in \eqref{h_k_def}.  To prove  \lemref{Djn_minus_D_ni}, we need the following technical lemmas proven respectively in \appsref{pf_of_ustat_results} and \appssref{pf_of_counting_identities}. 

\begin{lemma}[Useful kernel bounds] \label{lem:ustat_results}
Under assumptions \eqref{mean_zero_kernel}-\eqref{unit_var},
\begin{enumerate}
\item For any $k \in [m]$,
\[
\bE[\bar{h}_k^2  ] \leq \bE[h_k^2] \leq \frac{k}{m}\bE[h^2]
\]
\item For any $i \in [n]$,
\begin{multline*}
 \bE \left[ \left(\sum_{\substack{1 \leq i_1 < \dots < i_{m-1} \leq n\\ i_l \neq i \text{ for } l \in [m-1]}}\bar{h}_m(X_i, X_{i_1}, \dots, X_{i_{m-1}}) \right)^2\right] \\
 \leq  \frac{2(m-1)^2}{ n(n-m+1)}  {n-1 \choose m-1} {n \choose m} \bE[h^2] ;
\end{multline*}

\item For each $i \in [n]$, consider $\xi_{b, i}$ defined in \eqref{lower_and_upper_censored_xi_i} with $\xi_i$ defined in \eqref{u_stat_xi}. Given $k_1, k_2 \in [m]$, for any $1 \leq i_1 < \dots < i_{k_1} \leq n$ and $1 \leq j_1 < \dots < j_{k_2} \leq n$, we have
\begin{equation*}
\Big|\bE[\xi_{b, 1} \xi_{b, 2} \bar{h}_{k_1, \{ i_1, \dots, i_{k_1}\}}  \bar{h}_{k_2, \{j_1, \dots, j_{k_2}\}}]\Big| \\
 \leq \frac{9.5 \|g\|_3^2 \|h\|_3^2}{n} + \frac{2d\|h\|_2}{n}
\end{equation*}
where 
\[
d = | (\{ i_1, \dots, i_{k_1} \} \cap \{j_1, \dots, j_{k_2}\}) \backslash \{1, 2\}|,
\]
 the number of elements in the intersection of $\{ i_1, \dots, i_{k_1} \}$ and $\{ j_1, \dots, j_{k_2} \}$ that are not $1$ or $2$. 

\item If, in addition to all the conditions in $(iii)$, it is true that $1 \not \in \{j_1, \dots, j_{k_2}\}$ and $2 \not \in \{ i_1, \dots, i_{k_1} \} $, then we have the bound
\begin{equation*}
\Big|\bE[\xi_{b, 1} \xi_{b, 2} \bar{h}_{k_1, \{ i_1, \dots, i_{k_1}\}}  \bar{h}_{k_2, \{j_1, \dots, j_{k_2}\}}]\Big| 
 \leq \frac{9.5 \|g\|_3^2 \|h\|_3^2}{n} + \frac{2d\|h\|_2}{n^{3/2}}
\end{equation*}
\end{enumerate} 
\end{lemma}

\begin{lemma}[Counting identities and bounds]\label{lem:counting_identities}
Let $m , n $ be non-negative integers such that $m \leq n$.
\begin{enumerate}
\item Suppose $n_1$ and $n_2$ are non-negative integers such that $n_1 + n_2 = n$. Then
\[
\sum_{k = 0}^m {n_1 \choose k} {n_2 \choose m - k} = {n \choose m}.
\]
\item Suppose $k$ is a non-negative integer such that $k \leq m$. Then 
\[
{n \choose k} {n - k \choose m - k} = {n  \choose m} {m \choose k}.
\]
\item For positive integers  $a, b, e$ such that $b+ e \leq a$, we have
\[
\binom{a}{b} - \binom{a -e}{b} \leq \binom{a}{b} \frac{ b e}{a - b +1}.
\]
\end{enumerate}
\end{lemma}

In addition to the lemmas above, 
we will make use of the following enumerative equalities, whenever the binomial coefficients involved are well-defined:\begin{align}
       \binom{n-2}{m-1} &= \binom{n-1}{m-1}\frac{n-m}{n-1}, \label{enum_eq_1}\\
       \binom{n-2}{m-2} &= \binom{n-1}{m-1} \frac{m-1}{n-1},  \label{enum_eq_2}\\
  \binom{n-3}{m-2} &= \binom{n-1}{m-1} \frac{(m-1)(n-m)}{(n-1)(n-2)}   \label{enum_eq_3}\\
   \binom{n-3}{m-3} &=  \binom{n-1}{m-1} \frac{(m-1)(m-2)}{(n-1)(n-2)} \label{enum_eq_4}\text{, \quad and } \\
      \binom{n-4}{m-4} &=  \binom{n-1}{m-1} \frac{(m-1)(m-2)(m-3)}{(n-1)(n-2)(n-3)} \label{enum_eq_5}.
       \end{align}

\subsection{Proof of   \lemref{Djn_minus_D_ni}$(i)$} 
We shall further let
\begin{multline} \label{Pi1Pi2Pi3_def}
\Pi_{21} \equiv (n^{-1/2} | 0) - \sum_{i=1}^n \bE[ (\xi^2_i - 1)I(|\xi_i| > 1) ] \text{ and }\\
\Pi_{22} \equiv \delta_{2n, b} =  \frac{2 (n-1)}{(n-m)} {n-1 \choose m-1}^{-1} \sum_{i=1}^n \xi_{b, i} \Psi_{n,i},
\end{multline}
so $\Pi_2 = \Pi_{21} + \Pi_{22}$.
It suffices to show these  bounds for  $\Pi_{21}$ and $\Pi_{22}$ in \eqref{Pi1Pi2Pi3_def}:
 \begin{equation} \label{Pi21_bdd}
\|\Pi_{21}\|_2^2 \leq C \left(\frac{\|g\|_3^6}{n} + \frac{1}{n}\right)  \leq  C \frac{\|g\|_3^6}{n} .
\end{equation}
 \begin{equation}\label{Pi22_bdd}
\|\Pi_{22}\|_2^2\leq 
C\frac{m^2\|g\|_3^2 \|h\|_3^2}{n}  
 \end{equation}
 From there, since $\|\Pi_2\|_2 \leq \|\Pi_{21}\|_2 + \|\Pi_{22}\|_2$, \lemref{Djn_minus_D_ni}$(i)$ is proved.

\subsubsection{Proof of \eqref{Pi21_bdd}}

We first note that 
\[
\sum_{i=1}^n \bE\Big[ (\xi^2_i - 1)I(|\xi_i| > 1) \Big]
\leq \sum_{i=1}^n \bE \Big[ \xi^2_i I(|\xi_i| > 1) \Big] \leq \sum_{i=1}^n\bE[|\xi_i|^3] = \bE[|g|^3]/\sqrt{n},
\]
which gives $(\sum_{i=1}^n \bE[ (\xi^2_i - 1)I(|\xi_i| > 1) ])^2 \leq  (\bE[|g|^3])^2/n$, and hence \eqref{Pi21_bdd}.
 
 \subsubsection{Proof of \eqref{Pi22_bdd}}
It is trivial for $m=1$ since $\Psi_{n,i} = 0$. For $m\geq 2$, first write 
\[
\Pi_{22}^2 =  \frac{4 (n-1)^2}{(n-m)^2n} {n-1 \choose m-1}^{-2}    \left(\sum_{i=1}^n \xi_{b, i}  \sum_{\substack{1 \leq i_1 < \dots < i_{m-1} \leq n\\ i_l \neq i \text{ for } l \in [m-1]}} \bar{h}_m(X_i, X_{i_1}, \dots, X_{i_{m-1}}) \right)^2,
\]
which implies immediately from $2m < n$ in \eqref{u_stat_deg_assumption} that
\begin{equation}\label{Pi2_first_bdd}
\bE\left[\Pi_{22}^2\right] \leq \frac{16}{n} {n-1 \choose m-1}^{-2} \bE \left[    \left(\sum_{i=1}^n \xi_{b, i}  \sum_{\substack{1 \leq i_1 < \dots < i_{m-1} \leq n\\ i_l \neq i \text{ for } l\in [m-1]}} \bar{h}_m(X_i, X_{i_1}, \dots, X_{i_{m-1}}) \right)^2\right].
\end{equation}
Upon expanding the above expectation,
\begin{align}
&\bE\left[\left(\sum_{i=1}^n \xi_{b, i} \sum_{\substack{1 \leq i_1 < \dots < i_{m-1} \leq n\\ i_l \neq i \text{ for } l \in [m-1]}} \bar{h}_{m, \{i, i_1, \dots, i_{m-1}\}} \right)^2\right] \notag\\
&=  \sum_{i=1}^n 
\bE \left[ \left(\xi_{b, i} \sum_{\substack{1 \leq i_1 < \dots < i_{m-1} \leq n\\ i_l \neq i \text{ for } l \in [m-1]}}  \bar{h}_{m, \{i, i_1, \dots, i_{m-1}\}}\right)^2\right]\notag \\
&\  + \sum_{ 1\leq i \neq j \leq n} \bE
\Bigg[
\Bigg(\xi_{b, i} \sum_{\substack{1 \leq i_1 < \dots < i_{m-1} \leq n\\ i_l \neq i \text{ for } l \in [m-1]}}  \bar{h}_{m, \{i, i_1, \dots, i_{m-1}\}}\Bigg) \times 
\Bigg(
\xi_{b, j} \sum_{\substack{1 \leq j_1 < \dots < j_{m-1} \leq n\\ j_l \neq j \text{ for } l \in [m-1]}}  \bar{h}_{m, \{j, j_1, \dots, j_{m-1}\}}
\Bigg)\Bigg]\notag\\
&= n \bE \left[ \left(\xi_{b, 1} \sum_{\substack{1 \leq i_1 < \dots < i_{m-1} \leq n\\ i_l \neq 1 \text{ for } l = 1, \dots, m-1}}  \bar{h}_{m, \{1, i_1, \dots, i_{m-1}\}}\right)^2\right] + \label{exp1}\\
&n(n-1)
\bE
\Bigg[
\Bigg(\xi_{b, 1} \sum_{\substack{1 \leq i_1 < \dots < i_{m-1} \leq n\\ i_l \neq 1 \text{ for } l \in [m-1]}}  \bar{h}_{m, \{1, i_1, \dots, i_{m-1}\}}\Bigg) \Bigg(
\xi_{b, 2} \sum_{\substack{1 \leq j_1 < \dots < j_{m-1} \leq n\\ j_l \neq 2 \text{ for } l \in [m-1]}}  \bar{h}_{m, \{2, j_1, \dots, j_{m-1}\}}
\Bigg)\Bigg]. \label{exp2}
\end{align}
We need to control the two expectations   in \eqref{exp1} and \eqref{exp2}. We first bound the expectation  in \eqref{exp1}. With the definition in \eqref{h_k_def} and that
\[
\bE[\bar{h}_{m, \{1, i_1, \dots, i_{m-1}\}} \bar{h}_{m, \{1, j_1, \dots, j_{m-1}\}}] = \bE[\bar{h}_{1, \{1\}}^2] = 0 \text{ if }  |\{i_1, \dots, i_{m-1}\}\cap \{j_1, \dots, j_{m-1}\}| = 0, 
\]
we can write
\begin{align}
& \bE \left[ \left(\xi_{b, 1} \sum_{\substack{1 \leq i_1 < \dots < i_{m-1} \leq n\\ i_l \neq 1 \text{ for } l \in [m-1]}}  \bar{h}_{m, \{1, i_1, \dots, i_{m-1}\}}\right)^2\right] \notag\\
 &  = 
 \bE\left[
 \xi_{b, 1}^2
 \sum_{k=0}^{m-1}
\left( \sum_{\substack{1 \leq i_1 < \dots < i_{m-1} \leq n\\  
1 \leq j_1 < \dots < j_{m-1} \leq n
\\i_l, j_l \neq 1 \text{ for } l \in [m-1] \notag\\
|\{i_1, \dots, i_{m-1}\}\cap \{j_1, \dots, j_{m-1}\}| = k
}} 
 \bar{h}_{m, \{1, i_1, \dots, i_{m-1}\}}\bar{h}_{m, \{1, j_1, \dots, j_{m-1}\}} \right)
 \right]\notag\\
  &= 
\sum_{k=1}^{m-1}
{n-1 \choose k} {n -k  - 1\choose m - k - 1} {n - m  \choose m-k-1}   \bE\Bigg[       \xi_{b, 1}^2   
    \bar{h}_{k+1}^2(X_1 \dots, X_{k+1}) 
   \Bigg]\notag\\
      &\leq \sum_{k=1}^{m-1}
{n-1 \choose k} {n -k -1\choose m - k -1} {n - m  \choose m-k-1}
 \frac{k+1}{m}  \bE[ h^2 ], \label{exp1_bdd}
\end{align}
where the last inequality comes from \lemref{ustat_results}$(i)$ and that $\xi_{b, 1}^2 \leq 1$. Continuing from \eqref{exp1_bdd}, we can get
\begin{align}
& \bE \left[ \left(\xi_{b, 1} \sum_{\substack{1 \leq i_1 < \dots < i_{m-1} \leq n\\ i_l \neq 1 \text{ for } l \in [m-1]}}  \bar{h}_{m, \{1, i_1, \dots, i_{m-1}\}}\right)^2\right] \notag\\
&\leq \sum_{k=1}^{m-1}
{n-1 \choose k} {n -k -1\choose m - k -1} {n - m  \choose m-k-1}
 \frac{k+1}{m}  \bE[         
    h^2 ] \notag\\
   &= \frac{1}{m}\binom{n-1}{m-1}\sum_{k=1}^{m-1}\binom{m-1}{k}\binom{n-m}{m-k-1}(k+1) \bE[h^2] \text{ by \lemref{counting_identities}$(ii)$ }\notag\\
&= \frac{m-1}{m}\binom{n-1}{m-1}\sum_{k=1}^{m-1}\binom{m-2}{k-1} \frac{k+1}{k} \binom{n-m}{m-1-k}\bE [h^2] \notag\\
&\leq 2\binom{n-1}{m-1}\sum_{k=0}^{m-2}\binom{m-2}{k}\binom{n-m}{m-2-k}\bE[h^2] \notag\\
&=2\binom{n-1}{m-1}\binom{n-2}{m-2}\bE[h^2] \text{ by \lemref{counting_identities}$(i)$} \notag\\
& = 2\frac{m-1}{n-1}\binom{n-1}{m-1}^2\bE[h^2] \label{exp1_bdd_final}
\end{align}

Now we bound the expectation in \eqref{exp2}. We first expand it as
\begin{align}
&\bE
\Bigg[
\Bigg(\xi_{b, 1} \sum_{\substack{1 \leq i_1 < \dots < i_{m-1} \leq n\\ i_l \neq 1 \text{ for } l \in [m-1]}}  \bar{h}_{m, \{1, i_1, \dots, i_{m-1}\}}\Bigg)  \Bigg(
\xi_{b, 2} \sum_{\substack{1 \leq j_1 < \dots < j_{m-1} \leq n\\ j_l \neq 2 \text{ for } l \in [m-1]}}  \bar{h}_{m, \{2, j_1, \dots, j_{m-1}\}})
\Bigg)\Bigg] \label{exp2_expand}\\
&
 = {n-2 \choose m-1} {n-2-(m -1)\choose m-1}\Bigg(\underbrace{\bE\left[\xi_{b, 1}  \bar{h}_{m, \{1, \dots, m\}}\right]}_{
= \bE [\bE[\xi_{b,1} \bar{h}_{m, \{1, \dots, m\}} | X_1]] =\bE[\xi_{b,1}\bar{h}_1 (X_1)] =  0}\Bigg)^2 \notag\\
&
+ 2 \times  {n-2 \choose m-2} {n-2 - (m-2) \choose m-1}\underbrace{\bE \Bigg[ \xi_{b, 1} \xi_{b, 2} \bar{h}_{m, \{1, 2, \dots, m\}}  \bar{h}_{m, \{2, m+1, \dots, 2m-1 \}} \Bigg] }_{
 = \bE[ \bE[ \xi_{b, 1} \xi_{b, 2} \bar{h}_{2, \{1, 2\}}  \bar{h}_{1, \{2 \}}| X_1, X_2]]=0 \text{ since } \bar{h}_{1, \{ 2\}} = 0} \notag\\
 &
+ 2 \times 
\underbrace{\sum_{ \substack{ 1 \leq i_1 < \dots < i_{m-2} \leq n \\ 1 \leq j_1 < \dots < j_{m-1} \leq n\\
i_l, j_v \neq 1, 2, \text{ for } l \in [m-2], v \in [m-1]\\
|\{i_1, \dots, i_{m-2}\}\cap \{j_1, \dots, j_{m-1}\}| \geq 1
}}
\bE[\xi_{b, 1} \xi_{b, 2} \bar{h}_{m, \{1, 2, i_1, \dots, i_{m-2}\}}  \bar{h}_{m, \{2, j_1, \dots, j_{m-1}\}}] }_{\equiv EA}\notag\\
&
+ \underbrace{ \sum_{ \substack{ 1 \leq i_1 < \dots < i_{m-1} \leq n \\ 1 \leq j_1 < \dots < j_{m-1} \leq n\\
i_l, j_l \neq 1, 2, \text{ for } l \in [m-1]\\
|\{i_1, \dots, i_{m-1}\}\cap \{j_1, \dots, j_{m-1}\}| \geq 1
}}\bE[\xi_{b, 1} \xi_{b, 2} \bar{h}_{m, \{1, i_1, \dots, i_{m-1}\}}  \bar{h}_{m, \{2, j_1, \dots, j_{m-1}\}}]}_{\equiv EB} \notag\\
&
+\underbrace{  \sum_{ \substack{ 1 \leq i_1 < \dots < i_{m-2} \leq n \\ 1 \leq j_1 < \dots < j_{m-2} \leq n\\
i_l, j_l \neq 1, 2,  \text{ for } l \in [m-2]
}}\bE[\xi_{b, 1} \xi_{b, 2} \bar{h}_{m, \{1, 2, i_1, \dots, i_{m-2}\}}  \bar{h}_{m, \{1, 2, j_1, \dots, j_{m-2}\}}]}_{\equiv EC} \notag,
\end{align}
and will then bound each of $EA$, $EB$ and $EC$.

We start with $EA$, and it suffices to assume $m \geq 3$, otherwise one cannot expect  the two sets $\{i_1, \dots, i_{m-2}\}$ and $\{j_1, \dots, j_{m-1}\}$ indexing a given summand 
\[
\bE [ \xi_{b, 1} \xi_{b, 2} \bar{h}_{m, \{1, 2, i_1, \dots, i_{m-2}\}}  \bar{h}_{m, \{2, j_1, \dots, j_{m-1} \}}]
\]
of $EA$
to intersect for at least one element. Using the fact that the data $X_1, \dots, X_n$ are i.i.d., if the two index sets 
have $k \in [m-2]$ common elements not in the set $\{1, 2\}$, one can write the summand as
\begin{multline*}
\bE [ \xi_{b, 1} \xi_{b, 2} \bar{h}_{m, \{1, 2, i_1, \dots, i_{m-2}\}}  \bar{h}_{m, \{2, j_1, \dots, j_{m-1} \}}] = \\
\bE[\xi_{b, 1} \xi_{b, 2}\ \bar{h}_m(X_1, X_2, \dots, X_m)\ \bar{h}_m(X_2,X_3,\dots, X_{k+2}, X_{m+1},\dots, X_{2m-1-k})].
\end{multline*}
 From this, we can alternatively write 
\begin{multline*}
EA =  \\
\sum_{k=1}^{m-2}\binom{n-2}{k}\binom{n-2-k}{m-2-k} \binom{n-m}{m-1-k} 
  \bE\Big[\xi_{b, 1} \xi_{b, 2}\ \bar{h}_{m, [1:m]}\ \bar{h}_{m, [2: (k+2)]\cup[(m+1) : (2m-k-1)]}\Big];
\end{multline*}
from this, we can then form the bound
\begin{align*}
    &|EA| \\
    &\leq \sum_{k=1}^{m-2}\binom{n-2}{k}\binom{n-2-k}{m-2-k} \binom{n-m}{m-1-k}  \Big| \bE\Big[\xi_{b, 1} \xi_{b, 2}\ \bar{h}_{m, [1:m]}\ \bar{h}_{m, [2: (k+2)]\cup[(m+1) : (2m-k-1)]}\Big]\Big|\\
    &= \binom{n-2}{m-2}\sum_{k=1}^{m-2}\binom{m-2}{k}\binom{n-m}{m-1-k}  \Big| \bE\Big[\xi_{b, 1} \xi_{b, 2}\ \bar{h}_{m, [1:m]}\ \bar{h}_{m, [2: (k+2)]\cup[(m+1) : (2m-k-1)]}\Big]\Big|\\
        &  \hspace{10cm} \text{ by \lemref{counting_identities}$(ii)$ }\\
    &\leq \binom{n-2}{m-2}\sum_{k=1}^{m-2}\binom{m-2}{k}\binom{n-m}{m-1-k} \bigg\{\frac{9.5 \|g\|_3^2 \|h\|_3^2}{n} + \frac{2 k\|h\|_2}{n}\bigg\}\\
    &  \hspace{10cm} \text{ by \lemref{ustat_results} $(iii)$ } \\ 
    &= \binom{n-2}{m-2} \bigg\{ \bigg[\binom{n-2}{m-1} - \binom{n-m}{m-1} \bigg]\frac{9.5 \|g\|_3^2 \|h\|_3^2}{n}   + (m-2)\binom{n-3}{m-2}\frac{2 \|h\|_2}{n} \bigg\},
    \end{align*}
    where the last line comes from the equalities
    \begin{align*}
         \sum_{k=1}^{m-2}\binom{m-2}{k}\binom{n-m}{m-1-k} &=  \sum_{k=0}^{m-2}\binom{m-2}{k} \binom{n-m}{m-1-k}  - \binom{n-m}{m-1}\\
          &= \binom{n-2}{m-1} - \binom{n-m}{m-1}  \text{ by \lemref{counting_identities}$(i)$}
    \end{align*}
        and 
      \begin{align*}
      \sum_{k=1}^{m-2} k \binom{m-2}{k} \binom{n-m}{m-1-k} &= (m-2) \sum_{k=1}^{m-2}  \binom{m-3}{k -1} \binom{n-m}{m-1-k}  \\
      &= (m-2) \sum_{k=0}^{m-3} \binom{m-3}{k } \binom{n-m}{m-2-k}\\
      &= (m-2) \binom{n -3}{m-2}  \text{ coming from \lemref{counting_identities}$(i)$}
      \end{align*}
      Continuing, we get
    \begin{align}
    |EA|
    &\leq \binom{n-2}{m-2} \bigg\{  \binom{n-2}{m-1} \frac{9.5(m -2)(m-1)\|g\|_3^2 \|h\|_3^2}{(n -m)n}  +  (m-2)\binom{n-3}{m-2}\frac{2 \|h\|_2}{n} \bigg\} \notag\\
    & \hspace{9cm} \text{ by \lemref{counting_identities}$(iii)$} \notag\\
    &= \binom{n-1}{m-1}^2 \bigg\{ \frac{9.5(m-2)(m-1)^2   \|g\|_3^2 \|h\|_3^2}{(n-1)^2 n} + \frac{2 (m-1)^2 (m-2)(n-m) \|h\|_2}{n (n-1)^2 (n-2)}\bigg\} \notag\\
    & \hspace{8cm} \text{ by \eqref{enum_eq_1}, \eqref{enum_eq_2} and \eqref{enum_eq_3}} \notag\\
    &\leq  C\binom{n-1}{m-1}^2 \frac{m^3 \|g\|_3^2 \|h\|_3^2}{n^3} \label{EA_final_bound},
      \end{align} 
 where the last line uses $2m < n$, and $1= \sigma_g \leq \|h\|_2 \leq \|h\|_3$.

       Now we bound $EB$. Analogously to $EA$, we first write
 \begin{align*}
    |EB| &\leq \sum_{k=1}^{m-1}\binom{n-2}{k}\binom{n-2-k}{m-1-k}\binom{n-m-1}{m-1-k}\\
        &\qquad\quad\ \Big|\bE[\xi_{b, 1} \xi_{b, 2}\ \bar{h}_m(X_1, X_3, \dots, X_{m+1})\ \bar{h}_m(X_2,\underbrace{X_3,\dots, X_{k+2}}_{k\ \text{shared}}, X_{m+2},\dots, X_{2m-k})]\Big|\\
    &= \binom{n-2}{m-1}\sum_{k=1}^{m-1}\binom{m-1}{k}\binom{n-m-1}{m-1-k}\\  &\qquad\quad\ \Big|\bE[\xi_{b, 1} \xi_{b, 2}\ \bar{h}_m(X_1, X_3, \dots, X_{m+1})\ \bar{h}_m(X_2,\underbrace{X_3,\dots, X_{k+2}}_{k\ \text{shared}}, X_{m+2},\dots, X_{2m-k})]\Big|\\
    & \hspace{10cm} \text{ by \lemref{counting_identities} $(ii)$}\\
    &\leq  \binom{n-2}{m-1}\sum_{k=1}^{m-1}\binom{m-1}{k}\binom{n-m-1}{m-1-k}  \bigg(\frac{9.5 \|g\|_3^2 \|h\|_3^2}{n} + \frac{2k\|h\|_2}{n^{3/2}}\bigg) \text{ by \lemref{ustat_results}$(iv)$}\\
    &=\binom{n-2}{m-1}\bigg\{\bigg[\binom{n-2}{m-1}    - \binom{n -m - 1}{m-1}\bigg] \frac{9.5 \|g\|_3^2 \|h\|_3^2}{n} 
    + \binom{n -3}{m -2} \frac{2 (m-1)\|h\|_2}{n^{3/2}} \bigg\},
        \end{align*}
        where in the last equality, we have used
        \begin{align*}
         \sum_{k=1}^{m-1}\binom{m-1}{k}\binom{n-m-1}{m-1-k} &=  \sum_{k=0}^{m-1}\binom{m-1}{k}\binom{n-m-1}{m-1-k}- \binom{n-m-1}{m-1}\\
         &= \binom{n-2}{m-1} - \binom{n-m-1}{m-1}\text{ by \lemref{counting_identities}$(i)$ }
        \end{align*}
        and
        \begin{align*}
        \sum_{k=1}^{m-1}\binom{m-1}{k}\binom{n-m-1}{m-1-k} k 
        &= (m-1)     \sum_{k=1}^{m-1}\binom{m-2}{k - 1}\binom{n-m-1}{m-1-k}  \\
        &= (m-1)    \sum_{k=0}^{m-2} \binom{m-2}{k }\binom{n-m-1}{m-2-k} \\
        &=  (m-1)    \binom{n -3}{m -2} \text{ by \lemref{counting_identities}$(i)$}
        \end{align*}
        Continuing, we get
        \begin{align}
        |EB| &\leq \binom{n-2}{m-1}\bigg\{ \binom{n-2}{m-1} \frac{(m-1)^2}{n-m}\frac{9.5 \|g\|_3^2 \|h\|_3^2}{n} 
    + \binom{n -3}{m -2} \frac{2 (m-1)\|h\|_2}{n^{3/2}} \bigg\} \notag\\
    & \hspace{8cm}\text{ by \lemref{counting_identities}$(iii)$} \notag\\
    &=   \binom{n-1}{m-1}^2 \bigg\{ \frac{9.5(m-1)^2(n-m) \|g\|_3^2 \|h\|_3^2}{ (n-1)^2n}
    +\frac{2(m-1)^2(n-m)^2\|h\|_2}{(n-1)^2(n-2)n^{3/2}} \bigg\}
   \notag \\
    & \hspace{8cm}\text{ by \eqref{enum_eq_1} and \eqref{enum_eq_3}} \notag\\
    &\leq C\binom{n-1}{m-1}^2 \frac{ m^2\|g\|_3^2 \|h\|_3^2}{n^2} \label{EB_final_bdd},
        \end{align}
   where the last line uses $2m < n$, and $1= \sigma_g \leq \|h\|_2 \leq \|h\|_3$.
   
      Lastly, for $EC$, in an analogous manner as $EA$ and $EB$, we first write it as 
        \begin{align*}
      &EC = \sum_{k=0}^{m-2}\binom{n-2}{k}\binom{n-2-k}{m-2-k}\binom{n-m}{m-2-k}\\
    & \hspace{1cm}\bE[\xi_{b, 1} \xi_{b, 2}\ \bar{h}_m(X_1, X_2, \dots, X_{m})\ \bar{h}_m(X_1, X_2,\underbrace{X_3,\dots, X_{k+2}}_{k\ \text{shared, empty if }k=0}, X_{m+1},\dots, X_{2m-k-2})].
        \end{align*}       
        Then we can bound
\begin{align*}
    |EC| &\leq \sum_{k=0}^{m-2}\binom{n-2}{k}\binom{n-2-k}{m-2-k}\binom{n-m}{m-2-k}\\
    &\qquad\quad\ \Big|\bE[\xi_{b, 1} \xi_{b, 2}\ \bar{h}_m(X_1, X_2, \dots, X_{m})\ \bar{h}_m(X_1, \dots, X_{k+2}, X_{m+1},\dots, X_{2m-k-2}]\Big|\\
    &\leq \binom{n-2}{m-2} \sum_{k=0}^{m-2}\binom{m-2}{k}\binom{n-m}{m-2-k} \\
    &\qquad\quad\ \Big|\bE[\xi_{b, 1} \xi_{b, 2}\ \bar{h}_m(X_1, X_2, \dots, X_{m})\ \bar{h}_m(X_1,\dots, X_{k+2}, X_{m+1},\dots, X_{2m-k-2}]\Big|\\
    & \hspace{9cm}\text{by \lemref{counting_identities}$(ii)$ } \\ 
    &\leq \binom{n-2}{m-2} \sum_{k=0}^{m-2} \binom{m-2}{k}\binom{n-m}{m-2-k}\bigg(\frac{9.5 \|g\|_3^2 \|h\|_3^2}{n} + \frac{2k\|h\|_2}{n}\bigg) \text{ by \lemref{ustat_results}$(iii)$}\\
    &= \binom{n-2}{m-2} \bigg\{ \binom{n-2}{m-2}\frac{9.5 \|g\|_3^2 \|h\|_3^2}{n} +   \binom{n-3}{m-3}\frac{2(m-2)\|h\|_2}{n}\bigg\},
    \end{align*}
where the last equality comes from
\[
 \sum_{k=0}^{m-2}  \binom{m-2}{k}\binom{n-m}{m-2-k} =  \binom{n-2}{m-2} \text{ by  \lemref{counting_identities}$(i)$}
\]
and for $m \geq 3$,
\begin{align*}
\sum_{k=0}^{m-2} \binom{m-2}{k}\binom{n-m}{m-2-k} k &= \sum_{k=1}^{m-2} \binom{m-2}{k}\binom{n-m}{m-2-k} k \\
&= (m-2)\sum_{k=1}^{m-2} \binom{m-3}{k-1}\binom{n-m}{m-2-k} \\
&= (m-2)\sum_{k=0}^{m-3} \binom{m-3}{k}\binom{n-m}{m-3-k} \\
&= (m-2) \binom{n-3}{m-3} \text{ by \lemref{counting_identities}$(i)$}. 
\end{align*}
Continuing, we get by \eqref{enum_eq_2} and \eqref{enum_eq_4},
\begin{align}
|EC| &\leq   \binom{n-2}{m-2} \bigg\{ \binom{n-2}{m-2}\frac{9.5 \|g\|_3^2 \|h\|_3^2}{n} +   \binom{n-3}{m-3}\frac{2(m-2)\|h\|_2}{n}\bigg\}\notag\\
&=  \binom{n-1}{m-1}^2 \bigg\{ \frac{9.5 (m-1)^2\|g\|_3^2 \|h\|_3^2}{n (n-1)^2} +\frac{2(m-1)^2(m-2)^2\|h\|_2}{n(n-1)^2(n-2)}\bigg\} \notag\\
&\leq C \binom{n-1}{m-1}^2  \bigg\{ \frac{m^2\|g\|_3^2 \|h\|_3^2}{n^3}  + \frac{m^4 \|h\|_2}{n^4}\bigg\} \label{EC_final_bdd}
\end{align}

Substituting \eqref{EA_final_bound}, \eqref{EB_final_bdd} and \eqref{EC_final_bdd} into \eqref{exp2_expand}, we get that
\begin{multline} \label{exp2_bdd_final}
\Bigg|
\bE
\bigg[
\bigg(\xi_{b, 1} \sum_{\substack{1 \leq i_1 < \dots < i_{m-1} \leq n\\ i_l \neq 1 \text{ for } l \in [m-1]}}  \bar{h}_{m, \{1, i_1, \dots, i_{m-1}\}}\bigg)  \bigg(
\xi_{b, 2} \sum_{\substack{1 \leq j_1 < \dots < j_{m-1} \leq n\\ j_l \neq 2 \text{ for } l \in [m-1]}}  \bar{h}_{m, \{2, j_1, \dots, j_{m-1}\}})
\bigg)\bigg]
\Bigg| \\
\leq C\binom{n-1}{m-1}^2  \frac{m^2\|g\|_3^2 \|h\|_3^2}{n^2}, 
\end{multline}
where we have used that $2m < n$ and $1 = \|g\|_2 \leq \|h\|_2 \leq \|h\|_3$. Finally, collecting \eqref{Pi2_first_bdd}, \eqref{exp1}, \eqref{exp2},  \eqref{exp1_bdd_final} and \eqref{exp2_bdd_final}, we obtain \eqref{Pi22_bdd}.

\subsection{Proof of  \lemref{Djn_minus_D_ni}$(ii)$}  Note that
\[
\delta_{2n, b} -  \delta_{2n, b}^{(i)} = A + B,  
\]
where
\[
A = \frac{2 (n-1) }{\sqrt{n}(n-m)} {n-1 \choose m-1}^{-1} \xi_{b, i} \sum_{\substack{1 \leq i_1 < \dots < i_{m-1} \leq n \\
i_l \neq i \text{ for } l \in [m-1] }} \bar{h}_m (X_i, X_{i_1}, \dots, X_{i_{m-1}})
\]
and
\[
B = \frac{2 (n-1) }{\sqrt{n}(n-m)} {n-1 \choose m-1}^{-1} \sum_{\substack{1  \leq j \leq n \\ j \neq i}} \Bigg(\xi_{b, j} \sum_{\substack{1 \leq i_1 < \dots < i_{m-2} \leq n\\ i_l \neq j, i \text{ for } l = 1, \dots, m-2}}
\bar{h}_m (X_j, X_i, X_{i_1}, \dots, X_{i_{m-2}}) \Bigg).
\]
From \eqref{Pi2_for_ustat} and \eqref{Pi2_for_ustati_def}, we first write 
\begin{align} 
\|\Pi_2 - \Pi_2^{(i)} \|_2 &\leq  \bE[ (\xi_i^2 - 1)I(|\xi_i|> 1) ] +   \|\delta_{2n, b} -  \delta_{2n, b}^{(i)}\|_2 \notag\\
&\leq  \frac{\bE[g^2]}{n} + \| A\|_2 + \|B\|_2, \label{Djn_minus_D_ni_first_bdd}
\end{align}
by \lemref{exp_xi_bi_bdd}, where
\[
A = \frac{2 (n-1) }{\sqrt{n}(n-m)} {n-1 \choose m-1}^{-1} \xi_{b, i} \sum_{\substack{1 \leq i_1 < \dots < i_{m-1} \leq n \\
i_l \neq i \text{ for } l \in [m-1] }} \bar{h}_m (X_i, X_{i_1}, \dots, X_{i_{m-1}})
\]
and
\[
B = \frac{2 (n-1) }{\sqrt{n}(n-m)} {n-1 \choose m-1}^{-1} \sum_{\substack{1  \leq j \leq n \\ j \neq i}} \Bigg(\xi_{b, j} \sum_{\substack{1 \leq i_1 < \dots < i_{m-2} \leq n\\ i_l \neq j, i \text{ for } l = 1, \dots, m-2}}
\bar{h}_m (X_j, X_i, X_{i_1}, \dots, X_{i_{m-2}}) \Bigg).
\]
So we will bound $\|A\|_2$ and $\|B\|_2$, which is trivial for $m=1$ as $\bar{h}_1(\cdot) = 0$. For $m \geq 2$, by \lemref{ustat_results}$(ii)$,
\begin{multline}
\bE[A^2] \leq \frac{4 (n-1)^2 }{n(n-m)^2} {n-1 \choose m-1}^{-2}
\bE\Bigg[\Bigg( \sum_{\substack{1 \leq i_1 < \dots < i_{m-1} \leq n  \\
i_l \neq i \text{ for } l \in [m-1] }} \bar{h}_m (X_i, X_{i_1}, \dots, X_{i_{m-1}})\Bigg)^2\Bigg] \\
\leq  \frac{8 (n-1)^2 (m-1)^2 \bE[h^2]}{(n-m)^2n(n-m+1)m }    \leq C \frac{m\bE[h^2] }{n^2} \label{A_sq_final_bdd}
\end{multline}
Moreover, for $B$, we first expand its second moment as
\begin{align}
&\bE[B^2] \notag\\
&= \frac{4 (n-1)^2 }{n(n-m)^2} {n-1 \choose m-1}^{-2} \bE\Bigg[ \Bigg( \sum_{\substack{1  \leq j \leq n \\ j \neq i}} \Bigg(\xi_{b, j} 
\sum_{
\substack{1 \leq i_1 < \dots < i_{m-2} \leq n\\ i_l \neq j, i \text{ for } l = 1, \dots, m-2}
}
\bar{h}_{m, \{j, i, i_1, \dots, i_{m-2}\}} \Bigg)\Bigg)^2\Bigg] \notag\\
&=  \frac{4 (n-1)^2 }{n(n-m)^2} {n-1 \choose m-1}^{-2} \times \notag\\
& 
\quad \Bigg\{ (n-1)
\underbrace{
\sum_{
\substack{
 1 \leq i_1 < \dots < i_{m-2} \leq m-2\\ 
 1 \leq j_1 < \dots < j_{m-2} \leq m-2\\
 i_l, j_l \neq 1, 2 \text{ for } l \in [m-2]
}}
\bE[\xi_{b, 1}^2  \bar{h}_{m, \{1, 2, i_1, \dots, i_{m-2}\}} \bar{h}_{m, \{1, 2, j_1, \dots, j_{m-2}\}}] 
}_{\equiv ED}
+
\notag\\
&\quad (n-1)(n-2) 
\underbrace{
\sum_{\substack{1 \leq i_1 < \dots < i_{m-2} \leq n \\ 1 \leq j_1 < \dots < j_{m-2} \leq n \\
i_l \neq 1, 3 \text{ for } l \in [m-2]\\  j_l \neq 2, 3 \text{ for } l \in [m-2]}} 
\bE[\xi_{b, 1} \xi_{b, 2} \bar{h}_{m, \{1, 3, i_1, \dots, i_{m-2}\}}  \bar{h}_{m, \{2, 3, j_1, \dots, j_{m-2}\}}]
}_{\equiv EE}
\Bigg\} \label{B_sq_expansion}.
\end{align}
To bound $ED$, we first note that, by $|\xi_{b, 1}| \leq 1$, H\"older's inequality and \lemref{ustat_results}$(i)$, each of its summand can be bounded as
\begin{equation} \label{first_part_B_sq}
\Big|\bE[\xi_{b, 1}^2 \bar{h}_{m, \{1, 2, i_1, \dots, i_{m-2}\}} \bar{h}_{m, \{1, 2, j_1, \dots, j_{m-2}\}}]\Big| \leq \bE[h^2]
\end{equation}
Then, by considering the number of elements $k \in [m-2]$ shared by the sets $\{i_1, \dots, i_{m-2}\}$ and $\{j_1, \dots, j_{m-2}\}$ indexing each such summand, we have the bound
\begin{align}
|ED| &\leq  \sum_{k =0}^{m-2} \binom{n-2}{k}\binom{n-2-k}{m-2-k} \binom{n-m}{m-2-k} \bE[h^2] \notag\\
&= \binom{n-2}{m-2} \sum_{k =0}^{m-2} \binom{m-2}{k} \binom{n-m}{m-2-k} \bE[h^2] \text{ by \lemref{counting_identities}$(ii)$} \notag\\
&= \binom{n-2}{m-2}^2 \bE[h^2] \text{ by \lemref{counting_identities}$(i)$} \notag\\
&= \binom{n-1}{m-1}^2 \bigg(\frac{m-1}{n-1}\bigg)^2  \bE[h^2]\text{ by \eqref{enum_eq_2}} \label{ED_final_bdd}.
\end{align}
To bound $EE$, we  first break it down as
\begin{align}
&EE = \label{EE_decomposition}\\
&\underbrace{
\sum_{\substack{1 \leq i_1 < \dots < i_{m-2} \leq n \\ 1 \leq j_1 < \dots < j_{m-2} \leq n \\
i_l \neq 1, 2, 3 \text{ for } l \in [m-2]\\  j_l \neq 1, 2, 3 \text{ for } l \in [m-2]}}  
\bE[\xi_{b, 1} \xi_{b, 2}\bar{h}_{m, \{1,  3,  i_1, \dots, i_{m-2}\}} \bar{h}_{m, \{2, 3, j_1, \dots, j_{m-2}\}}}_{\equiv EE_1}\notag\\
&+\underbrace{
\sum_{\substack{1 \leq i_1 < \dots < i_{m-3} \leq n \\ 1 \leq j_1 < \dots < j_{m-2} \leq n \\
i_l \neq 1, 2, 3 \text{ for } l \in [m-3]\\  j_l \neq 1, 2, 3 \text{ for } l \in [m-2]}}  \bE[\xi_{b, 1}\xi_{b, 2} \bar{h}_{m, \{1, 2, 3, i_1, \dots, i_{m-3}\}} \bar{h}_{m, \{2, 3, j_1, \dots, j_{m-2}\}}] }_{\equiv EE_2} \notag\\
&+ 
\underbrace{
\sum_{\substack{1 \leq i_1 < \dots < i_{m-2} \leq n \\ 1 \leq j_1 < \dots < j_{m-3} \leq n \\
i_l \neq 1, 2, 3 \text{ for } l \in [m-2]\\  j_l \neq 1, 2, 3 \text{ for } l \in [m-3]}}  
\bE[\xi_{b, 1} \xi_{b, 2} \bar{h}_{m, \{1, 3,  i_1, \dots, i_{m-2}\}} \bar{h}_{m, \{1, 2, 3, j_1, \dots, j_{m-3}\}}]}_{\equiv EE_3} \notag\\
&+ 
\underbrace{
\sum_{\substack{1 \leq i_1 < \dots < i_{m-3} \leq n \\ 1 \leq j_1 < \dots < j_{m-3} \leq n \\
i_l \neq 1, 2, 3 \text{ for } l \in [m-3]\\  j_l \neq 1, 2, 3 \text{ for } l \in [m-3]}}  
\bE[\xi_{b, 1} \xi_{b, 2} \bar{h}_{m, \{1, 2, 3,  i_1, \dots, i_{m-3}\}} \bar{h}_{m, \{1, 2, 3, j_1, \dots, j_{m-3}\}} }_{\equiv EE_4}. \notag
\end{align}
Using \lemref{ustat_results}$(iv)$, one can then bound $EE_1$ as
\begin{align}
&|EE_1| \notag\\
 &\leq \sum_{k=0}^{m-2} \binom{n-3}{k} \binom{n-3-k}{m-2-k} \binom{n-1 -m}{m-2-k} \bigg( \frac{9.5 \|g\|_3^2 \|h\|_3^2}{n} + \frac{2d\|h\|_2}{n^{3/2}}\bigg) \notag\\
&\leq \sum_{k=0}^{m-2} \binom{n-3}{k} \binom{n-3-k}{m-2-k} \binom{n-1 -m}{m-2-k} \bigg( \frac{9.5 \|g\|_3^2 \|h\|_3^2}{n} + \frac{2\|h\|_2^2}{n}\bigg)\notag\\
& \hspace{7cm} \text{ by \eqref{extract_extra_m} and $d \leq m \leq n$} \notag\\
&\leq  11.5 \binom{n-3}{m-2} \sum_{k=0}^{m-2} \binom{m-2}{k} \binom{n-1 -m}{m-2-k}  \frac{\|g\|_3^2 \|h\|_3^2}{n} \notag\\
& \hspace{6cm}\text{ by \lemref{counting_identities}$(ii)$ and  $\|h\|_2 \leq \|h\|_3$}\notag\\
&= 11.5 \binom{n-3}{m-2}^2 \frac{\|g\|_3^2 \|h\|_3^2}{n} \text{ by \lemref{counting_identities}$(i)$}  \notag\\
&=  11.5 \binom{n-1}{m-1}^2 \frac{(m-1)^2(n-m)^2\|g\|_3^2 \|h\|_3^2}{n(n-1)^2(n-2)^2 } \text{ by \eqref{enum_eq_3}}  \label{EE1_final_bdd}.
\end{align}
For $EE_2$ and $EE_3$, using \lemref{ustat_results}$(iii)$, one can bound them similarly as
\begin{align} 
&\max(|EE_2|, |EE_3|)  \notag\\
&\leq \sum_{k =0}^{m-3} \binom{n-3}{k} \binom{n-3-k}{m-3-k} \binom{n-m}{m-2-k}\bigg(  \frac{9.5 \|g\|_3^2 \|h\|_3^2}{n}+ \frac{2(2+k)\|h\|_2}{n}\bigg)\notag\\
&= 
 \binom{n-3}{m-3}  \sum_{k =0}^{m-3}\binom{m-3}{k} \binom{n-m}{m-2-k} \bigg(  \frac{9.5 \|g\|_3^2 \|h\|_3^2}{n}+ \frac{2(2+k)\|h\|_2}{n}\bigg)\text{ by \lemref{counting_identities}$(ii)$} \notag\\
&= 
\binom{n-3}{m-3} \binom{n-3}{m-2} \frac{9.5\|g\|_3^2\|h\|_3^2 + 4 \|h\|_2}{n} \notag\\
& \hspace{1cm}
+ 
 \binom{n-3}{m-3} \sum_{k =1}^{m-3}  \binom{m-4}{k-1} \binom{n-m}{m-2-k} \frac{2 (m-3)\|h\|_2 }{n}  \text{ by \lemref{counting_identities}$(i)$}
\notag\\
&= 
\binom{n-3}{m-3}\bigg\{ \binom{n-3}{m-2} \frac{9.5\|g\|_3^2\|h\|_3^2 + 4 \|h\|_2}{n}
+ 
\sum_{k =0}^{m-4}  \binom{m-4}{k} \binom{n-m}{m-3-k} \frac{2 (m-3)\|h\|_2 }{n}  \bigg\} \notag\\
&= 
\binom{n-3}{m-3}\bigg\{ \binom{n-3}{m-2} \frac{9.5\|g\|_3^2\|h\|_3^2 + 4 \|h\|_2}{n}
+ 
\binom{n-4}{m-3} \frac{2 (m-3)\|h\|_2 }{n}  \bigg\} \text{by \lemref{counting_identities}$(i)$}\notag\\
&= 
\binom{n-3}{m-3}\bigg\{ \binom{n-3}{m-2} \frac{9.5\|g\|_3^2\|h\|_3^2 + 4 \|h\|_2}{n}
+ 
\binom{n-3}{m-3} \frac{2 (m-3)(n-m)\|h\|_2 }{(n-3)n}  \bigg\}\notag\\
&= \binom{n-1}{m-1}^2 \bigg\{ \frac{(m-1)^2(m-2)(n-m)(9.5\|g\|_3^2\|h\|_3^2 + 4 \|h\|_2)}{(n-1)^2(n-2)^2 n} \notag\\
&\hspace{4cm}
+ \frac{2(m-1)^2(m-2)^2(m-3)(n-m)\|h\|_2}{(n-1)^2(n-2)^2(n-3)n}
\bigg\} \text{ by \eqref{enum_eq_3} and \eqref{enum_eq_4}}\notag\\
&\leq C \binom{n-1}{m-1}^2 \bigg\{ \frac{m^3 \|g\|_3^2\|h\|_3^2 }{n^4} + \frac{ m^5\|h\|_2}{n^5} \bigg\}
\text{by $1 \leq \|g\|_3$ and $\|h\|_2 \leq \|h\|_3$}. \label{EE2_EE3_final_bdd}
 \end{align}
Lastly, for $EE_4$, using \lemref{ustat_results}$(iii)$, one can bound it as
\begin{align}
&|EE_4| \notag\\
&\leq  \sum_{k =0}^{m-3} {n -3 \choose k} {n -3  - k \choose m - 3 -  k}  {n -m \choose m - 3 -  k}  \bigg(  \frac{9.5 \|g\|_3^2 \|h\|_3^2}{n}+ \frac{2(3+k)\|h\|_2}{n}\bigg) \notag\\
&= {n -3 \choose m -3} \sum_{k =0}^{m-3} {m-3 \choose k} {n -m \choose m - 3 -  k} \bigg(  \frac{9.5 \|g\|_3^2 \|h\|_3^2}{n}+ \frac{2(3+k)\|h\|_2}{n}\bigg) \text{ by \lemref{counting_identities} $(ii)$} \notag\\
&= {n -3 \choose m -3}
 \bigg\{ 
{n-3 \choose m-3}\frac{9.5 \|g\|_3^2 \|h\|_3^2 + 6 \|h\|_2}{n} +
\frac{2 (m-3)\|h\|_2}{n}
\sum_{k =1}^{m-3} {m-4 \choose k-1} {n-m \choose m - 4 - (k-1)}
\bigg\}\notag\\
&= {n -3 \choose m -3}
 \bigg\{ 
{n-3 \choose m-3}\frac{9.5 \|g\|_3^2 \|h\|_3^2 + 6 \|h\|_2}{n} +
\frac{2 (m-3)\|h\|_2}{n} {n -4 \choose m - 4} \bigg\} \text{ by \lemref{counting_identities} $(i)$} \notag\\
&\leq C{n -1 \choose m-1}^2 \bigg\{ \frac{ m^4 \|g\|_3^2 \|h\|_3^2  }{n^5 } +
\frac{m^6 \|h\|_2}{n^6}
\bigg\} \text{by \eqref{enum_eq_4}, \eqref{enum_eq_5}, $1 \leq \|g\|_3$ and $\|h\|_2 \leq \|h\|_3$}. \label{EE4_final_bdd}
\end{align}
Combining \eqref{EE_decomposition}, \eqref{EE1_final_bdd}, \eqref{EE2_EE3_final_bdd}, \eqref{EE4_final_bdd} and $2m <n$, we get that
\begin{equation}
 \label{EE_final_bdd}
|EE| \leq C {n-1 \choose m-1}^2 \bigg\{ \|g\|_3^2 \|h\|_3^2 \bigg(\frac{m^2}{n^3} + \frac{m^3}{n^4} + \frac{m^4}{n^5}\bigg) + \|h\|_2 \bigg(\frac{m^5}{n^5}  \bigg) \bigg\}.
\end{equation}
Combining \eqref{B_sq_expansion}, \eqref{ED_final_bdd} and \eqref{EE_final_bdd}, we get
\begin{align} \bE[B^2] &
\leq 
 C   
  \bigg\{\frac{m^2}{n^2} \bE[h^2] + \bigg[ \|g\|_3^2 \|h\|_3^2 \bigg(\frac{m^2}{n^2} + \frac{m^3}{n^3} + \frac{m^4}{n^4}\bigg) + \|h\|_2 \bigg(\frac{m^5}{n^4}  \bigg) \bigg] \bigg\} \notag\\
  &\leq  
   C   
  \bigg\{ \frac{m^2 \|g\|_3^2 \|h\|_3^2}{n^2}  +\frac{m^5  \|h\|_2 }{n^4}   \bigg\},
  \label{B_sq_final_bdd}
\end{align}
where we have used $2m < n$, as well as $\|h\|_2 \leq \|h\|_3$ and $1 = \|g\|_2 \leq \|g\|_3$ in the last line. Combining \eqref{Djn_minus_D_ni_first_bdd}, \eqref{A_sq_final_bdd} and \eqref{B_sq_final_bdd} gives \lemref{Djn_minus_D_ni}$(ii)$. 

\section{Proof of \lemsref{ustat_results} and \lemssref{counting_identities}}
\subsection{Proof of \lemref{ustat_results} } \label{app:pf_of_ustat_results}

The proof for $(i)$ and $(ii)$ can be found in \citet[Ch.10, Appendix]{chen2011normal}. We will focus on proving $(iii)$ and $(iv)$. For any subset $ \{i_1, \dots, i_k\} \subset [n]$, we will denote
\[
X_{\{i_1, \dots, i_k\}} = \{X_{i_1}, \dots, X_{i_k}\}.
\] 
To simplify the notation, we also denote 
\[
I = \{i_1, \dots, i_{k_1}\} \text{ and } J = \{j_1, \dots, j_{k_2}\},
\]
as well as
\[
h_I  = h_{k_1}(X_{i_1}, \dots, X_{i_{k_1}}) \text{ and } \bar{h}_I  = \bar{h}_{k_1, \{i_1, \dots, i_{k_1}\} }
\]
and
\[
h_J  = h_{k_2}(X_{j_1}, \dots, X_{j_{k_2}}) \text{ and } \bar{h}_I  = \bar{h}_{k_2, \{j_1, \dots, i_{j_2}\} }.
\]
First, it suffices to assume both 
\begin{equation*}\label{working_assumption_1}
k_1, k_2 \geq 2
\end{equation*}
 because if any of $k_1$ and $k_2$ is equal to $1$, then one of $\bar{h}_{k_1, \{ i_1, \dots, i_{k_1}\}}$ and $\bar{h}_{k_2, \{j_1, \dots, j_{k_2}\}}$ must be equal to zero by the definition in  \eqref{h_k_def}, so the bound is trivial. Moreover, one can further assume without loss of generality that the index sets $I$ and $J$ are such that 
 \begin{equation}\label{working_assumption_2}
I \backslash \{1, 2\} =  J \backslash \{1, 2\}  = [3: (d+2)] \text{ if } d >0,
 \end{equation}
 in which case it must be true that $|I \backslash \{1, 2\}| = |J \backslash \{1, 2\} | = d$. 
This is because for any $I$ and $J$, we have
\begin{align*}
&\bE[ \xi_{b, 1} \xi_{b, 2}\bar{h}_I  \bar{h}_J   ] \\
&=  \bE\Big[\bE[ \xi_{b, 1} \xi_{b, 2}\bar{h}_I  \bar{h}_J  \mid X_{\{1, 2\}\cup (I \cap J)}  ] \Big]\\
&=  \bE\Big[\xi_{b, 1} \xi_{b, 2} \bE[ \bar{h}_I  \bar{h}_J  \mid X_{\{1, 2\}\cup (I \cap J)}  ] \Big] \\
&=  \bE\Big[\xi_{b, 1} \xi_{b, 2} \bE[ \bar{h}_I    \mid X_{\{1, 2\}\cup (I \cap J)}  ] \bE[\bar{h}_J \mid X_{\{1, 2\}\cup (I \cap J)}]\Big]\\
& \hspace{1cm} \text{ because $I \backslash \Big(\{1, 2\}\cup (I \cap J)\Big)$ and $J \backslash \Big( \{1, 2\}\cup (I \cap J)\Big)$ are disjoint}\\
&=\bE\Big[\xi_{b, 1} \xi_{b, 2} \bE[ \bar{h}_I    \mid X_{(I \cap\{1, 2\})\cup (I \cap J)}  ] \bE[\bar{h}_J \mid X_{(J\cap\{1, 2\})\cup (I \cap J)}]\Big] \\
&= \bE\Big[\xi_{b, 1} \xi_{b, 2}  \bar{h}_{(I \cap\{1, 2\})\cup (I \cap J)}   \bar{h}_{(J\cap\{1, 2\})\cup (I \cap J)}\Big].
\end{align*}
Since 
\[
\Big((I \cap\{1, 2\})\cup (I \cap J) \Big)\backslash \{1, 2\} = (I \cap J) \backslash \{1, 2\} = \Big((J \cap\{1, 2\})\cup (I \cap J) \Big)\backslash \{1, 2\}
\]
 and 
 \[
 |(I \cap J) \backslash \{1, 2\}| = d \text{ by assumption},
 \] by the i.i.d.'ness of the data $X_1, \dots, X_n$ it suffices to assume \eqref{working_assumption_2}.

By the definition in \eqref{h_k_def},   we  perform the expansion
\begin{align*}
    &\bE[\xi_{b,1}\xi_{b,2}\ \bar{h}_{I}\ \bar{h}_{J}]\\
    &= \bE\Big[\xi_{b,1}\xi_{b,2}\Big(h_{I} - \sum_{i\in I\cap\{1,2\}}g(X_i)- \sum_{i\in I\setminus\{1,2\}}g(X_i)\Big)\Big(h_{J} -\sum_{j\in J\cap\{1,2\}}g(X_j)- \sum_{j\in J\setminus\{1,2\}}g(X_j)\Big)\Big]\\
    &= \underbrace{\bE[\xi_{b,1}\xi_{b,2}\ h_{I}\ h_{J}]}_{\equiv HH} \\
    &\quad - \underbrace{\sum_{i\in I\cap \{1,2\}}\bE[\xi_{b,1}\xi_{b,2}g(X_i)\ h_{J}]}_{\equiv GH_{1}} - \underbrace{\sum_{j\in J\cap \{1,2\}}\bE[\xi_{b,1}\xi_{b,2}g(X_j)\ h_{I}]}_{\equiv GH_{2}}\\
    &\quad - \underbrace{ \sum_{ i \in I \backslash \{1, 2\} }\bE[\xi_{b,1}\xi_{b,2}g(X_i)\ h_J]}_{\equiv GH_{3}} - \underbrace{\sum_{ j \in J \backslash \{1, 2\} }\bE[\xi_{b,1}\xi_{b,2}g(X_j)\ h_I]}_{\equiv GH_{4}}\\
    &\quad + \underbrace{\sum_{i\in I\cap \{1,2\}}\sum_{j\in J\cap \{1,2\}}\bE[\xi_{b,1}\xi_{b,2}g(X_i)g(X_j)]}_{\equiv GG_1} + \underbrace{ \sum_{i \in I \backslash \{1, 2\}} \sum_{j \in J \backslash \{1, 2\}} \bE[\xi_{b,1}\xi_{b,2} g(X_i) g(X_j))]}_{\equiv GG_2}, 
\end{align*}
recognizing that the last batch of expansion terms 
\[
\sum_{i\in I\cap \{1,2\}} \sum_{j\in J\backslash \{1,2\}}\underbrace{\bE[\xi_{b,1}\xi_{b,2}g(X_i)g(X_j)]}_{=\bE[\xi_{b,1}\xi_{b,2}g(X_i)]\bE[g(X_j)] = 0} + \sum_{i\in I\backslash \{1,2\}} \sum_{j\in J\cap \{1,2\}}\underbrace{\bE[\xi_{b,1}\xi_{b,2}g(X_i)g(X_j)]}_{ \bE[\xi_{b,1}\xi_{b,2}g(X_j)]\bE[g(X_i)]=0}
\]
vanish. The remaining  terms in each row of the expansion above are bounded as follows:

\subsubsection{Bound on $HH$:}
\begin{align}
    |HH| &= \Big| \bE[\xi_{b, 1} \xi_{b, 2}\ h_{I}\ h_{J}]\Big| \leq \Big\|\xi_{b, 1} \xi_{b,2}\Big\|_3 \Big\| h_{I}\ h_{J}\Big\|_{3/2} \notag\\
&=\Big(\bE[|\xi_{b, 1} |^3] \bE[|\xi_{b, 2} |^3]\Big)^{1/3} 
\Big( \bE\Big[ \Big| h_{I}\Big|^{3/2} \Big| h_{J}\Big|^{3/2}\Big]\Big)^{2/3} \notag\\
&\leq  \Big(\bE[|\xi_{b, 1} |^3] \bE[|\xi_{b, 2} |^3]\Big)^{1/3}   \left(\| |h_{I}|^{3/2}\|_2 \| |h_{J}|^{3/2}\|_2 \right)^{2/3}     \text{ by Cauchy's inequality}\notag\\
&\leq n^{-1}\|g\|_3^2 \|h\|_3^2, \label{HH_final_bdd}
\end{align}
where the last line come from \eqref{Jensen} with $|I|\vee |J| \leq m$.

\subsubsection{Bound on $GH_1 + G H_2$:}
\begin{align}
    &|GH_1+GH_2|  \notag\\
    &\leq \sum_{i \in I \cap \{1, 2\}} \| \xi_{b,1}\xi_{b,2}g(X_i) \|_{3/2} \| h_J \|_{3} + \sum_{j \in J \cap \{1, 2\}} \| \xi_{b,1}\xi_{b,2}g(X_j) \|_{3/2} \| h_I \|_{3} \notag\\
    &= |I \cap \{1, 2\}| \cdot \| \xi_{b,1}\xi_{b,2}g(X_1) \|_{3/2} \| h_J \|_{3} + |J \cap \{1, 2\}| \cdot  \| \xi_{b,1}\xi_{b,2}g(X_1) \|_{3/2} \| h_I \|_{3} \notag\\
    &\leq 4  \| \xi_{b,1}\xi_{b,2}g(X_1) \|_{3/2}  \|h\|_3  \quad \text{ by } \eqref{Jensen} \notag\\
    &= 4 \|\xi_{b, 1} g(X_1)\|_{3/2}\|\xi_{b, 2}\|_{3/2} \|h\|_3 \quad \text{ by independence } \notag \\
    &\leq 4 \Big(\bE[n^{-3/4}|g(X_1)|^3]\Big)^{2/3}\Big(\bE[n^{-3/4}|g(X_2)|^{3/2}]\Big)^{2/3} \|h\|_3 \notag\\
    &= 4n^{-1} \|g\|_3^2 \|g\|_{3/2} \|h\|_3 \notag\\
    &\leq 4n^{-1}\|g\|_3^2 \|h\|_3, \label{GH_1_plus_G H_2_final_bdd} 
\end{align}
where the last inequality is true because  $\|g\|_{3/2}  \leq \|g\|_2  = \sigma_g = 1$.

\subsubsection{General bound on $GH_3  + G H_4$:}

\begin{align}
   & |GH_3+GH_4| \notag\\
   & \leq  \sum_{ i \in I \backslash \{1, 2\} } \|\xi_{b,1}\xi_{b,2}g(X_i)\|_2 \| h_J\|_2+ \sum_{ j \in J \backslash \{1, 2\} } \|\xi_{b,1}\xi_{b,2}g(X_j)\|_2 \| h_I\|_2 \notag\\
    &= |I \backslash \{1, 2\}| \cdot \|\xi_{b,1}\xi_{b,2}g(X_3)\|_2 \|h_J\|_2 + |J \backslash \{1, 2\}| \cdot \|\xi_{b,1}\xi_{b,2}g(X_3)\|_2 \|h_I\|_2 \notag \\
    &\leq 2d \|\xi_{b,1}\xi_{b,2}g(X_3)\|_2 \|h\|_2 \text{ by } \eqref{Jensen} \text{ and } \eqref{working_assumption_2} \notag\\
    &\leq 2 d\|\xi_1\|_2 \|\xi_2\|_2  \|g(X_3)\|_2   \|h\|_2 \text{ by independence } \notag\\
    &= 2dn^{-1}\|h\|_2 \text{ by  \eqref{unit_var}}.  \label{GH_3_plus_GH_4_final_general_bdd}
\end{align}

\subsubsection{Special bound on $GH_3  + G H_4$ under $1 \not \in J$ and $2 \not \in I$:}
\begin{align}
    &|GH_3+GH_4| \notag\\
    &= \bigg|\sum_{ i \in I \backslash \{1, 2\} }\bE[\xi_{b,1}] \bE[\xi_{b,2}g(X_i)\ h_J] + \sum_{ j \in J \backslash \{1, 2\} }\bE[\xi_{b,2} ] \bE[\xi_{b,1}g(X_j)\ h_I] \bigg| \notag\\
    & \hspace{8cm}\text{ by $1 \not \in J$ and $2 \not \in I$} \notag\\
    &\leq  \sum_{ i \in I \backslash \{1, 2\} } \big|\bE[\xi_{b,1}]\big| \ \|\xi_{b,2}g(X_i)\|_2 \|h\|_2 + \notag\\
    &\hspace{2cm} \sum_{ j \in J \backslash \{1, 2\} } \big|\bE[\xi_{b,2} ] \big| \  \|\xi_{b,1}g(X_j)\|_2 \|h\|_2 \text{ by } \eqref{Jensen}  \notag\\
    &\leq  2d  \cdot \big|\bE[\xi_{b,1}]\big| \ \|\xi_{b,1}g(X_3)\|_2 \|h\|_2  \text{ by }\eqref{working_assumption_2} \notag\\
    &\leq 2d \bE[\xi_1^2]  \ \|\xi_1\|_2 \ \|g(X_3)\|_2 \|h\|_2 \text{ by \lemref{exp_xi_bi_bdd} and independence} \notag\\
    &= 2d n^{-3/2} \|h\|_2  \text{ by $\sigma_g^2 =1$ in \eqref{unit_var}}. \label{GH_3_plus_GH_4_final_special_bdd}
\end{align}

\subsubsection{Bound on $GG_1 + GG_2$}
\begin{align*}
    &|GG_1+GG_2| \\
    &\leq 2 \Big( \bE[|\xi_{b,1}g^2(X_1)|] \cdot  |\bE[\xi_{b,2}]| + \bE[|\xi_{b,1}g(X_1)|]\cdot \bE[|\xi_{b,2}g(X_2)|] \Big)\\
    & \hspace{6cm}+ \Big|\sum_{i \in I \backslash \{1, 2\}} \sum_{j \in J \backslash \{1, 2\}} \bE[\xi_{b,1}\xi_{b,2} g(X_i) g(X_j))] \Big|\\ 
    &=  2 \Big( \bE[|\xi_{b,1}g^2(X_1)|] \cdot  |\bE[\xi_{b,2}]| + \bE[|\xi_{b,1}g(X_1)|]\cdot \bE[|\xi_{b,2}g(X_2)|] \Big) \\
   & \hspace{6cm} + d\cdot |\bE[\xi_{b, 1}] | \cdot  |\bE[ \xi_{b, 2}]|\cdot \bE[ g^2(X_3)],
\end{align*}
where the last equality uses that 
\[
\bE[\xi_{b,1}\xi_{b,2} g(X_i) g(X_j))] = \bE[\xi_{b,1}\xi_{b,2}]  \bE[ g(X_i)] \bE[ g(X_j))] = 0 \text{ if } i \neq j \text{ and } i, j \not \in \{1, 2\},
\]
  as well as the working assumption in \eqref{working_assumption_2}. Continuing, we get
\begin{align}
&|GG_1+GG_2| \notag\\
 &\leq  2 \Big( \bE[g^2(X_1)] \cdot  |\bE[\xi_{b,2}]| +  n^{-1}\bE[g^2(X_1)]\cdot \bE[g^2(X_2)] \Big) + d\cdot |\bE[\xi_{b, 1}] | \cdot  |\bE[ \xi_{b, 2}]|\cdot \bE[ g^2(X_3)] \notag\\
&\leq  2 (n^{-1} + n^{-1} ) +  d n^{-2}  \text{ by \lemref{exp_xi_bi_bdd} and } \bE[g(X_1^2)] = 1 \text{ in \eqref{unit_var} } \notag\\
&\leq 4n^{-1} + \frac{d}{2m} n^{-1} \text{ by $2m < n$} \notag\\
& \leq 4.5 n^{-1} \text{ by $d \leq m$}\label{GG_1_plus_GG_2_final_bdd}.
\end{align}

\subsubsection{Summary} Recall $1 = \sigma_g \leq \|g\|_3 \leq \|h\|_3$. Combining \eqref{HH_final_bdd}, \eqref{GH_1_plus_G H_2_final_bdd}, \eqref{GH_3_plus_GH_4_final_general_bdd}, \eqref{GG_1_plus_GG_2_final_bdd}  gives \lemref{ustat_results}$(iii)$, and combining \eqref{HH_final_bdd}, \eqref{GH_1_plus_G H_2_final_bdd}, \eqref{GH_3_plus_GH_4_final_special_bdd}, \eqref{GG_1_plus_GG_2_final_bdd} gives \lemref{ustat_results}$(iv)$.

\subsection{Proof of \lemref{counting_identities}} \label{app:pf_of_counting_identities}

Statement $(i)$ is the Vandermonde's identity, which counts the number of ways to choose $m$ balls from $n_1$ red balls and $n_2$ green balls,  by summing over $k \in [0:m]$ the number of ways to choose $k$ red balls and $m-k$ green balls.  Statement $(ii)$ counts the number of ways to choose $m$ balls out of a bag of $n$ balls and paint $k$ of the $m$ chosen balls as red, in two different ways. Statement $(iii)$ comes from 
\begin{align*}
\binom{a}{b} - \binom{a -e}{b} &= \binom{a}{b}  \bigg( 1 - \frac{(a-e) \dots (a -e -b +1)}{a \cdots (a-b+1)} \bigg)\\
&  = \binom{a}{b} \bigg( 1 - \prod_{j = a - b+1}^{a} \Big(1 - \frac{e}{j} \Big)  \bigg) \\
&\leq  \binom{a}{b}  \sum_{j = a-b+1}^a \frac{e}{j}\\
&\leq   \binom{a}{b} \frac{ b e}{a - b +1}.
\end{align*}

\bibliographystyle{ba}

\bibliography{student_Ustat_BE_unif.bib}

\end{document}